\documentclass[10pt]{article}
\usepackage[dvipdf]{graphicx}
\usepackage[latin1]{inputenc}
\usepackage{latexsym}
\usepackage{amsmath,amssymb,amsfonts,amsthm}
\usepackage{xcolor}
\usepackage{mathrsfs}
\usepackage{booktabs}
\usepackage{dsfont}
\usepackage{bigints}

\numberwithin{equation}{section}
\usepackage{paralist}
\usepackage{enumitem}

\usepackage[toc,page]{appendix}
\usepackage[numbers,comma,sort]{natbib}
\usepackage{stmaryrd}
\usepackage{tikz}
\definecolor{db}{RGB}{0, 0, 130}
\usepackage[colorlinks=true,citecolor=db,linkcolor=black,urlcolor=black,pdfstartview=FitH]{hyperref}

\usepackage{pdfsync}
\usepackage{breqn}

\definecolor{rp}{rgb}{0.25, 0, 0.75}
\definecolor{dg}{rgb}{0, 0.5, 0}

\newcommand{\R}{\mathbb{R}}

\newcommand{\N}{\mathbb{N}}
\newcommand{\EE}{\mathbb{E}}

\newcommand{\diff}{\, d}

\newcommand{\B}{\mathcal{B}}

\newcommand{\M}{\mathbf{M}}
\newcommand{\T}{\mathbb{T}}

\def \limbashaut#1#2#3{\mathrel{\mathop{\kern 0pt#1}\limits_{#2}^{#3}}}

\DeclareMathOperator{\Leb}{Leb} 

\makeatletter
\def\namedlabel#1#2{\begingroup
    #2%
    \def\@currentlabel{#2}%
    \phantomsection\label{#1}\endgroup
}
\makeatother

\makeatletter
\newcommand{\myitem}[1]{%
\item[#1]\protected@edef\@currentlabel{#1}%
}
\makeatother

\newtheorem{definition}{Definition}[section]
\newtheorem{theorem}[definition]{Theorem}
\newtheorem{prop}[definition]{Proposition}
\newtheorem{corollary}[definition]{Corollary}

\newtheorem{lemma}[definition]{Lemma}

\newtheorem{proposition}[definition]{Proposition}
\newtheorem{remark}[definition]{Remark}

\author{Ludovic Gouden\`ege\footnote{F\'ed\'eration de Math\'ematiques de CentraleSup\'elec, CNRS FR-3487; Universit\'e Paris-Saclay, France, \texttt{goudenege@math.cnrs.fr}. This work is supported by the SIMALIN project ANR-19-CE40-0016 and the SDAIM project ANR-22-CE40-0015 from the French National Research Agency.}  \and
El Mehdi Haress\footnote{Universit\'e Paris-Saclay, CentraleSup\'elec, MICS and CNRS FR-3487, France.} \textsuperscript{,}\footnote{\texttt{el-mehdi.haress@centralesupelec.fr}. EH acknowledges the support of the Labex Math\'ematique Hadamard.} \and 
Alexandre Richard$^\dagger$\textsuperscript{,}\footnote{\texttt{alexandre.richard@centralesupelec.fr}}}

\title{ \Large{\textbf{Numerical approximation of the stochastic heat equation with a distributional reaction term}}}

\begin{document}

\maketitle

\medskip

\begin{abstract}
We study the numerical approximation of the stochastic heat equation with a distributional reaction term. Under a condition on the Besov regularity of the reaction term, it was proven recently that a strong solution exists and is  unique in the pathwise sense, in a class of H\"older continuous processes. For a suitable choice of sequence $(b^k)_{k\in \N}$ approximating $b$, we prove that the error between the solution $u$ of the SPDE with reaction term $b$ and its tamed Euler finite-difference scheme with mollified drift $b^k$, converges to $0$ in $L^m(\Omega)$ with a rate that depends on the Besov regularity of $b$. In particular, one can consider two interesting cases: first, even when $b$ is only a (finite) measure, a rate of convergence is obtained. On the other hand, when $b$ is a bounded measurable function, the (almost) optimal rate of convergence $(\frac{1}{2}-\varepsilon)$-in space and $(\frac{1}{4}-\varepsilon)$-in time is achieved.

Stochastic sewing techniques are used in the proofs, in particular to deduce new regularising properties of the discrete Ornstein-Uhlenbeck process.
\end{abstract}

\noindent\textit{\textbf{Keywords and phrases:} Numerical approximation of SPDEs, regularisation by noise, stochastic sewing.} 

\medskip

\noindent\textbf{MSC2020 subject classification: } 65C30, 60H50, 60H15, 60H17, 30H25.

\section{Introduction}\label{sec:intro}
Numerical schemes for Stochastic Partial Differential Equations (SPDEs) have been extensively studied in the literature, see for instance the monographs~\cite{JentzenKloeden, kruse2012strong, LordPowellShardlow}.
In this paper, we aim at improving results on the numerical analysis of stochastic reaction-diffusion equations, a field with roots tracing back to Gy\"ongy's pioneering work in \cite{Gyongy:98, Gyongy:99}, where he introduced the first fully discrete numerical scheme.
Compared to Stochastic Differential Equations (SDEs), a numerical approximation of SPDEs involved both temporal and spatial discretization to describe a fully discrete scheme exhibiting convergence rate with respect to time and also with respect to space.
Gy\"ongy showed that a space-time finite differences approximation converges to the true solution under the condition that the reaction term, also known as the drift, remains a bounded measurable function.
Furthermore, he demonstrated that the scheme exhibits strong convergence rates of $1/4$ with respect to time and $1/2$ with respect to space, given that the drift is a Lipschitz continuous function.
Later, Davie and Gaines in \cite{davie2001convergence} proved that these rates are sharp.

Many results in the literature cover the numerical analysis of SPDEs with regular coefficients (see e.g. \cite{LordPowellShardlow,kruse2012strong,jentzen2011taylor}), and further works also explore situations with relaxed regularity assumptions on the coefficients, expanding into higher dimensions, exploring rougher initial conditions, and refining discretization techniques.
We recall that strong convergence usually refers to convergence in the mean-square sense,
see for instance \cite{GyongyMillet:05,GyongyNualart:95, GyongyNualart:97} for implicit scheme, \cite{LordTambue:13,anton2020fully} for exponential scheme, or \cite{Jentzen:11} for higher order scheme.
With respect to the approximation in space, we can cite \cite{KloedenShott:01} for Galerkin approximation or \cite{Wang:17} for finite element methods.
We refer also to \cite{Printems:01} for a statement about convergence in probability with a rate, 
and \cite{BrehierDebussche:17, ConusJentzenKurniawan:14, Debussche:11, JentzenKurniawan:15} for 
weak convergence, which refers to convergence in distribution.

\smallskip

In this paper, we consider the following $(1+1)-$stochastic heat equation with singular reaction term:
\begin{align}\label{eq:spde}
\left\{
    \begin{array}{ll}
    \partial_t u_t(x) = \Delta u_t(x) + b(u_t(x)) + \xi & \text{for } (t,x)\in (0,1] \times \mathbb{T}, \\
     u_0(x) = \psi_0(x) & \text{for } x\in\mathbb{T},
\end{array}
\right.
\end{align}
where $\mathbb{T}=\R \backslash \mathbb{Z}$ is the one-dimensional torus (in other words, we consider the equation with periodic boundary conditions), $b$ is a distribution in the Besov space $\mathcal{B}_{p,\infty}^\gamma(\mathbb{R}, \mathbb{R}) \equiv \mathcal{B}_{p}^\gamma$, $\gamma  \in \mathbb{R}$, $p \in [1,\infty]$, $\xi$ is a space-time white noise and $\psi_0 : \T \rightarrow \mathbb{R}$ is a bounded measurable function.
In \cite{athreya2020well}, Athreya, Butkovsky, L\^e and Mytnik showed that when $\gamma-1/p \geq -1$, Equation~\eqref{eq:spde} admits an adapted solution which is pathwise unique in a class of processes with a certain regularity. In particular for $\gamma < 0$, this equation is not well-posed in the standard sense, since in this case $b$ might only be a distribution. Hence the solution is constructed by mollification of the drift and by establishing Davie's type estimates \cite{Davie}, to then perform a tightness-stability argument on a sequence of approximate solutions. 
Sometimes combined with the recently introduced Stochastic Sewing Lemma (SSL) of \citet{le2020stochastic}, this approach has found a wide range of applications. For instance, it has been used to study well-posedness of rough SDEs \cite{FHL}, SDEs driven by fractional Brownian motion \cite{le2020stochastic,anzeletti2021regularisation,GaleatiGerencser,ButkovskyLeMytnik}, McKean-Vlasov SDEs \cite{GaleatiHarangMayorcas}, as well as to obtain path-by-path uniqueness of Stochastic Differential Equations (SDEs) with singular drift \cite{AnzelettiLeLing}. 
On the numerical side, sewing techniques have also permitted to obtain sharp strong error rates for the Euler scheme of singular SDEs \cite{le2021taming} and fractional SDEs \cite{butkovsky2021approximation,GHR2023}.
For the SPDE \eqref{eq:spde}, Butkovsky, Dareiotis and Gerencs\'er~\cite{butkovsky2021optimal} have recently extended Gy\"ongy's strong error rate to drifts which are merely bounded measurable functions.

\smallskip

Our goal in this paper is to further extend the numerical approximation of \eqref{eq:spde} to the whole regime of well-posedness established in \cite{athreya2020well}.
As in \cite{butkovsky2021optimal}, we want to define an explicit Euler  scheme with central finite-difference approximation in space. 
However, to deal with its singularity, the drift needs to be tamed in the  scheme and particular care needs to be paid in order to obtain convergence to $u$ with a suitable rate.

For $n \in \mathbb{N}\setminus\{0\}$, let $\delta x= (2n)^{-1}$ be a space step of a uniform discretization of $\mathbb{T}$ and let $h$ be a time step of a uniform discretization of $[0,1]$ of the form $h= c (2n)^{-2}$, where $c$ is a constant. Thus the following time and space grids are introduced:
\begin{align}\label{eq:time-space-grids}
\mathbb{T}_n=\left\{0,  (2n)^{-1}, \ldots,  (2n-1) (2n)^{-1}  \right\}, \quad \Lambda_{h}=\left\{0, h, 2 h, \ldots, \lfloor \frac{1}{h} \rfloor h \right \},
\end{align}
and are assumed throughout to satisfy the usual Courant-Friedrichs-Lewy (CFL) condition $c < \frac{1}{2}$.
This natural condition ensures stability of the scheme and its convergence towards the continuous equation as time step $h$ and space step $\delta x$ tend to $0$ together, i.e. when $n$ goes to infinity.
We now define the numerical scheme for $x \in \mathbb{T}_n$ and $t \in \Lambda_h$ as 
\begin{align}\label{eq:scheme-bounded-drift}
\left\{
\begin{array}{ll}
u_{t+h}^{n,k}(x) =u_{t}^{n,k}(x)+ h \Delta_{n} u_{t}^{n,k}(x)+h b^k (u_{t}^{n,k}(x))+n \xi_{n}(t, x) \\
u_0^{n,k} (x) = \psi_0 (x),
\end{array}
\right.
\end{align}
where $\Delta_n$ is the discrete Laplacian, $\xi_n$ is a time and space discretization of the noise $\xi$, which we will both define later, and $(b^k)_{k \in \mathbb{N}}$ is a sequence of smooth functions that approximates $b$.
Due to the explicit relation between $h$ and $n$ ($h= c (2n)^{-2}$), we denote the scheme by $u^{n,k}$ instead of $u^{h,n,k}$.
We prove that the moments of $u_t(x) - u^{n,k}_t(x)$ are controlled by
\begin{align}\label{eq:main-result-SDE-0}
\| b-b^k \|_{\mathcal{B}_p^{\gamma-1}} + (1+\|b^k\|_\infty) n^{-\frac{1}{2}+\varepsilon} +  (1+\|b^k\|_\infty)  \|b^k\|_{\mathcal{C}^1} n^{-1+\varepsilon},
\end{align}
 see Theorem \ref{thm:main-SPDE} for a detailed statement.
 We see that the error is decomposed into a stability error $\| b-b^k \|_{\mathcal{B}_p^{\gamma-1}}$ due to the taming, and the rest which corresponds to the optimal error of the scheme for a bounded drift (see \cite{butkovsky2021optimal}). The stability estimate in $\mathcal{B}_p^{\gamma-1}$ corresponds to the gap between the solution with drift $b$ and the one with drift $b^k$. It is similar to what was found for SDEs in \cite{GaleatiHarangMayorcas,GHR2023} and extends $L^\infty$ stability estimates of standard ODEs with Lipschitz coefficients.

Choosing an approximating sequence $(b^{k_n})_{n \in \mathbb{N}}$ which permits an explicit computation of the norms of $b^{k_n}$ and $b-b^{k_n}$, we obtain the following rate of convergence:
\begin{align*}
\sup_{\substack{t \in \Lambda_h \\ x \in \mathbb{T}_n}} \| u_t(x) - u_t^{n,k_{n}}(x) \|_{L^m(\Omega)} \leq 
C \, n^{-\alpha}
\end{align*}
with $\alpha= \frac{1}{2(1-\gamma+1/p)}$
when $\gamma-1/p>-1$ (sub-critical case). In particular, given a generic $b\in \mathcal{B}_{\infty}^\gamma$ and the error \eqref{eq:main-result-SDE-0}, we find that this rate cannot be improved (see Section~\ref{sec:optimality} for details). In the limit case $\gamma-1/p=-1$, which allows e.g. to choose $b$ as a finite measure, we obtain a non-explicit rate of convergence $\alpha>0$.

\medskip

We now explain briefly the ideas that are used to deal with integrals of singular drifts and their approximation. When comparing $u$ and $u^{n,k}$, we consider mild forms of the solution and of the scheme which introduce error terms of the form
\[
\int_0^t \int_{\T} p_{t-r}(x,y) b(u_r(y)) - p_{(t-r)_h}^n(x,y) b^k(u_{r_h}^{n,k}(y_n)) \, dy dr
\]
with $p_{t}$ the heat semigroup, $p_{t}^n$ the discrete heat semigroup, 
$r_h$ the projection of $r$ on $\Lambda_h$, and $y_n$ the projection of $y$ on $\mathbb{T}_n$.

First, since $b$ is not a smooth function, we rely on regularisation properties induced by the Ornstein-Uhlenbeck process.
Formally, freezing $u_r(y)-O_{r}(y)$ as the variable $z$, this means that we control moments of 
\begin{align*}
\int_s^t \int_{\T} p_{t-r}(x,y) b^k(z+O_{r}(y)) \, dy dr
\end{align*}
by $(t-s)^{\frac{1}{2}+\varepsilon}$ and a weak norm of the drift, namely Besov norm $\|b^k\|_{\mathcal{B}^\gamma_p}$.
This permits to take the limit $k\rightarrow +\infty$ to give a meaning to the integral even when $b$ is very singular.
The precise regularisation results are either taken from \cite{athreya2020well} or newly established lemmas presented in Section~\ref{sec:reg-O}.
The regularisation effect comes from the strong oscillations of the noise, which at the technical level are exploited by applying the Stochastic Sewing Lemma \cite{le2020stochastic}.

Secondly, in the numerical approximation procedure, we are led to establish new regularisation properties for the discrete and continuous Ornstein-Uhlenbeck processes. Namely, denoting by $O_{r_h}^n(y_n)$ the discrete Ornstein-Uhlenbeck process for $r_h \in \Lambda_h$ and $y_n\in \mathbb{T}_n$, we control moments of quantities of the form
\begin{align*}
\int_s^t \int_{\T} p_{t-r}(x,y) b^k(z+O_{r}(y)) - p_{(t-r)_h}^n(x,y)  b^k(z' + O_{r_h}^n(y_n)) \, dy dr
\end{align*}
by $z-z'$, $(t-s)^{\frac{1}{2}+\varepsilon}$, the Besov norm $\|b^k\|_{\mathcal{B}^\gamma_p}$, a negative power of $n$,
$\| b^k\|_{\mathcal{C}^1} n^{-1+\varepsilon}$
and
$\| b^k \|_{\infty} n^{-\frac{1}{2}+\varepsilon}$.
The last two terms are the price to pay when using discrete Ornstein-Uhlenbeck process
whose oscillations, at the small scales, are not as strong as those of the continuous Ornstein-Uhlenbeck process.

\smallskip

For future research, it could be interesting to extend this numerical study to various other numerical schemes, including implicit schemes to treat large time behavior, as well as studying weak error approximation.

\paragraph{Organisation of the paper.}
In Section~\ref{sec:notations}, we begin with definitions and notations, then recall the definition of a solution to \eqref{eq:spde} with distributional drift, and define the numerical scheme in Section~\ref{subsec:defsol&schema}. The main results are stated in Section \ref{sec:main-results}. In Section~\ref{sec:convergence}, we prove the convergence of the tamed Euler finite-difference scheme and obtain a rate of convergence, in both sub-critical and limit cases. The proof relies strongly on the regularisation lemmas for quantities involving the Ornstein-Uhlenbeck process, the discrete Ornstein-Uhlenbeck process and both, which are stated and proven respectively in Section~\ref{sec:reg-O}, Section~\ref{sec:reg-On-1} and Section~\ref{sec:reg-On}. In the Appendix \ref{app:general-lemmas}, we gather some technical lemmas on the continuous and discrete heat semigroups, on Besov estimates and we recall the Stochastic Sewing Lemma with a couple of corollaries. Eventually, in Appendix~\ref{app:B}, we state and prove a critical Gr\"onwall-type lemma for $2$-parameter functions that is used to quantify the error in the limit case.

\section{Framework and results}\label{sec:numerical-scheme}

\subsection{Notations and definitions}\label{sec:notations}

We define some notations that will be useful throughout the paper.

\begin{compactitem}

\item On $(\Omega, \mathcal{F},\mathbb{P})$ a probability space, denote by $\mathbb{F}=\left(\mathcal{F}_{t}\right)_{t \in[0,1]}$ a filtration satisfying the usual conditions.

\item Denote by $\mathbb{E}^{t}$ the conditional expectation given $\mathcal{F}_{t}$.

\item  An $\mathbb{R}$-valued random process $(u_t(x))_{\substack{t \in [0,1], x \in \mathbb{T}}}$ is said to be $\mathbb{F}$-adapted if for all $t$ in $[0,1]$ and $x$ in $\mathbb{T}$, $u_t(x)$ is $\mathcal{F}_{t}$-measurable.

\item The $L^m(\Omega)$ norm is denoted $\|\cdot \|_{L^m}$, and the space $L^m(\Omega)$ is simply denoted by $L^m$.

\item We encounter three different heat kernels in the article: the Gaussian kernel on $\mathbb{R}$, $
g_{t}(x)=\frac{1}{\sqrt{2 \pi} t} \exp \left(-\frac{x^{2}}{2 t}\right)$, the heat kernel on $\mathbb{T}$ associated with periodic boundary conditions $p_t(x,y)= \sum_{k \in \mathbb{Z}}  \frac{1}{\sqrt{2 \pi }t} \exp\Big(-\frac{(x-y + k)^2}{4t}\Big) = \sum_{k \in \mathbb{Z}} e^{-4\pi^2 k^2 t} e^{i 2 \pi k (x-y)}$, and the discrete heat kernel on $[0,1]$ which will be recalled later. We denote by $G_t$ and $P_t$ the respective convolutions with $g_t$ and $p_t$. That is, for any bounded measurable function $f$, we write $G_t f(x) = \int_{\R} g_t(x-y) f(y) d y$, $P_t f(x) = \int_{\T} p_t(x,y) f(y) d y $.

\item For a Borel-measurable function $f: \R \to \R$, we denote the $L^\infty$ and $\mathcal{C}^1$ norms of $f$ by $\|f \|_\infty = \sup_{x \in \R} |f(x)|$ and $\|f\|_{\mathcal{C}^1} = \| f \|_\infty + \sup_{x \neq y} \frac{|f(x)-f(y)|}{|x-y|}$. 

\item For any interval  $I$ of $[0,1]$, we denote the simplex $\Delta_{I}$ by
\begin{align*}
\Delta_{I} = \{ (s,t) \in I, s < t \}.
\end{align*}

\item  For a partition $\Pi$ of an interval, the mesh size is denoted by $|\Pi|$.

\item For a $2$-parameter function $A: \Delta_{[S, T]} \rightarrow \R$, define the $3$-increment as follows: for any $(s, u, t)$ such that $S \leq s \leq u \leq t \leq T$,
\begin{equation*}
\delta A_{s, u, t}:=A_{s, t}-A_{s, u}-A_{u, t} .
\end{equation*}

\item We use the classical definition of nonhomogeneous Besov spaces \cite{bahouri2011fourier}. We write $\mathcal{B}_p^\gamma$ for the space $\mathcal{B}_{p, \infty}^\gamma(\R)$. We recall that for $1\leq p_1 \leq p_2 \leq \infty$, the space $\mathcal{B}_{p_1}^\gamma$ is continuously embedded into $\mathcal{B}^{\gamma-1/p_1-1/p_2}_{p_2}$, which is written as ${\mathcal{B}_{p_1}^\gamma \hookrightarrow \mathcal{B}^{\gamma-1/p_1-1/p_2}_{p_2}}$, see e.g.  
\cite[Prop.~2.71]{bahouri2011fourier}.

\item C will denote a constant that may change from line to line. It does not depend on any parameter other than those specified in the associated result in which it is stated.

 \end{compactitem}

\subsection{Definition of solutions and the numerical scheme}\label{subsec:defsol&schema}
We explain now what we mean by a solution to \eqref{eq:spde} and define the numerical scheme.
\begin{definition} The white noise $\xi$ on $\T \times [0, 1]$ is defined as a mapping from the Borel sets $\mathscr{B}(\T\times [0, 1])$ to $L^{2}(\Omega)$ such that for any elements $A_{1}, \ldots, A_{k}$ of $\mathscr{B}(\T\times [0, 1])$, the vector $\left(\xi\left(A_{1}\right), \ldots, \xi\left(A_{k}\right)\right)$ has a Gaussian distribution with mean 0 and covariance $\mathbb{E}\left[\xi(A_{i})\, \xi(A_{j})\right]= \Leb(A_{i} \cap A_{j})$. We denote by $\mathbb{F}^\xi= (\mathcal{F}_t^\xi)_{t \in [0,1]}$ the filtration generated by $\xi$, i.e. $\mathcal{F}_t^{\xi}$ is the $\sigma$-algebra generated by $\{\xi^{-1}(A),~A \in \mathscr{B}(\T\times [0, t))\}$. 
\end{definition}
Denote by $O$ the space-time Ornstein-Uhlenbeck process and by $Q$ its variance:
\begin{align}\label{def:O-Q}
\begin{split}
O_t(x) & = \int_0^t \int_{\T} p_{t-r}(x,y) \, \xi(dy, dr) \\
Q(t) & = \EE O_t(x)^2 = \int_0^t \int_\T | p_{t-r}(x,y)|^2 dy dr = \int_0^t \int_\T |p_r(y,0)|^2 dy dr .
\end{split}
\end{align}
The \emph{mild} form associated to \eqref{eq:spde} is
\begin{align}\label{eq:spde-mild}
\forall t \in [0,1],~\forall x \in \T,~u_t(x) = P_t \psi_0 (x) + \int_0^t \int_{\T} p_{t-r}(x,y) b(u_r(y)) \, dy dr + O_t(x).
\end{align}
The mild form is not clearly defined when $b$ is a genuine distribution in some Besov space. To define a solution of \eqref{eq:spde-mild}, we approximate $b$ via a smooth sequence. 
\begin{definition}\label{def:conv-gamma-}
A sequence $(b^k)_{k \in \mathbb{N}}$ in the space $\mathcal{C}_{b}^\infty(\R)$ of smooth and bounded functions is said to converge to $b$ in $\mathcal{B}_p^{\gamma-}$ as $k$ goes to infinity if
\begin{align}\label{eq:conv-in-gamma-}
\left\{ 
\begin{array}{ll}
\sup_{k \in \mathbb{N}} \|b^k\|_{\B_p^\gamma} < \|b\|_{\B_p^\gamma} < \infty \\
\lim_{k \rightarrow \infty} \|b- b^k\|_{\B_p^{\gamma'}} = 0 & \forall \gamma' < \gamma.
\end{array}
\right.
\end{align}
\end{definition}
We now recall the definition a solution of \eqref{eq:spde} from \cite[Definition 2.3]{athreya2020well}.
\begin{definition}\label{def:true-solution}
A measurable process $(u_{t}(x) )_{t \in [0,1],\, x \in \mathbb{T}}$ is a strong solution to \eqref{eq:spde} if there exists a bounded measurable function $\psi_0: \T \to \R$ and a measurable process $v : [0,1] \times \T \times \Omega \rightarrow \mathbb{R}$ such that:
\begin{itemize}
\item[(1)] $u$ is adapted to $\mathbb{F}^\xi$;
\item[(2)] $u_t(x) = P_t \psi_0(x) + v_t(x) + O_t(x) \ \text{a.s},~\forall x \in \mathbb{T},~\forall t \in [0,1]$;
\item[(3)] For any sequence $(b^k)_{k \in \mathbb{N}}$ in $\mathcal{C}^\infty_{b}$ converging to $b$ in $\B_p^{\gamma-}$, we have 
\begin{align*}
\sup_{\substack{t \in [0,1] \\ x \in \mathbb{T}}} \left| \int_0^t \int_{\T} p_{t-r}(x,y) b^k(u_r(y)) \diff y \diff r - v_t(x) \right| \rightarrow 0 \ \text{in probability as } k \rightarrow \infty \ .
\end{align*}
\item[(4)] Almost surely, the process $u$ is continuous on $[0,1] \times \mathbb{T}$.
\end{itemize}
\end{definition}
When $b$ is a regular function, say $b \in \mathcal{C}^\gamma$ for some $\gamma>0$, we can choose $b^k$ to converge uniformly to $b$. Then the solution $u$ is equivalent to the usual notion of a mild solution. The numerical approximation we consider is an Euler in time and a finite-difference in space scheme.

\begin{definition}\label{def:scheme} Recall that $h=c(2n)^{-2}$ and the definitions of $\T_n$ and $\Lambda_h$ in \eqref{eq:time-space-grids}. For $k \in \mathbb{N}$, the numerical scheme is defined for $x \in \mathbb{T}_n$ and $t \in \Lambda_h$ as
\begin{align}\label{eq:numerical-scheme}
\left\{
\begin{array}{ll}
u_0^{n,k}(x) = \psi_0(x) \\
u_{t+h}^{n,k}(x)=u_{t}^{n,k}(x)+ h \Delta_{n} u_{t}^{n,k}(x)+ h b^k(u_{t}^{n,k}(x))+h \xi_{n}(x,t) ,
\end{array}
\right.
\end{align}
where $\Delta_n$ is the discrete Laplacian defined as
\begin{align*}%
\Delta_{n} f(x)=(2n)^2 \left( f(x+(2n)^{-1} )-2 f(x)+f(x-(2n)^{-1} ) \right),
\end{align*}
and the discrete noise is given by
\begin{align*}%
\xi_{n}(x,t)= (2n) h^{-1} \xi\left(\left[x, x+(2n)^{-1} \right]\times [t, t+h] \right) .
\end{align*}
\end{definition}

As for the continuous equation, we will rely on the mild formulation of the scheme. To do so, let $p^n_{t}$ be the discrete analogue of $p_{t}$ defined in \cite[Equation (2.11)]{butkovsky2021optimal} and let $P^n_{t}$ denote the discrete convolution $P^n_{t} f := \int_{\T} p^n_{t}(x,y) f(y_n) \, d y $. 
\begin{lemma}\label{def:discrete-kernel}
For $t \in [0,1]$ and $x \in \mathbb{T}$, denote by $t_h$ and $y_n$ the leftmost gridpoint from $t$ in $\Lambda_h$ and from $y$ in $\mathbb{T}_n$ respectively. For $k,n\in\N$, the process $u^{n,k}$ defined in \eqref{eq:numerical-scheme} on the domain $\T_{n}\times \Lambda_{h}$ extends to $[0,1]\times \T$ by the following mild formula:
\begin{align}\label{eq:discrete-mild}
u_{t}^{n,k}(x)=P_{t}^{n} \psi_0 (x)+\int_{0}^{t} \int_{\T} p_{(t-r)_h}^{n}(x,y) \, b^k\left(u_{r_h}^{n,k}(y_n)\right) dy dr+\int_{0}^{t} \int_{\T} p_{(t-r)_h}^{n}(x, y) \, \xi(dy, dr).
\end{align}
\end{lemma}
We refer to \cite[Section 2]{butkovsky2021optimal} for a proof of the previous lemma, as we follow the same discretisation, except that we further mollify the drift.
The following CFL condition is imposed in the whole paper:
\begin{align*}
c = h (2n)^2 \equiv \frac{\delta t}{(\delta x)^2} < \frac{1}{2}.
\end{align*}
This determines the behaviour of the eigenvalues of $(\text{Id} - \delta t \Delta_{n})$ (see \cite[Lemma 2.1.5]{butkovsky2021optimal}), and ensures good stability properties of $p^n$ (see Lemma \ref{lem:reg-Pnh}) and its convergence towards $p$ (see Lemma \ref{lem:P-Pn}).

Let also define $O^n$ the discrete Ornstein-Uhlenbeck process and $Q^n$ its variance:
\begin{align}\label{def:discrete-OU}
\begin{split}
O_{t}^{n}(x) &  :=\int_{0}^{t} \int_{\T} p_{(t-r)_h}^{n}(x, y)\, \xi(d y, d r) \\
Q^n(t) & := \int_0^t \int_\T | p^n_{r_h}(x,y) |^2 dy dr.
\end{split}
\end{align}
To see that $Q^n$ does not depend on $x$ we refer to its definition in \cite[page 1121]{butkovsky2021optimal}.

\subsection{Main results}\label{sec:main-results}
In \cite[Theorem 2.6$(ii)$]{athreya2020well}, the authors prove that if $\gamma-1/p \geq -1$ and $\gamma>-1$, there exists a unique strong solution to \eqref{eq:spde} such that for any $m \ge 2$,
\begin{align}\label{eq:holder-reg-sol}
\sup_{\substack{(s,t) \in \Delta_{[0,1]} \\  x \in \T}} \frac{\| \left( \EE^s | u_t(x)-O_t(x) - P_{t-s}(u_s(x)-O_s(x)) |^m  \right)^{\frac{1}{m}}\|_{L^\infty}}{|t-s|^{\frac{3}{4}}} < \infty .
\end{align}
The following proposition completes the pathwise uniqueness to \eqref{eq:spde} proven in \cite[Proposition 3.6]{athreya2020well}, in the sub-critical regime.

\begin{prop}\label{cor:extend-uniqueness}
Let $\gamma \in \R$, $p \in [1, \infty]$ such that $\gamma-1/p > -1$. Let $\eta \in (0,1)$, then pathwise uniqueness for \eqref{eq:spde} holds in the class of solutions $u$ such that
\begin{align}\label{eq:weak-reg-sol}
\sup_{\substack{(s,t) \in \Delta_{[0,1]} \\  x \in \T}} \frac{\| \left( \EE^s | u_t(x)-O_t(x) - P_{t-s}(u_s(x)-O_s(x)) |^2  \right)^{\frac{1}{2}}\|_{L^\infty}}{|t-s|^{\frac{1}{4}(1-\gamma+\frac{1}{p})+\eta}} < \infty .
\end{align}
\end{prop}
The proof is based on showing that \eqref{eq:weak-reg-sol} implies \eqref{eq:holder-reg-sol}, and then concludes by uniqueness under \eqref{eq:holder-reg-sol}. We omit it as it follows the same lines as \cite[Appendix B]{GHR2023}.

\smallskip

Our main result is the convergence of the tamed Euler finite-difference scheme $u^{n,k}$ towards this solution $u$, see Theorem \ref{thm:main-SPDE}. It provides  an extension of \cite[Theorem 1.3.2]{butkovsky2021optimal}, which covered drifts in $L^\infty(\R)$, to drifts with negative regularity, and in particular to distributions.

\begin{theorem}\label{thm:main-SPDE} 
Let $\gamma \in \mathbb{R}$, $p \in [1,\infty]$ be such that
\begin{align}\label{eq:cond-gamma-p-H}
0 > \gamma - \frac{1}{p}\geq -1~ \text{ and } ~ \gamma > -1. \tag{H1}
\end{align}
Let $b \in \mathcal{B}_p^\gamma$, $m \in [2, \infty)$, $\varepsilon \in (0,1/2)$ and let $\psi_0 \in \mathcal{C}^{\frac{1}{2}-\varepsilon} (\T, \mathbb{R})$. Let $u$ be the strong solution to \eqref{eq:spde} that satisfies \eqref{eq:holder-reg-sol}. Let $(b^k)_{k\in \N}$ be a sequence of smooth functions that converges to $b$ in $\mathcal{B}_p^{\gamma-}$ (in the sense of Definition~\ref{def:conv-gamma-}) and $(u^{n,k})_{n \in \N, k \in \mathbb{N}}$ be the tamed Euler finite-difference scheme defined in \eqref{def:scheme}, on the same probability space and with the same space-time white noise $\xi$ as $u$.

\begin{enumerate}[label=(\alph*)]
\item[(a)] \underline{The sub-critical case}: Let $0>\gamma-1/p>-1$. Then there exists a constant $C$ that depends on $m, p, \gamma, \varepsilon,  \|b\|_{\mathcal{B}_p^{\gamma}}, \| \psi_0 \|_{\mathcal{C}^{\frac{1}{2}-\varepsilon}}$ such that for all $n\in \N^*$ and $k \in \mathbb{N}$, the following bound holds: 
\begin{align}\label{eq:main-result-SDE}
\begin{split}
& \sup_{t \in [0,1], x \in \mathbb{T}} \| u_t(x)-u^{n,k}_t(x) \|_{L^m} \\ 
&  \leq C  \left( \| b-b^k \|_{\mathcal{B}_p^{\gamma-1}} + (1+\|b^k\|_\infty) n^{-\frac{1}{2}+\varepsilon} +  (1+\|b^k\|_\infty)  \|b^k\|_{\mathcal{C}^1} n^{-1+\varepsilon} \right).
\end{split}
\end{align}

\item[(b)] \underline{The limit case}: Let $\gamma-1/p=-1$ and $p<+\infty$. Let $\M$ be the constant given by Proposition~\ref{prop:bound-E1-critic} and  $\delta \in (0, e^{-\M})$.  Let $\mathcal{D}$ be a subset of $\mathbb{N}^2$ such that 
\begin{align}\label{eq:assump-bn-bounded}
\sup_{(n,k) \in \mathcal{D}} \| b^{k} \|_\infty n^{-\frac{1}{2}+\varepsilon} < \infty ~\text{and}~ \sup_{(n,k) \in \mathcal{D}} \| b^{k} \|_{\mathcal{C}^1} n^{-1} < \infty. \tag{H2}
\end{align}
Then there exists a constant $C$ that depends on $m, p, \gamma, \delta, \varepsilon, \|b\|_{\mathcal{B}_p^{\gamma}}, \| \psi_0 \|_{\mathcal{C}^{\frac{1}{2}-\varepsilon}}$ such that for all $(n,k) \in \mathcal{D}$, the following bound holds:
\begin{equation*}%
\begin{split}
& \sup_{t \in [0,1], x \in \mathbb{T}} \| u_t(x)-u^{n,k}_t(x) \|_{L^m} \\ 
&  \leq C \Big(  \| b - b^k \|_{\mathcal{B}_p^{\gamma-1}} (1+ |\log \| b - b^k \|_{\mathcal{B}_p^{\gamma-1}} |) + (1+\| b^k \|_{\infty}) n^{-\frac{1}{2}+\varepsilon} \\ 
& \quad + (1+\| b^k \|_{\infty} ) \| b^k \|_{\mathcal{C}^1}   n^{-1+\varepsilon}  \Big)^{e^{-\M}-\delta}.
\end{split}
\end{equation*}

\end{enumerate}

\end{theorem}

\begin{remark}
The constant $\M$ that appears in the limit case is not explicitly known. Since it determines the rate of convergence in the limit case, we name it so as to keep track of it in the paper.
\end{remark}

As we have in mind that $b$ is a distribution, we expect $\|b^k\|_\infty$ and $\|b^k\|_{\mathcal{C}^1}$ to diverge as $k\to \infty$. 
Hence we are careful when choosing $(b^k)_{k \in \mathbb{N}}$ in order to get a rate of convergence.

Choosing $b^k=G_{\frac{1}{k}}b$ for $k \in \mathbb{N}^*$, owing to Lemma \ref{eq:reg-S} it comes that for $\gamma-1/p<0$,
\begin{align}
&\| b - b^k \|_{\mathcal{B}_p^{\gamma-1}}  \leq C\, \|b\|_{\mathcal{B}_p^\gamma} \  k^{-\frac{1}{2}},  \label{eq:bn-b} \\
&\|b^k\|_\infty  \leq C\, \|b\|_{\mathcal{B}_p^\gamma}  \ k^{-\frac{1}{2}(\gamma-\frac{1}{p} )}, \label{eq:bn-inf} \\
&\| b^k \|_{\mathcal{C}^1}  \leq C\, \|b\|_{\mathcal{B}_p^\gamma}  \, k^{\frac{1}{2}} \ k^{-\frac{1}{2}(\gamma-\frac{1}{p})} \label{eq:bn-C1} .
\end{align}
Using these results in \eqref{eq:main-result-SDE} and optimising over $n$ and $k$, we deduce the following corollary. The optimality of this choice is discussed in Section~\ref{sec:optimality}.

\begin{corollary}\label{cor:bn-choice}
Let the same assumptions and notations as in Theorem \ref{thm:main-SPDE} hold. Let $n \in \N^*$ and define
 \begin{equation*}
k_n = \lfloor n^{\frac{1}{1-\gamma+\frac{1}{p}}}\rfloor  ~~\mbox{and}~~ b^{k_n}=G_{\frac{1}{k_n}} b. 
 \end{equation*}
\begin{enumerate}[label=(\alph*)]
\item \underline{The sub-critical case}: assume that $0>\gamma-1/p>-1$. 
Then there exists a constant $C$ that depends on $ m, p, \gamma, \varepsilon,  \|b\|_{\mathcal{B}_p^{\gamma}},\| \psi_0 \|_{\mathcal{C}^{\frac{1}{2}-\varepsilon}}$ such that the following bound holds:
\begin{align}
 \sup_{t \in [0,1], x \in \mathbb{T}} \| u_t(x)-u^{n,k_n}_t(x) \|_{L^m} &  \leq C \, n^{-\frac{1}{2(1-\gamma+\frac{1}{p})}-\varepsilon}, \label{eq:rate1}
\end{align}

\item \underline{The limit case}: assume that $\gamma-1/p=-1$ and $p<+\infty$. Let $\M$ be the constant given by Proposition \ref{prop:bound-E1-critic} and let $\delta \in (0, e^{-\M})$. Then there exists a constant $C$ that depends on $m, p, \gamma, \|b\|_{\mathcal{B}_p^{\gamma}}, \| \psi_0 \|_{\mathcal{C}^{\frac{1}{2}-\varepsilon}}, \delta$ such that the following bound holds: 
\begin{align}
\begin{split}
\sup_{t \in [0,1], x \in \mathbb{T}} \| u_t(x)-u^{n,k_n}_t(x) \|_{L^m}  &  \leq C \, n^{-\frac{1}{4}(e^{-\M}-\delta)} . \label{eq:rate1-critic}
\end{split}
\end{align}

\end{enumerate}
\end{corollary}

~~

Theorem \ref{thm:main-SPDE} and Corollary \ref{cor:bn-choice} can be extended to $\gamma- 1/p \geq 0$. In that case, by embedding, $b \in \mathcal{B}_\infty^{\gamma-1/p}$. Then $\mathcal{B}_\infty^{\gamma-1/p} \subset \mathcal{B}_\infty^{\eta}$,  for any $\eta<0$, thus one can apply Theorem \ref{thm:main-SPDE} and Corollary \ref{cor:bn-choice} in $ \mathcal{B}_\infty^{\eta}$. Since we exhibit no gain if $b$ has positive regularity, we express the following corollary in $\mathcal{B}_\infty^0$, but it obviously holds also if $b$ is assumed with more regularity.

\begin{corollary}\label{cor:gama=d/p}
Let the assumptions of Theorem \ref{thm:main-SPDE} hold and assume further that $\psi_0 \in \mathcal{C}^{\frac{1}{2}}(\T,\R)$. Let $b \in \mathcal{B}_\infty^0$ and $m \ge 2$. Let $u$ be a strong solution to \eqref{eq:spde} that satisfies \eqref{eq:holder-reg-sol}. Let $\varepsilon \in (0,1/2)$.
Then there exists $C$ depending on $m, \varepsilon,  \|b\|_{\mathcal{B}_\infty^{0}}, \| \psi_0 \|_{\mathcal{C}^{\frac{1}{2}}}$ only and such that for any $n\in \N^*$ and $k \in \N$, the following bound holds:
\begin{align*}%
\begin{split}
&\sup_{t \in [0,1], x \in \mathbb{T}} \| u_t(x)-u^{n,k}_t(x) \|_{L^m}  \\&  \leq C \left( \| b-b^k \|_{\mathcal{B}_{\infty}^{-1}} + (1+\|b^k \|_\infty) n^{-\frac{1}{2}+\varepsilon} +  (1+\|b^k\|_\infty) \|b^k\|_{\mathcal{C}^1}\, n^{-1+\varepsilon} \right).
\end{split}
\end{align*}
Moreover, for $k_n=n$ and $b^{k_n}=G_{\frac{1}{k_n}} b$, we have
\begin{align*}
 \sup_{t \in [0,1], x \in \mathbb{T}} \| u_t(x)-u^{n,k_n}_t(x) \|_{L^m} &  \leq C n^{-\frac{1}{2}+\varepsilon} .
\end{align*}
\end{corollary}

The above results are proven in Section \ref{sec:convergence}.

\begin{remark} 
\begin{itemize}
\item  If $b$ is a signed measure, for instance a Dirac distribution, then $b \in \mathcal{B}_1^0$ and one can apply Corollary \ref{cor:bn-choice}$(b)$. This quantifies the approximation of a skew-type stochastic heat equation, see \cite{BounebacheZambotti} for a related skew equation.

\item The $n^{-\frac{1}{2}+\varepsilon}$ rate from Corollary~\ref{cor:gama=d/p} is the optimal strong error rate, which coincides with the discretisation of the Ornstein-Uhlenbeck process, see Equation~\eqref{eq:additional-error-terms}.
\item 
In \cite{butkovsky2021optimal}, the authors also found a rate $n^{-\frac{1}{2}+\varepsilon}$ for drifts in the space of bounded measurable functions, which is strictly included in $ \mathcal{B}_\infty^0$ (see e.g. \cite[Section 2.2.2, Eq. (8) and Section 2.2.4, Eq. (4)]{runst2011sobolev}).

\end{itemize}
\end{remark}

\subsection{Discussion on the optimality of the rate of convergence}\label{sec:optimality}

In the sub-critical regime, the bound \eqref{eq:main-result-SDE} consists of  two terms of different nature. The first term $\|b-b^k\|_{\mathcal{B}_p^{\gamma-1}}$ represents a stability error which appears due to the taming of the Euler scheme and the $\mathcal{B}_p^{\gamma-1}$ norm is the natural scale to control such error (see also  \cite[Theorem 3.2]{GaleatiGerencser}).
As for the second term $\|b^k \|_{\infty} n^{-1/2+\varepsilon} + \| b^k\|_{\infty} \| b^k \|_{\mathcal{C}^1} n^{-1+\varepsilon}$, it is of the same order as $\|b^k \|_{\infty} n^{-1/2+\varepsilon}$ for some simple examples (see below). The latter expression represents the optimal error of the Euler scheme for bounded drifts so we also expect such a term to arise in our context. Hence we believe that the bound \eqref{eq:main-result-SDE} cannot be improved.

\smallskip

In the following paragraph, we show that given the bound \eqref{eq:main-result-SDE}, given any sequence $(b_{k})_{k\in \N}$ that approximates $b$ and any sequence $(k_{n})_{n\in \N}$ such that $k_{n}\to +\infty$, one cannot improve the rate from \eqref{eq:rate1}. Stated otherwise, the choice of $b^k=G_{\frac{1}{k}}b$ and $k_n = n^{-\frac{\gamma}{2(1-\gamma)}}$ yields the optimal rate.

\smallskip

The definition of the Littlewood-Paley blocks $\Delta_{j}$ and of the Besov norm are recalled in Section~\ref{subsec:LPblocks} (see also \cite[Sections 2.2 and 2.7]{bahouri2011fourier}).
Let $\gamma\in (-1,0)$ and without loss of generality, let $p=+\infty$. From Lemma~\ref{lem:example} (up to changing slightly $\gamma$), there exists $b\in \mathcal{B}^\gamma_{\infty}$ such that for some $J \in \N$ and $C>0$, 
\begin{align}\label{eq:delta_j-0}
\forall j\geq J,\quad \| \Delta_j b \|_{\infty} \ge C\, 2^{-j\gamma} .
\end{align}
In particular, $b\notin \bigcup_{\gamma'>\gamma} \mathcal{B}^{\gamma'}_{\infty}$. 
Let $(b^k)_{k \in \N}$ be a sequence of smooth functions converging to $b$ in $\mathcal{B}_{\infty}^{\gamma-}$.
If we had $\sup_{k} \| b^k \|_{\infty}<\infty$, then by a convolution inequality it would hold that  $\|\Delta_j b^k \|_{\infty} \leq  \| b^k \|_{\infty} $ and by the fact that $\| \Delta_j b^k \|_{\infty} $ converges to $\|\Delta_j b \|_{\infty} $ as $k \to \infty$, we would get that $\| b \|_{\mathcal{B}_\infty^0} = \sup_{j\geq -1} \|\Delta_j b \|_{\infty} < \infty$. This would be a contradiction with $b\notin \mathcal{B}_{\infty}^{0}$.
Hence up to relabelling the sequence, assume that $\| b^k \|_{\infty}\in [k-1,k]$. This implies in particular that
\begin{align*}
\|\Delta_j b^k \|_{\infty} \leq  \| b^k \|_{\infty} \leq k.
\end{align*} 
Using \eqref{eq:delta_j-0}, we have
\begin{align*}
\| b- b^k \|_{\mathcal{B}_\infty^{\gamma-1}} & \ge \sup_{j \ge 0} 2^{j (\gamma-1)}  | \| \Delta_j b\|_{\infty} - \| \Delta_j b^k \|_{\infty} | \\
&  \ge \sup_{j \ge J, C 2^{- j\gamma}  \ge k} 2^{j (\gamma-1)}  ( C2^{- j\gamma} -k)  .
\end{align*}
The quantity $2^{j (\gamma-1)}  ( C2^{- j\gamma} -k)$ is decreasing in $j$ when $j \ge -\frac{1}{\gamma} \frac{\log(k(1-\gamma)/C)}{\log 2}$. Hence letting $j^\star = \lfloor -\frac{1}{\gamma} \frac{\log(k(1-\gamma)/C)}{\log 2}\rfloor$, and taking $k$ large enough so that $j^\star \ge J$, we  have
\begin{align*}
\| b- b^k \|_{\mathcal{B}_{\infty}^{\gamma-1}} 
 \ge C 2^{-j^\star} -2^{j^\star (\gamma-1)} k  = (1-\gamma) 2^{j^\star(\gamma-1)} k - 2^{j^\star(\gamma-1)}k 
& = -\gamma 2^{j^\star(\gamma-1)} k \\
& \geq -\gamma \left( k(1-\gamma)/C \right)^{-\frac{\gamma-1}{\gamma}} k \\
& \gtrsim k^{\frac{1-\gamma}{\gamma}+1} = k^{\frac{1}{\gamma}} .
\end{align*}
Finally, the bound \eqref{eq:main-result-SDE} now reads
\begin{align*}
\| b-b^k \|_{\mathcal{B}_{\infty}^{\gamma-1}} + \|b^k\|_\infty n^{-\frac{1}{2}+\varepsilon} +  \|b^k\|_\infty  \|b^k\|_{\mathcal{C}^1} n^{-1+\varepsilon}  \ge C  k^{\frac{1}{\gamma}} + k n^{-\frac{1}{2}+\varepsilon} .
\end{align*}
The right-hand side is minimised when 
$k_n = n^{-\frac{\gamma}{2(1-\gamma)}}$. Plugging this into the previous lower bound, we obtain the rate of convergence $\frac{1}{2(1-\gamma)}$ up to an $\varepsilon$ power, which is the same rate as in \eqref{eq:rate1}.

\section{Convergence of the tamed Euler finite-difference scheme}\label{sec:convergence}

The goal of this section is to prove the main results of the paper, stated in Section~\ref{sec:main-results}. First, we define notations that will be useful throughout the proof and in Section~\ref{sec:reg-O}, Section~\ref{sec:reg-On-1} and Section~\ref{sec:reg-On}.  For $m \in [1,\infty]$, $\alpha > 0$, $I$ an interval of $[0,1]$ and $W: \Delta_{[0,1]} \times \mathbb{T} \times \Omega \to \R$, define
\begin{align}\label{def:holder-norm-W}
[W]_{\mathcal{C}^{\alpha,0}_{I,x} L^{m}} := \sup_{\substack{ (s,t) \in \Delta_{[0,1]} \\ x \in \mathbb{T} }}\frac{\left\| W_{s,t}(x) \right\|_{L^m}}{|t-s|^{\alpha}} .
\end{align}
Moreover, 
for $Z: [0,1] \times \mathbb{T} \times \Omega \to \R$ we write 
\begin{align*}%
\|Z \|_{L^{\infty,\infty}_{I,x}  L^m} := \sup_{\substack{ t \in I \,  \\ x \in \mathbb{T}  }} \left\|Z_{t}(x) \right\|_{L^{m}}   .
\end{align*}
In order to bound $\sup_{t \in [0,1], x \in \mathbb{T}} \| u_t(x)-u^{n,k}_t(x) \|_{L^m}$, we introduce an intermediate term denoted by $\mathcal{E}^{n,k}$, defined for $(s,t) \in \Delta_{[0,1]}$ and $x \in \mathbb{T}$ by
\begin{align}\label{def:error}
\mathcal{E}^{n,k}_{s,t}(x) = u_t(x)-O_t(x) - P_{t-s} (u_s-O_s)(x) - \int_s^t P^n_{(t-r)_h} b^k(u^{n,k}_{r_h})(x)\, dr ,
\end{align}
where it is recalled that $r_h$ denotes the leftmost gridpoint from $r$ in $\Lambda_h$, defined in \eqref{eq:time-space-grids}. 
Then recalling the mild form \eqref{eq:spde-mild} of $u$ and the mild form \eqref{eq:discrete-mild} of the numerical scheme $u^{n,k}$, 
\begin{align*}
\sup_{t \in [0,1], x \in \mathbb{T}} \| u_t(x)-u^{n,k}_t(x) \|_{L^m} & \leq \| \mathcal{E}^{n,k}_{0,\cdot} \|_{L^{\infty,\infty}_{[0,1],x} L^m } + \sup_{t \in [0,1], x \in \mathbb{T}} \| O_t(x) - O^n_t(x) \|_{L^m}  \\ & \quad + \sup_{t \in [0,1], x \in \mathbb{T}} | P_t \psi_0 (x) - P^n_t \psi_0(x) | .
\end{align*}
Applying \cite[Corollary 2.3.2]{butkovsky2021optimal} to the term $ \| O_t(x) - O^n_t(x) \|_{L^m}$, and Lemma~\ref{lem:P-Pn}$(ii)$ to the term $| P_t \psi_0 (x) - P^n_t \psi_0(x) |$, one gets
\begin{equation}
\label{eq:additional-error-terms}
\begin{split}
 \sup_{t \in [0,1], x \in \mathbb{T}} \| O_t(x) - O^n_t(x) \|_{L^m} & \leq C n^{-\frac{1}{2}+\varepsilon} \\
 \sup_{t \in [0,1], x \in \mathbb{T}} | P_t \psi_0 (x) - P^n_t \psi_0(x) | & \leq C \| \psi_0 \|_{\mathcal{C}^{\frac{1}{2}-\varepsilon}} n^{-\frac{1}{2}+\varepsilon} ,
\end{split}
\end{equation}
from which it follows that
\begin{align}\label{eq:lien-u-E}
\sup_{t \in [0,1], x \in \mathbb{T}} \| u_t(x)-u^{n,k}_t(x) \|_{L^m} \leq  \| \mathcal{E}^{n,k}_{0,\cdot} \|_{L^{\infty,\infty}_{[0,1],x} L^m }  + C (\| \psi_0 \|_{\mathcal{C}^{\frac{1}{2}-\varepsilon}}+1)\, n^{-\frac{1}{2}+\varepsilon} .
\end{align}
Therefore, it suffices to bound $[\mathcal{E}^{n,k}]_{L^{\infty,\infty}_{[0,1],x} L^m }$ to deduce a bound on $\sup_{t \in [0,1], x \in \mathbb{T}} \| u_t(x)-u^{n,k}_t(x) \|_{L^m}$. In the sub-critical case, when $\gamma-1/p>-1$, we will bound $\| \mathcal{E}^{n,k}_{0,\cdot} \|_{L^{\infty,\infty}_{[0,1],x} L^m}$ by $[\mathcal{E}^{n,k}]_{\mathcal{C}^{1/2,0}_{[0,1],x} L^m }$ to get the final bound on $\sup_{t \in [0,1], x \in \mathbb{T}} \| u_t(x)-u^{n,k}_t(x) \|_{L^m}$. In Proposition \ref{prop:main-SPDE}, we present bounds on $\mathcal{E}^{n,k}$ which combined with the previous analysis allow to deduce the results of Theorem \ref{thm:main-SPDE}, Corollary \ref{cor:bn-choice} and Corollary \ref{cor:gama=d/p}.

\begin{remark}
Bounding the supremum norm of $\mathcal{E}^{n,k}_{0,\cdot}$ by its H\"older norm may seem sub-optimal, but we explain here why the two are linked. Roughly, we use the Stochastic Sewing Lemma to bound $\sup_{x} \|\mathcal{E}^{1,n,k}_{s,t}(x)\|_{L^m}$ (a quantity closely related to $\mathcal{E}^{n,k}$) by a quantity of the form $C ( [\mathcal{E}^{n,k}]_{\mathcal{C}^{1/2,0}_{[s,t],x} L^{m}} + \epsilon(n,k)) (t-s)^{1/2+\varepsilon}$, with some error term $\epsilon(n,k)$ that will be specified later (see Proposition~\ref{prop:ssl-o-on-2} for a precise statement). The power $1/2+\varepsilon$ on $(t-s)$ comes directly from the sewing and to obtain it, the time regularity of $\mathcal{E}^{n,k}$ is needed, which explains the presence of its H\"older norm. Besides, it is not possible to require less than this $1/2+\varepsilon$ time regularity in the sewing step, which also explains why we consider precisely the $1/2$-H\"older norm. Then choosing $(t-s)$ small enough permits to close the aforementioned estimate and deduce a bound in H\"older norm.

From thereon, one could try to go further and obtain error bounds in H\"older norm on $u-u^{n,k}$. This would require to study additional terms and would lower the rate of convergence that is obtained in supremum norm. Hence we choose not to pursue in this direction.
\end{remark}

\subsection{Organisation of the proofs of the main results}\label{sec:proof-main}

In this section, we state intermediate results on $\mathcal{E}^{n,k}$ and also study the regularity of the numerical scheme $u^{n,k}$ which provides a useful \emph{a priori} estimate. The latter is given in terms of the following pseudo-norm defined by  
\begin{align}\label{eq:discrete-seminorm}
\big\{ u^{n,k} \big\}_{n,k,m, \tau} := \sup_{\substack{(s,t) \in \Delta_{[0,1]} \\ x \in [0,1]}}  \frac{\Big\| \EE^s \bigg[ \Big| {\displaystyle \int_{s}^t }  P^n_{(t-r)_h} b^k( u_{r_h}^{n,k}) (x) \, dr \Big|^m \bigg]^{\frac{1}{m}} \Big\|_{L^{\infty}}}{|t-s|^{\tau}} .
\end{align}

\begin{remark}
Formally, we expect $u^{n,k}$ to converge to $u$ as $n,k \to \infty$ and thus $\big\{ u^{n,k} \big\}_{n,k,m, 3/4}$ to converge to 
\begin{align*}
\sup_{\substack{(s,t) \in \Delta_{[0,1]} \\ x \in [0,1]}}  \frac{\| \left( \EE^s | u_t(x)-O_t(x) - P_{t-s}(u_s(x)-O_s(x)) |^m  \right)^{\frac{1}{m}}\|_{L^\infty}}{|t-s|^{\frac{3}{4}}} ,
\end{align*}
which is finite by \eqref{eq:holder-reg-sol}. This suggests why $\big\{ u^{n,k} \big\}_{n,k,m, \tau}$ will arise in the computations below, but to optimise the rate of convergence, we shall rather choose $\tau$ close to $1/2$ rather than $3/4$.
\end{remark}

\begin{prop}\label{prop:main-SPDE} 
~
\begin{enumerate}[label=\upshape(\Roman*)]
\item \label{MainPropI} Let the same assumptions and notations as in Theorem \ref{thm:main-SPDE} hold.%
\begin{enumerate}[label=(\alph*)]
\item \underline{Regularity of the tamed Euler scheme}: First, for $\mathcal{D}$ a subset of $\N^2$ such that \eqref{eq:assump-bn-bounded} holds, we have
\begin{align}\label{eq:unifschemev0}
\sup_{(n,k) \in \mathcal{D}} \big\{ u^{n,k} \big\}_{n,k,m, \frac{1}{2} + \frac{\varepsilon}{2}}  < \infty .
\end{align}
In particular for $k_n=\lfloor n^{\frac{1}{1-\gamma+1/p}} \rfloor$ and $b^{k_n}=G_{\frac{1}{k_n}} b$, we have
 \begin{align}
 &\sup_{n\in \N} \big\{ u^{n,k_{n}} \big\}_{n,k_{n},m, \frac{1}{2} + \frac{\varepsilon}{2}} < \infty \label{eq:unifscheme} .
 \end{align}
\end{enumerate}
\begin{enumerate}
\item[(b)] \underline{The sub-critical case}: If $0>\gamma-1/p>-1$, then there exists a constant $C$ that depends on $m, p, \gamma, \varepsilon,  \|b\|_{\mathcal{B}_p^{\gamma}}, \| \psi_0 \|_{\mathcal{C}^{\frac{1}{2}-\varepsilon}}$ such that for all $n \in \N^*$ and $k \in \mathbb{N} $, the following bound holds:
\begin{align}\label{eq:main-result-SPDE2}
[\mathcal{E}^{n,k}]_{\mathcal{C}^{\frac{1}{2},0}_{[0,1],x} L^{m}} &  \leq C  \left( \| b-b^k \|_{\mathcal{B}_p^{\gamma-1}} + (1+\|b^k\|_\infty) n^{-\frac{1}{2}+\varepsilon} +  (1+\|b^k\|_\infty) \|b^k\|_{\mathcal{C}^1}\, n^{-1+\varepsilon} \right).
\end{align}
In particular for $k_n=\lfloor n^{\frac{1}{1-\gamma+1/p}} \rfloor$ and $b^{k_n}=G_{\frac{1}{k_n}} b$, we have
\begin{align}
&  [\mathcal{E}^{n,k_n}]_{\mathcal{C}^{\frac{1}{2},0}_{[0,1],x} L^{m}}  \leq C \, n^{-\frac{1}{2(1-\gamma+\frac{1}{p})}+\varepsilon}. \label{eq:rate2}
\end{align}
\item[(c)] \underline{The limit case}: Let $\gamma-1/p=-1$ and $p<+\infty$. Let $\M$ be the constant given by Proposition \ref{prop:bound-E1-critic}, let $\delta \in (0, e^{-\M})$ and $\zeta \in (0,1/2)$. For $\mathcal{D}$ a subset of $\N^2$ such that \eqref{eq:assump-bn-bounded} holds, there exists a constant $C$ that depends on $m, p, \gamma, \delta, \zeta, \varepsilon, \|b\|_{\mathcal{B}_p^{\gamma}}, \| \psi_0 \|_{\mathcal{C}^{\frac{1}{2}-\varepsilon}}$ such that for all $(n,k) \in \mathcal{D}$, the following bound holds:
\begin{equation}\label{eq:main-result-SPDE-critic2}
\begin{split}
\left[\mathcal{E}^{n, k}\right]_{\mathcal{C}_{[0,1],x}^{\frac{1}{2}-\zeta,0} L^m} &  \leq C \Big(  \| b - b^k \|_{\mathcal{B}_p^{\gamma-1}} (1+ |\log \| b - b^k \|_{\mathcal{B}_p^{\gamma-1}} |) + (1+\| b^k \|_{\infty}) n^{-\frac{1}{2}+\varepsilon} \\ 
& \quad + (1+\| b^k \|_{\infty} ) \| b^k \|_{\mathcal{C}^1}\,  n^{-1+\varepsilon}  \Big)^{e^{-\M}-\delta}.
\end{split}
\end{equation}
In particular for $k_n=\lfloor \sqrt{n} \rfloor$ and $b^{k_n}=G_{\frac{1}{k_n}} b$, we have
\begin{align}
&\left[\mathcal{E}^{n, k_{n}}\right]_{\mathcal{C}_{[0,1],x}^{\frac{1}{2}-\zeta,0} L^m}  \leq C \, n^{-\frac{1}{4}(e^{-\M}-\delta)} . \label{eq:rate2-critic}
\end{align}
\end{enumerate}
\item\label{MainPropII} Let the assumptions of Corollary \ref{cor:gama=d/p} hold. Then \ref{MainPropI}$(a)$ and \ref{MainPropI}$(b)$ hold with $\gamma=0$ and $p=\infty$.
\end{enumerate}

\end{prop} 
The proof of Proposition~\ref{prop:main-SPDE} is given in Section~\ref{sec:error-analysis}.

~

The proofs of Theorem \ref{thm:main-SPDE}, Corollary \ref{cor:bn-choice} and Corollary \ref{cor:gama=d/p} follow by using Proposition \ref{prop:main-SPDE} and Equation \eqref{eq:lien-u-E}.

~

For the rest of Section \ref{sec:convergence}, let $\gamma$ and $p$ which satisfy \eqref{eq:cond-gamma-p-H} and $b \in \mathcal{B}_p^\gamma$. There is $\mathcal{B}_p^\gamma \hookrightarrow \mathcal{B}_q^{\gamma-\frac{1}{p}-\frac{1}{q}}$ for any $q \geq p$, by Besov embedding. Setting $\tilde{\gamma}=\gamma-1/p+1/q$ and $\tilde{p}=q$, we have $b \in \mathcal{B}_{\tilde{p}}^{\tilde{\gamma}}$ and $\gamma-1 / p=\tilde{\gamma}-1 / \tilde{p}$, so that \eqref{eq:cond-gamma-p-H} is still satisfied in $\mathcal{B}_{\tilde{p}}^{\tilde{\gamma}}$. Thus there is no loss of generality in assuming that $p$ is as large as needed. We always assume that $p \geq m$. This allows us to apply the regularisation lemmas given in Section~\ref{sec:reg-O}, Section~\ref{sec:reg-On-1} and Section~\ref{sec:reg-On}.

\subsection{Proof of Proposition \ref{prop:main-SPDE}}\label{sec:error-analysis}

The proof of Proposition \ref{prop:main-SPDE} is split into three parts. First, in Section~\ref{subsubsec:prop31-}, we prove \eqref{eq:unifschemev0}, \eqref{eq:main-result-SPDE2} and \eqref{eq:main-result-SPDE-critic2}. Then in Section \ref{subsubsec:prop31}, we fix $k_n = \lfloor n^{1/(1-\gamma+1/p)} \rfloor $ and deduce \eqref{eq:unifscheme} and the rates of convergence \eqref{eq:rate2} and \eqref{eq:rate2-critic}, in the sub-critical case and the limit case respectively. Finally, the proof of Proposition \ref{prop:main-SPDE}\ref{MainPropII} is given in Section \ref{sec:proofgamma=d/p}.

\subsubsection{Proof of \eqref{eq:unifschemev0}, \eqref{eq:main-result-SPDE2} and \eqref{eq:main-result-SPDE-critic2}}\label{subsubsec:prop31-}

First, apply Proposition~\ref{cor:bound-Khn} to obtain directly \eqref{eq:unifschemev0}. 

\smallskip

Now for $t \in [0,1]$, $x \in \mathbb{T}$, recall that $v_t(x)$ is introduced in Definition \ref{def:true-solution}, and for $n \in \N^*$, $k\in \N$, define further
\begin{align}
v^{n,k}_t(x)  & = u^{n,k}_t(x) - O^n_t(x)- P^n_t \psi_0(x)  = \int_0^t P^n_{(t-r)_h} b^k(u^{n,k}_{r_h}) (x)\, dr , \label{def:v-hk}\\
v^k_t(x) & = \int_0^t P_{t-r} b^k(u_r)(x)\, dr \label{def:v-k} .
\end{align}
Recall that $\mathcal{E}^{n,k}$ was defined in \eqref{def:error}. For all $(s,t) \in \Delta_{[0,1]}$ and $x \in \mathbb{T}$, write
\begin{align*}
\mathcal{E}^{n,k}_{s,t}(x) & = v_t(x) - P_{t-s}v_s(x) - \int_s^t  P^n_{(t-r)_h} b^k (v^{n,k}_{r_h}+ P^n_{r_h} \psi_0+ O^{n}_{r_h})(x) \, dr  \\
& = v_t(x)- v^{k}_t(x) - P_{t-s}v_s(x) + P_{t-s}v^k_s(x) \\
&\quad + \int_s^t  \Big( P_{(t-r)} b^k (v_{r}+ P_r \psi_0 + O_r)(x)  -  P^n_{(t-r)_h} b^k (v^{n,k}_{r_h}+ P^n_{r_h} \psi_0 + O^{n}_{r_h})(x) \Big) dr \\
& = V^k_{s,t}(x)  + \mathcal{E}^{1,n,k}_{s,t}(x)+ \mathcal{E}^{2,n,k}_{s,t}(x)+ \mathcal{E}^{3,n,k}_{s,t} (x),
\end{align*}
where
\begin{align}
V^k_{s,t}(x) & := v_t(x)-v_t^k(x)-P_{t-s}(v_s-v^k_s)(x) , \label{def:Vk} \\
\mathcal{E}^{1,n,k}_{s,t}(x)  & :=  \int_s^t  \int_{\T} p_{t-r}(x,y) \Big( b^k(v_r(y)+ P_{r} \psi_0(y) +O_r(y)) -b^k(v^{n,k}_r(y) +P^n_{r} \psi_0(y)+ O_r(y)) \Big) \, dy dr  , \label{def:E1}\\
\mathcal{E}^{2,n,k}_{s,t}(x)  & := \int_{s}^t \int_{\T} p_{t-r}(x,y) \Big( b^k(v_r^{n,k}(y) +P^n_{r} \psi_0(y)+ O_r(y)) \nonumber \\ & \hspace{1.5cm}  - b^k(v_{r_h}^{n,k}(y_n) +P^n_{r_h} \psi_0(y_n)+ O^n_{r_h}(y_n)) \Big) \, dy dr ,\label{def:E2} \\
\mathcal{E}^{3,n,k}_{s,t} (x)& = \int_s^t \int_{\T} (p_{t-r}-p^n_{(t-r)_h})(x,y)\, b^k(v_{r_h}^{n,k}(y_n) + P^n_{r_h} \psi_0(y_n)+ O^n_{r_h}(y_n))\, dy dr \label{def:E4} .
\end{align}
Let $(S,T) \in \Delta_{[0,1]}$. For any $\zeta \in [0, \frac{1}{2})$, we have the upper bound
\begin{align}\label{eq:error-decomp}
[\mathcal{E}^{n,k}]_{\mathcal{C}^{\frac{1}{2}-\zeta,0}_{[S,T],x} L^m}  \leq [V^k]_{\mathcal{C}^{\frac{1}{2}-\zeta,0}_{[S,T],x} L^m} + [\mathcal{E}^{1,n,k}]_{\mathcal{C}^{\frac{1}{2}-\zeta,0}_{[S,T],x} L^m}
+[\mathcal{E}^{2,n,k}]_{\mathcal{C}^{\frac{1}{2}-\zeta,0}_{[S,T],x} L^m} 
+[\mathcal{E}^{3,n,k}]_{\mathcal{C}^{\frac{1}{2}-\zeta,0}_{[S,T],x} L^m}   .
\end{align}

We will bound the quantities that appear in the right-hand side of \eqref{eq:error-decomp}. The bound on $V^k$ is proven in Section~\ref{sec:reg-O-1}, see Corollary \ref{cor:bound-Vk}. The bound on $\mathcal{E}^{1,n,k}$ is proven in Section~\ref{sec:reg-E1}, see Corollary~\ref{cor:bound-E1} (sub-critical case) and Proposition~\ref{prop:bound-E1-critic} (limit case). The bound on $\mathcal{E}^{2,n,k}$ is proven in Section~\ref{subsec:reg-E2}, see Corollary~\ref{cor:bound-E2}. The bound on $\mathcal{E}^{3,n,k}$ is proven in Section~\ref{app:E3}, see Lemma~\ref{lem:bound-E4}.

\paragraph{Bound on $V^k$.}
\textit{In the sub-critical case}, we apply Corollary \ref{cor:bound-Vk}$(a)$ to get
\begin{align}\label{eq:proba-conv-SDE}
[V^{k}]_{\mathcal{C}^{\frac{1}{2},0}_{[S,T],x} L^{m}} & \leq  C \| b - b^k \|_{\mathcal{B}_p^{\gamma-1}}.
\end{align}
\textit{In the limit case}, we apply Corollary \ref{cor:bound-Vk}$(b)$ to get
\begin{align}\label{eq:proba-conv-SDE-critic}
[V^k]_{\mathcal{C}^{\frac{1}{2},0}_{[S,T],x} L^{m}} & \leq  C \| b - b^k \|_{\mathcal{B}_p^{\gamma-1}} (1+ |\log \| b - b^k \|_{\mathcal{B}_p^{\gamma-1}} |) .
\end{align}

\paragraph{Bound on $\mathcal{E}^{1,n,k}$.}

Define
\begin{align}\label{eq:defepsilonhn}
\epsilon(n,k) :=  [V^k]_{\mathcal{C}^{\frac{1}{2},0}_{[0,1],x} L^m}  +[\mathcal{E}^{2,n,k}]_{\mathcal{C}^{\frac{1}{2},0}_{[0,1],x} L^m} +[\mathcal{E}^{3,n,k}]_{\mathcal{C}^{\frac{1}{2},0}_{[0,1],x} L^m}  +  \sup_{\substack{t \in [0,1] \\x \in \T}} | P_t \psi_0(x) - P^n_t \psi_0(x) | .
\end{align}
\textit{In the sub-critical case}, Corollary~\ref{cor:bound-E1} gives the existence of $C>0$ such that for any $n\in \N^*$, $k\in\N$,
\begin{align*}
[\mathcal{E}^{1,n,k}]_{\mathcal{C}^{\frac{1}{2},0}_{[S,T],x} L^m}  & \leq C  \Big( \| \mathcal{E}^{n,k}_{0,\cdot} \|_{L^{\infty,\infty}_{[S,T],x} L^m}  + [\mathcal{E}^{n,k}]_{\mathcal{C}^{\frac{1}{2},0}_{[S,T],x} L^m}   + \epsilon(n,k) \Big) (T-S)^{\frac{1}{4}(1+\gamma-\frac{1}{p})} .
\end{align*}
Note that $ \| \mathcal{E}^{n,k}_{0,\cdot} \|_{L^{\infty,\infty}_{[S,T],x} L^m} \leq [\mathcal{E}^{n,k}]_{\mathcal{C}^{\frac{1}{2},0}_{[S,T],x} L^{m}}  + [\mathcal{E}^{n,k}]_{\mathcal{C}^{\frac{1}{2},0}_{[0,S],x} L^{m}}$. Therefore,
\begin{align}\label{eq:bound-E1-SDE}
[ \mathcal{E}^{1,n,k} ]_{\mathcal{C}^{\frac{1}{2},0}_{[S,T],x} L^m} \leq C \Big( [\mathcal{E}^{n,k}]_{\mathcal{C}^{\frac{1}{2},0}_{[S,T],x} L^{m}} + [\mathcal{E}^{n,k}]_{\mathcal{C}^{\frac{1}{2},0}_{[0,S],x} L^{m}}+ \epsilon(n,k) \Big) (T-S)^{\frac{1}{4}(1+\gamma-\frac{1}{p})} .
\end{align}
\textit{In the limit case}, Proposition~\ref{prop:bound-E1-critic}, applied with $\eta=\frac{\varepsilon}{2}$, gives the existence of $\ell_0$ such that for $\zeta \in (0,1/2)$ and $T-S \leq \ell_0$,
\begin{align*}
\sup_{x \in \mathbb{T}} \|\mathcal{E}^{1,n,k}_{s,t}(x) \|_{L^m} 
& \leq \mathbf{M}\Bigg(1+\bigg|\log \frac{T^{\frac{\varepsilon}{2}} \Big(1+ \big\{ u^{n,k} \big\}_{n,k,m, \frac{1}{2}+\frac{\varepsilon}{2}} \Big) }{\| \mathcal{E}^{1,n,k}_{0,\cdot} \|_{L_{[S, T],x}^{\infty,\infty} L^m}+\epsilon(n, k)}\bigg|\Bigg)\left(\| \mathcal{E}^{1,n,k}_{0,\cdot}  \|_{L_{[S, T],x}^{\infty,\infty} L^m}+\epsilon(n, k)\right)(t-s)
\\ 
& \quad +\mathbf{M}\left(\| \mathcal{E}^{1,n,k}_{0,\cdot} \|_{L_{[S, T],x}^{\infty,\infty} L^m}+\left[ \mathcal{E}^{1,n,k} \right]_{\mathcal{C}_{[S, T]}^{\frac{1}{2}-\zeta,0} L^m} + \epsilon(n,k) \right)(t-s)^{\frac{1}{2}}.
\end{align*} 
Since in the limit case $\mathcal{D}$ satisfies \eqref{eq:assump-bn-bounded}, Proposition~\ref{cor:bound-Khn} yields that
$\sup_{(n,k) \in \mathcal{D}} \big\{ u^{n,k} \big\}_{n,k,m, \frac{1}{2}+\frac{\varepsilon}{2}}< \infty$.
It follows that
\begin{align*}
\Bigg|\log \frac{T^{\frac{\varepsilon}{2}} \Big(1+\big\{ u^{n,k} \big\}_{n,k,m,\frac{1}{2}+\frac{\varepsilon}{2}}\Big)}{\| \mathcal{E}^{1,n,k} _{0,\cdot} \|_{L_{[S, T],x}^{\infty,\infty} L^m}+\epsilon(n, k)}\Bigg| 
\leq C+ {\frac{\varepsilon}{2}} |\log(T)| + \left| \log\left( \| \mathcal{E}^{1,n,k}_{0,\cdot} \|_{L_{[S, T],x}^{\infty,\infty} L^m}+\epsilon(n, k) \right) \right|.
\end{align*}
Therefore, we have
\begin{align*}
\sup_{x \in \mathbb{T}} \| \mathcal{E}^{1,n,k}_{s,t} (x) \|_{L^{m}}  &\leq  \M \, \Big(1+ \left| \log \Big( \|\mathcal{E}^{1,n,k}_{0,\cdot} \|_{L^{\infty,\infty}_{[S,T],x} L^{m}} +\epsilon(n,k)\Big) \right| \Big) \, \Big( \|\mathcal{E}^{1,n,k}_{0,\cdot} \|_{L^{\infty,\infty}_{[S,T],x} L^{m}} + \epsilon(n,k)\Big) \, (t-s) \\
 &\quad+ C \, \big( 1+ | \log(T) | (t-s)^{\frac{1}{2}} \big) \Big(\|\mathcal{E}^{1,n,k}_{0,\cdot} \|_{L^{\infty,\infty}_{[S,T],x} L^{m}} + [\mathcal{E}^{1,n,k} ]_{\mathcal{C}^{\frac{1}{2}-\zeta,0}_{[S,T],x} L^{m}} + \epsilon(n,k) \Big)\, (t-s)^{\frac{1}{2}}.
\end{align*}
Since $1 \geq T \geq t-s$, we get that $(s,t,T)\mapsto |\log T|(t-s)^{\frac{1}{2}}$ is a bounded mapping on the domain $\{(s, t, T): T \in$ $(0,1]$,\ $s<t \leq T\}$. It follows that
\begin{equation}\label{eq:bound-E1-SDE-critic}
\begin{split}
\sup_{x \in \mathbb{T}} \| \mathcal{E}^{1,n,k}_{s,t} (x) \|_{L^{m}} &\leq  \M \,  \left|\log \Big( \|\mathcal{E}^{1,n,k}_{0,\cdot} \|_{L^{\infty,\infty}_{[S,T],x} L^{m}} +\epsilon(n,k)\Big) \right| \, \Big(  \|\mathcal{E}^{1,n,k}_{0,\cdot} \|_{L^{\infty,\infty}_{[S,T],x} L^{m}} + \epsilon(n,k)\Big) \, (t-s) \\
 &\quad+ C \, \Big( \|\mathcal{E}^{1,n,k}_{0,\cdot} \|_{L^{\infty,\infty}_{[S,T],x} L^{m}} + [\mathcal{E}^{1,n,k} ]_{\mathcal{C}^{\frac{1}{2}-\zeta,0}_{[S,T],x} L^{m}} + \epsilon(n,k) \Big)\, (t-s)^{\frac{1}{2}}.
\end{split}
\end{equation}

\paragraph{Bounds on $\mathcal{E}^{2,n,k}$ and $\mathcal{E}^{3,n,k}$.}
The quantities $\mathcal{E}^{2,n,k}$ and $\mathcal{E}^{3,n,k}$ are respectively bounded using Corollary \ref{cor:bound-E2} and Lemma \ref{lem:bound-E4}. We get that for $\varepsilon \in (0, 1/2)$ and any $(s,t) \in \Delta_{[S,T]}$,
\begin{align*}
\sup_{x\in \mathbb{T}} \| \mathcal{E}^{2,n,k}_{s,t}(x) \|_{  L^{m}}  
& \leq C (1+  \| \psi_0 \|_{\mathcal{C}^{\frac{1}{2}-\varepsilon}}) \Big( (1+ \| b^k \|_{\infty}) n^{-\frac{1}{2}+\varepsilon} + (1+\| b^k \|_{\infty}) \| b^k \|_{\mathcal{C}^1}  n^{-1+\varepsilon} \Big) (t-s)^{\frac{1}{2}} 
\end{align*}
and
\begin{align*}
\sup_{x\in \mathbb{T}} \| \mathcal{E}^{3,n,k}_{s,t}(x)\|_{L^{m}}  \leq & C \|b^k\|_\infty n^{-\frac{1}{2}+\varepsilon} (t-s)^{\frac{1}{2}+\frac{\varepsilon}{2}} .
\end{align*}
Recall that $\| \psi_0 \|_{\mathcal{C}^{\frac{1}{2}-\varepsilon}}  < \infty$. Dividing by $(t-s)^{\frac{1}{2}}$ and taking the supremum over $(s,t) \in \Delta_{[S,T]}$ leads to
\begin{align}\label{eq:bound-E2-E3}
[ \mathcal{E}^{2,n,k} ] _{\mathcal{C}^{\frac{1}{2},0}_{[S,T],x} L^{m} } +   [\mathcal{E}^{3,n,k}]_{\mathcal{C}^{\frac{1}{2},0}_{[S,T],x} L^{m} } & \leq C \Big( (1+\| b^k \|_{\infty}) n^{-\frac{1}{2}+\varepsilon} + (1+\| b^k \|_{\infty} ) \| b^k \|_{\mathcal{C}^1}  n^{-1+\varepsilon} \Big) .
\end{align}
One could also have bounded the term $\mathcal{E}^{2,n,k}$ using Girsanov's theorem as this is done in \cite[Corollary 3.3.2]{butkovsky2021optimal}). But this leads to an exponential dependence on the norm $\|b^k\|_\infty$ which here would yield a poor rate of convergence as $\|b^k\|_\infty$ might go to infinity when $k\to \infty$.

\paragraph{Conclusion in the sub-critical case.} 
We start from \eqref{eq:error-decomp} with $\zeta=0$ and bound the terms $V^k$, $\mathcal{E}^{2,n,k}$ and $\mathcal{E}^{3,n,k}$ by the quantity $\epsilon(n,k)$ defined in \eqref{eq:defepsilonhn}, then bound $\mathcal{E}^{1,n,k}$ using \eqref{eq:bound-E1-SDE}:
\begin{align*}
[\mathcal{E}^{n,k}]_{\mathcal{C}^{\frac{1}{2},0}_{[S,T],x} L^{m}} &\leq \epsilon(n,k) + [\mathcal{E}^{1,n,k}]_{\mathcal{C}^{\frac{1}{2}-\zeta,0}_{[S,T],x} L^m} \\
&\leq (1+C)\, \epsilon(n,k) +C  \Big( [\mathcal{E}^{n,k}]_{\mathcal{C}^{\frac{1}{2},0}_{[S,T],x} L^{m}} + [\mathcal{E}^{n,k}]_{\mathcal{C}^{\frac{1}{2},0}_{[0,S],x} L^{m}}\Big) (T-S)^{\frac{1}{4}(1+\gamma-\frac{1}{p})}.
\end{align*}
Hence for $T-S\leq (2C)^{-(\frac{1}{4}(1+\gamma-\frac{1}{p}))^{-1}} =: \tilde{\ell}$, there is
\begin{align}\label{eq:boundseminorm-}
[\mathcal{E}^{n,k}]_{\mathcal{C}^{\frac{1}{2},0}_{[S,T],x} L^{m}} \leq C\epsilon(n,k) + [\mathcal{E}^{n,k}]_{\mathcal{C}^{\frac{1}{2},0}_{[0,S],x} L^m} .
\end{align}
Then the inequality 
\begin{align*}
 [\mathcal{E}^{n,k}]_{\mathcal{C}^{\frac{1}{2},0}_{[0,S],x} L^m} \leq  [\mathcal{E}^{n,k}]_{\mathcal{C}^{\frac{1}{2},0}_{[0,S-\tilde{\ell}],x} L^m}+ [\mathcal{E}^{n,k}]_{\mathcal{C}^{\frac{1}{2},0}_{[S-\tilde{\ell},S],x} L^m}
\end{align*}
can be plugged in \eqref{eq:boundseminorm-} and iterated until $S-m\tilde{\ell}$ is smaller than $0$ for  some $m\in \N$. It follows that 
\begin{align*}
[\mathcal{E}^{n,k}]_{\mathcal{C}^{\frac{1}{2},0}_{[0,1],x} L^{m}} \leq (m+1)\, C\, \epsilon(n,k) .
\end{align*}
In view of \eqref{eq:additional-error-terms}, \eqref{eq:proba-conv-SDE}, \eqref{eq:defepsilonhn} and \eqref{eq:bound-E2-E3}, we obtain the inequality \eqref{eq:main-result-SPDE2}.

\paragraph{Conclusion in the limit case.}
Let $C$ be the constant appearing in \eqref{eq:bound-E1-SDE-critic} and let ${\ell \in (0, C^{-\frac{1}{\zeta}} \wedge \ell_0)}$.
Fix $\zeta\in (0,1/2)$, divide by $(t-s)^{1/2-\zeta}$ in \eqref{eq:bound-E1-SDE-critic} and take the supremum for $(s,t) \in \Delta_{[S,T]}$ with $T-S \leq \ell$ to get
\begin{equation}\label{eq:bound-E1-SDE-critic-1}
\begin{split}
 [\mathcal{E}^{1,n,k} ]_{\mathcal{C}^{\frac{1}{2}-\zeta,0}_{[S,T],x} L^{m}} & \leq  \frac{\M}{1-C\ell^\zeta} \Big( \big|\log \big( \|\mathcal{E}^{1,n,k}_{0,\cdot} \|_{L^{\infty,\infty}_{[S,T],x} L^{m}}  +\epsilon(n,k)\big)\big| \Big) \, \Big( \|\mathcal{E}^{1,n,k}_{0,\cdot} \|_{L^{\infty,\infty}_{[S,T],x} L^{m}}  + \epsilon(n,k) \Big) \, (T-S)^{\frac{1}{2}+\zeta}\\ 
 & \quad + \frac{C}{1-C\ell^\zeta} \Big(\|\mathcal{E}^{1,n,k}_{0,\cdot} \|_{L^{\infty,\infty}_{[S,T],x} L^{m}} + \epsilon(n,k)\Big)\, (T-S)^{\zeta} .
\end{split}
\end{equation}
Denote $C_{1} =  \frac{C}{1-C\ell^\zeta}$ and $\mathbf{C}_{2} = \frac{\M}{1-C\ell^\zeta}$. For any $(s,t) \in \Delta_{[S,T]}$, there is
\begin{align*}
\sup_{x \in \mathbb{T}} \| \mathcal{E}^{1,n,k}_{0,t} (x) \|_{L^m} - \sup_{x \in \mathbb{T}} \| P_{t-s} \mathcal{E}^{1,n,k}_{0,s} (x) \|_{L^m}  
& \leq  \sup_{x \in \mathbb{T}} \| \mathcal{E}^{1,n,k}_{0,t} (x)-P_{t-s}\mathcal{E}^{1,n,k}_{0,s}(x) \|_{L^{m}}   \\
& = \sup_{x \in \mathbb{T}} \| \mathcal{E}^{1,n,k}_{s,t} (x) \|_{L^{m}}  \\
& \leq [\mathcal{E}^{1,n,k} ]_{\mathcal{C}^{\frac{1}{2}-\zeta,0}_{[S,T],x} L^{m}} (T-S)^{\frac{1}{2}-\zeta} .
\end{align*}
Moreover, for $x \in \mathbb{T}$ we have $\| P_{t-s} \mathcal{E}^{1,n,k}_{0,s} (x) \|_{L^m}  \leq P_{t-s}\big[\|\mathcal{E}^{1,n,k}_{0,s}\|_{L^m}\big](x)$. 
Hence the contraction property of the semigroup $P$ yields $\| P_{t-s} \mathcal{E}^{1,n,k}_{0,s} (x) \|_{L^m}  \leq \sup_{x \in \mathbb{T}}  \|\mathcal{E}^{1,n,k}_{0,s}(x) \|_{L^m}$. In particular for $s=S$, it follows that
\begin{align*}
\begin{split}
&  \sup_{x \in \mathbb{T}} \| \mathcal{E}^{1,n,k}_{0,t}(x) \|_{L^m} - \sup_{x \in \mathbb{T}} \| \mathcal{E}^{1,n,k}_{0,S}(x) \|_{L^m}  \\ 
& \leq \mathbf{C}_{2} \Big( \big|\log \big( \|\mathcal{E}^{1,n,k}_{0,\cdot} \|_{L^{\infty,\infty}_{[S,T],x} L^{m}}  +\epsilon(n,k)\big)\big| \Big) \, \Big( \|\mathcal{E}^{1,n,k}_{0,\cdot} \|_{L^{\infty,\infty}_{[S,T],x} L^{m}}  + \epsilon(n,k) \Big) \, (T-S) \\ 
& \quad+ C_{1} \Big(\|\mathcal{E}^{1,n,k}_{0,\cdot} \|_{L^{\infty,\infty}_{[S,T],x} L^{m}} + \epsilon(n,k)\Big)\, (T-S)^{\frac{1}{2}}.
\end{split}
\end{align*}
Now in the space $E = L^\infty_{x}L^m$, apply Lemma~\ref{lem:rate-critical} to the $E$-valued functions $f^{n,k}_{s,t} = \mathcal{E}^{1,n,k}_{s,t}$ which satisfy $f^{n,k}_{s,t} = f^{n,k}_{0,t} - P_{t-s}f^{n,k}_{0,s}$ and by what precedes, $ \|f^{n,k}_{0,t}\|_{L^\infty_{x}L^m} - \|f^{n,k}_{0,s}\|_{L^\infty_{x}L^m} \leq \|f^{n,k}_{s,t}\|_{L^\infty_{x}L^m}$.
Hence it comes that for any $\delta \in (0, e^{-\M/4})$, there exists $\bar{\epsilon}$ such that whenever $\epsilon(n,k) \leq \bar{\epsilon}$,
\begin{align*}
\| \mathcal{E}^{1,n,k}_{0,\cdot} \|_{L^{\infty,\infty}_{[0,1],x} L^m} \leq \epsilon(n,k)^{e^{-\mathbf{C}_{2}}-\delta} .
\end{align*}
Choosing $\ell$ small enough, we get $e^{-\mathbf{C}_{2}}  = e^{-\frac{\M}{1-C\ell^\zeta}} \geq e^{-\M}-\delta$ and thus
\begin{align*}
\| \mathcal{E}^{1,n,k}_{0,\cdot} \|_{L^{\infty,\infty}_{[0,1],x} L^m} \leq \epsilon(n,k)^{e^{-\M}-2\delta} ,
\end{align*}
whenever $\epsilon(n,k) \leq \bar{\epsilon}$.

Define $\epsilon <\bar{\epsilon} \wedge 1$ so that $\epsilon^{e^{-\mathbf{M}}-2 \delta}<\frac{e^{-1}}{2}$. Now if $\epsilon(n, k)<\epsilon$, it comes that
\begin{align*}
\| \mathcal{E}^{1,n,k}_{0,\cdot} \|_{L^{\infty,\infty}_{[0,1],x} L^m}+\epsilon(n,k) \leq \epsilon(n,k)^{e^{-\M}-2\delta}+\epsilon(n,k)  \leq 2 \epsilon(n,k)^{e^{-\M}-2 \delta} < e^{-1}.
\end{align*}
The mapping $x \mapsto x|\log (x)|$ is increasing on the interval $\left(0, e^{-1}\right)$, thus by \eqref{eq:bound-E1-SDE-critic-1}, it follows that on any interval $I$ of size $\ell$, there is
\begin{align*}
\left[\mathcal{E}^{1,n, k}\right]_{\mathcal{C}_{I,x}^{\frac{1}{2}-\zeta,0} L^m} \leq C \epsilon(n, k)^{e^{-\M}-2 \delta}(1+|\log (\epsilon(h, n))|) .
\end{align*}
Since the definition of $\ell$  does not depend on $\epsilon(n, k)$, one needs to sum at most $\frac{1}{\ell}$ of these bounds, so that if $\epsilon(n, k)<\epsilon$, it comes that
\begin{align}\label{eq:boundE1-critic-epsilon}
\left[\mathcal{E}^{1,n, k}\right]_{\mathcal{C}_{[0,1],x}^{\frac{1}{2}-\zeta,0} L^m} \leq C \epsilon(n, k)^{e^{-\M}-2 \delta}(1+|\log (\epsilon(n, k))|) \leq C \epsilon(n, k)^{e^{-\M}-4 \delta} .
\end{align}

Now to prove that \eqref{eq:boundE1-critic-epsilon} holds even for $\epsilon(n, k)$ larger than $\epsilon$, it suffices to establish that ${\sup_{n,k\in \mathcal{D}} [ \mathcal{E}^{1,n,k}]_{\mathcal{C}_{[0,1],x}^{\frac{1}{2}-\zeta,0} L^m} < \infty}$.
Recall the definitions \eqref{def:E1}, \eqref{def:Vk}, \eqref{def:error} and \eqref{eq:discrete-seminorm} and observe that
\begin{align*}
\sup_{n,k\in \mathcal{D}} [ \mathcal{E}^{1,n,k} ]_{\mathcal{C}_{[0,1],x}^{\frac{1}{2}-\zeta,0} L^m}  
& \leq \sup_{n,k\in \mathcal{D}} [ \mathcal{E}^{n,k} ]_{\mathcal{C}_{[0,1],x}^{\frac{1}{2},0} L^m} + \sup_{n,k\in \mathcal{D}} [ V^{k}]_{\mathcal{C}_{[0,1],x}^{\frac{1}{2},0} L^m}  \\
& \leq \sup_{n,k\in \mathcal{D}}[ u-O  ]_{\mathcal{C}_{[0,1],x}^{\frac{1}{2},0} L^m} +\sup_{n,k\in \mathcal{D}}\big\{ u^{n,k} \big\}_{n,k,m, \frac{1}{2}+\frac{\varepsilon}{2}}+ \sup_{ k\in \N}  [V^{k} ]_{\mathcal{C}^{\frac{1}{2},0}_{[0,1],x}L^m}.
\end{align*}
From \eqref{eq:proba-conv-SDE}, \eqref{eq:proba-conv-SDE-critic}, Proposition~\ref{cor:bound-Khn} and the property $[u-O]_{\mathcal{C}^{3/4,0}_{[0,1],x} L^{m,\infty}} < \infty$, we have  
\begin{align*}
\sup_{n,k\in \mathcal{D}} [ \mathcal{E}^{1,n,k} ]_{\mathcal{C}_{[0,1],x}^{\frac{1}{2}-\zeta,0} L^m}   & \leq  C + C\sup_{ k\in \N} \| b - b^k \|_{\mathcal{B}_p^{\gamma-1}} (1+ |\log \| b - b^k \|_{\mathcal{B}_p^{\gamma-1}} |) ,
\end{align*}
which is finite since $\| b - b^k \|_{\mathcal{B}_p^{\gamma-1}} \leq \| b \|_{\mathcal{B}_p^{\gamma}} +  \| b^k \|_{\mathcal{B}_p^{\gamma}}  \leq 2 \| b \|_{\mathcal{B}_p^{\gamma}}$. 
Using \eqref{eq:boundE1-critic-epsilon}, we conclude that there exists $C$ such that for all $(n, k) \in \mathcal{D}$,
\begin{align*}
\left[\mathcal{E}^{1,n,k}\right]_{\mathcal{C}_{[0,1],x}^{\frac{1}{2}-\zeta,0} L^m} \leq C \epsilon(n, k)^{e^{-\M}-4 \delta} .
\end{align*}
Thus in view of the previous inequality and of the inequalities \eqref{eq:additional-error-terms}, \eqref{eq:proba-conv-SDE-critic}, \eqref{eq:defepsilonhn} and \eqref{eq:bound-E2-E3}, we obtain the inequality~\eqref{eq:main-result-SPDE-critic2}.

\subsubsection{Proof of \eqref{eq:unifscheme}, \eqref{eq:rate2} and \eqref{eq:rate2-critic}} \label{subsubsec:prop31}%
Let $b^{k_n}=G_{\frac{1}{k_n}}b$ with $k_n$ given by $k_n = \lfloor n^{\frac{1}{1-\gamma+1/p}} \rfloor $. The inequalities \eqref{eq:bn-inf}-\eqref{eq:bn-C1} imply that $(b^{k_n})_{n\in \N}$ satisfies \eqref{eq:assump-bn-bounded} on the set $\mathcal{D}=\{(n,k_{n}),\ n\in \N^*\}$. Thus we deduce \eqref{eq:unifscheme} directly from \eqref{eq:unifschemev0}.

\paragraph{The sub-critical case.}
Let $b^{k_{n}}=G_{\frac{1}{k_{n}}}b$ with $k_{n}$ given by $k_n = \lfloor n^{\alpha^\star} \rfloor $ and $\alpha^\star =1/(1-\gamma+1/p)$. In view of Lemma~\ref{eq:reg-S}, $b^{k_{n}}$ satisfies \eqref{eq:bn-b}, \eqref{eq:bn-inf}, \eqref{eq:bn-C1}. Thus in view of \eqref{eq:main-result-SPDE2}, we get
\begin{align*}
[\mathcal{E}^{n,k_n}]_{\mathcal{C}^{\frac{1}{2},0}_{[0,1],x} L^m} &
 \leq C \Big( \lfloor n^{\alpha^\star} \rfloor ^{-\frac{1}{2}}  + (1+ \lfloor n^{\alpha^\star} \rfloor ^{-\frac{1}{2}(\gamma-\frac{1}{p})} ) n^{-\frac{1}{2}+\varepsilon} + (1+\lfloor n^{\alpha^\star} \rfloor^{-\frac{1}{2}(\gamma-\frac{1}{p})}) ( 1+\lfloor n^{\alpha^\star} \rfloor^{\frac{1}{2}-\frac{1}{2}(\gamma-\frac{1}{p} )} ) n^{-1+\varepsilon}  \Big) .
\end{align*}
Using $-1/2(\gamma-1/p)>0$, we get
\begin{align*}
[\mathcal{E}^{n,k_n}]_{\mathcal{C}^{\frac{1}{2},0}_{[0,1],x} L^m} & \leq C \Big(n^{-\frac{\alpha^\star}{2}} +  n^{- \frac{\alpha^\star}{2}(\gamma-\frac{1}{p})} n^{-\frac{1}{2}+\varepsilon} +  n^{\frac{\alpha^\star}{2}} n^{-\alpha^\star (\gamma-\frac{1}{p} )} n^{-1+\varepsilon}\Big) \\
& \leq C n^{-\frac{1}{2(1-\gamma+\frac{1}{p})}+\varepsilon} .
\end{align*}
This proves \eqref{eq:rate2}. We remark that $\alpha^\star$ has been chosen to minimise the error $[\mathcal{E}^{n,k_n}]_{\mathcal{C}^{1/2,0}_{[0,1],x} L^m}$, as soon as $b^{k_{n}}$ is of the form $G_{\frac{1}{k_{n}}}b$ and $k_n$ is of the form $\lfloor n^{\alpha} \rfloor$.

\paragraph{The limit case.} 
As before, we get using \eqref{eq:main-result-SPDE-critic2} that for $\delta \in (0, e^{-\M})$, with $\M$ the constant given by Proposition \ref{prop:bound-E1-critic},
\begin{align*}
\left[\mathcal{E}^{1,n,k}\right]_{\mathcal{C}_{[0,1],x}^{\frac{1}{2}-\zeta,0} L^m} & \leq C \Big( (1+\| b^k \|_\infty) n^{-\frac{1}{2}+\varepsilon} + (1+\| b^k \|_\infty) \| b^k \|_{\mathcal{C}^1} n^{-1+\varepsilon}  \\ & \quad + \| b^k - b \|_{\mathcal{B}_p^{\gamma-1}} (1+| \log(\| b^k - b \|_{\mathcal{B}_p^{\gamma-1}}) | )  \Big)^{e^{-\M}-\delta} .
\end{align*}
Then letting $b^{k_{n}}=G_{\frac{1}{k_{n}}} b$ with $k_{n}=\lfloor n^{\frac{1}{2}} \rfloor$ and using \eqref{eq:bn-b}, \eqref{eq:bn-inf} and \eqref{eq:bn-C1}, 
\begin{align*}
\left[\mathcal{E}^{1,n,k}\right]_{\mathcal{C}_{[0,1],x}^{\frac{1}{2}-\zeta,0} L^m} & \leq C \Big( \lfloor n^{\frac{1}{2}} \rfloor ^{-\frac{1}{2}} (1+|\log( \lfloor n^{\frac{1}{2}} \rfloor ^{-\frac{1}{2}})|)  +  \lfloor n^{\frac{1}{2}} \rfloor ^{-\frac{1}{2}(\gamma-\frac{1}{p})}  n^{-\frac{1}{2}+\varepsilon} + \lfloor n^{\frac{1}{2}} \rfloor ^{\frac{1}{2}-(\gamma-\frac{1}{p} )} n^{-1+\varepsilon}  \Big)^{e^{-\M}-\delta} \\
& \leq C \left( n^{-\frac{1}{4}+\varepsilon} |\log(n)|  \right)^{e^{-\M}-\delta},
\end{align*}
which proves \eqref{eq:rate2-critic}. 
We remark again that $k_{n}=\lfloor n^{\frac{1}{2}} \rfloor$ has been chosen to minimise the error $\left[\mathcal{E}^{1,n,k}\right]_{\mathcal{C}_{[0,1],x}^{\frac{1}{2}-\zeta,0} L^m}$, as soon as $b^{k_{n}}$ is of the form $G_{\frac{1}{k_{n}}}b$ and $k_n$ is of the form $\lfloor n^{\alpha} \rfloor$.

\subsubsection{Proof of Proposition \ref{prop:main-SPDE}\ref{MainPropII}}\label{sec:proofgamma=d/p}

Assume that $b \in \mathcal{B}_\infty^0$. Let $\varepsilon \in (0, 1/2)$ be such that \eqref{eq:assump-bn-bounded} holds. Then $b$ also belongs to $\mathcal{B}_\infty^{-\varepsilon/2}$ and \eqref{eq:cond-gamma-p-H} holds for $\gamma \equiv -\varepsilon/2$ and $p\equiv \infty$. Thus \eqref{eq:unifschemev0} and \eqref{eq:unifscheme} hold by Proposition~\ref{prop:main-SPDE}\ref{MainPropI}$(a)$.

 To prove \eqref{eq:main-result-SPDE2} and \eqref{eq:rate2}, we aim to apply Proposition~\ref{prop:main-SPDE}\ref{MainPropI}$(b)$ with $\gamma\equiv -\varepsilon/2$ and $p\equiv\infty$. Since \eqref{eq:assump-bn-bounded} holds for $\varepsilon$, it also holds for $\varepsilon/2$. Besides, $\psi_{0}$ is assumed to be in $\mathcal{C}^{1/2}$, so it is also in $\mathcal{C}^{1/2-\varepsilon/2}$, thus we can apply Proposition~\ref{prop:main-SPDE}\ref{MainPropI}$(b)$ with $\varepsilon/2$ instead of $\varepsilon$ in the original statement to get that
\begin{align*}
 [\mathcal{E}^{n,k}]_{\mathcal{C}^{\frac{1}{2},0}_{[0,1],x} L^{m}} 
 & \leq C \left( \| b^k-b \|_{\mathcal{B}_\infty^{-\frac{\varepsilon}{2}-1}} + (1+\|b^k\|_\infty )n^{-\frac{1}{2}+\frac{\varepsilon}{2}} + (1+\|b^k\|_\infty) \|b^k \|_{\mathcal{C}^1}\,  n^{-1+\frac{\varepsilon}{2}}  \right) .
\end{align*}
 Then use finally that $\| b^k-b \|_{\mathcal{B}_\infty^{-\varepsilon/2-1}}\leq C\| b^k-b \|_{\mathcal{B}_\infty^{-1}}$ to get \eqref{eq:main-result-SPDE2} with $\gamma=0$ and $p=\infty$.
  
Now  define $b^{k_{n}}= G_{\frac{1}{k_{n}}} b$ and $k_{n}=\lfloor n^{-\frac{1}{1+\varepsilon/2}}\rfloor$. Using \eqref{eq:bn-b}, \eqref{eq:bn-inf} and \eqref{eq:bn-C1} as in Section~\ref{subsubsec:prop31} (in the sub-critical case) leads to 
\begin{align*}
 [\mathcal{E}^{n,k_{n}}]_{\mathcal{C}^{\frac{1}{2},0}_{[0,1],x} L^{m}} 
 &  \leq C \left( n^{-\frac{1}{2(1+\varepsilon/2)}} + n^{-\frac{1}{2}+\frac{\varepsilon}{2}+\frac{\varepsilon/2}{2(1+\varepsilon/2)}} +  n^{-1+\frac{\varepsilon}{2}+\frac{1}{2(1+\varepsilon/2)}+\frac{\varepsilon/2}{1+\varepsilon/2}} \right) \\
 & \leq C\, n^{-\frac{1}{2}+\varepsilon}. %
\end{align*}

\section{Regularisation effect of the Ornstein-Uhlenbeck process, and H\"older regularity of $V^k$ and $\mathcal{E}^{1,n,k}$}\label{sec:reg-O}

In this section, the regularisation properties of the process $O$ are proven. We define some notations that will be useful throughout the whole section. Let $m\in [1,\infty)$, $\alpha > 0$, $I$ an interval of $[0,1]$, and let $w: [0,1] \times \T \times \Omega \to \R$. 
The process $\delta_P w : \Delta_{[0,1]} \times \T \times \Omega \to \R$ is defined by
\begin{align*}
\delta_P w_{s,t}(x) = w_t(x) - P_{t-s}w_s(x),~\forall (s,t) \in \Delta_{[0,1]},~\forall x \in \T .
\end{align*}
In view of the definition \eqref{def:holder-norm-W} of $[\delta_P w]_{\mathcal{C}^{\alpha,0}_{I,x} L^{m}}$, we write
\begin{align}\label{def:holder-norm}
[w]_{\mathcal{C}^{\alpha,0}_{I,x} L^{m}}  \equiv [\delta_P w]_{\mathcal{C}^{\alpha,0}_{I,x} L^{m}} = \sup_{\substack{(s,t)\in \Delta_{I}\\  x \in \T}} \frac{\|  w_t(x) - P_{t-s}w_s(x)\|_{L^m}}{|t-s|^{\alpha}} .
\end{align}
For $q \in [1,\infty]$, define also
\begin{align*}%
[w]_{\mathcal{C}^{\alpha,0}_{I,x} L^{m,q}}  
= \sup_{\substack{(s,t)\in \Delta_{I} \\  x \in \T}} \frac{\| \left( \EE^s | w_t(x) - P_{t-s}w_s(x) |^m  \right)^{\frac{1}{m}}\|_{L^q}}{|t-s|^{\alpha}} .
\end{align*}
For $\phi : [0,1] \times \T \times \Omega \rightarrow \mathbb{R}$, the natural quantity that will appear in the regularisation lemmas of this section are the following pseudo-norms:
\begin{equation}\label{def:holder-norm-phi}
\begin{split}
&\big( \phi \big)_{\mathcal{C}^{\alpha,0}_{I,x} L^{m,q}} = \sup_{\substack{s,u,r \in I \\ s \leq u \leq r,\, s\neq r \\ x \in \T}}  \frac{\|\left( \EE^s| \EE^u \phi_r(x) - \EE^s \phi_r(x)|^m \right)^{\frac{1}{m}} \|_{L^q}}{|r-s|^\alpha}, \\
&\big( \phi \big)_{\mathcal{C}^{\alpha,0}_{I,x} L^{m}} = \big( \phi \big)_{\mathcal{C}^{\alpha,0}_{I,x} L^{m,m}} = \sup_{\substack{ s \leq u \leq r \in I,\, s\neq r \\ x \in \T}}  \frac{\| \EE^u \phi_r(x) - \EE^s \phi_r(x) \|_{L^m}}{|r-s|^\alpha} .
\end{split}
\end{equation}

The following simple inequalities will be frequently used in the sequel:
\begin{equation}
[w ]_{\mathcal{C}_{I,x}^{\alpha,0} L^m}= [w ]_{\mathcal{C}_{I,x}^{\alpha,0} L^{m,m}}\leq [w ]_{\mathcal{C}_{I,x}^{\alpha,0} L^{m,\infty}} \quad \text{and} \quad (\phi)_{\mathcal{C}_{I,x}^{\alpha,0} L^m}= (\phi)_{\mathcal{C}_{I,x}^{\alpha,0} L^{m,m}}\leq (\phi)_{\mathcal{C}_{I,x}^{\alpha,0} L^{m,\infty}}.
\end{equation}

~

In the general regularisation results of this section, namely Proposition~\ref{prop:drift-approx}, Lemma~\ref{lem:ssl-o-on}, Lemma~\ref{cor:ssl-4} and Proposition~\ref{prop:ssl-o-on-2}, upper bounds will mostly be given in terms of the pseudo-norm $(\cdot)_{\mathcal{C}_{I,x}^{\alpha,0} L^{m,q}}$. As we shall see in the next paragraph, $(\cdot)_{\mathcal{C}_{I,x}^{\alpha,0} L^{m,\infty}} \lesssim [\cdot]_{\mathcal{C}_{I,x}^{\alpha,0} L^{m,\infty}}$, see \eqref{eq:comp-seminorms1}. Hence the aforementioned results are sharper when expressed in terms of $(\cdot)_{\mathcal{C}_{I,x}^{\alpha,0} L^{m,q}}$. However, as seen in Section~\ref{sec:convergence}, we are able to develop a Gr\"onwall argument on $[\mathcal{E}^{k,n}]_{\mathcal{C}^{1/2,0}_{[0,1],x} L^{m}}$, which provides therefore a stronger control than if we had bounded $(\mathcal{E}^{k,n})_{\mathcal{C}^{1/2,0}_{[0,1],x} L^{m}}$, due to the inequality \eqref{eq:comp-seminorms1}. Hence the latter inequality will be used in Section~\ref{sec:reg-E1}, combined with the general regularisation results, to bound $[\mathcal{E}^{1,n,k}]_{\mathcal{C}^{1/2,0}_{[S,T],x} L^m}$ in terms of $[\cdot]_{\mathcal{C}_{I,x}^{\alpha,0} L^{m,\infty}}$ pseudo-norms.

 Let $s \in [0,1]$ and $Y$ be an $\mathcal{F}_s$-measurable random variable. Then for any integrable random variable $X$, the conditional Jensen's inequality yields $\|( \EE^s| X - \EE^s X|^m )^{\frac{1}{m}} \|_{L^q} \leq 2\|( \EE^s| X - Y |^m )^{\frac{1}{m}} \|_{L^q}$.
  It follows that for any $\mathbb{F}$-adapted process $\rho : [0,1] \times \T \times \Omega \rightarrow \mathbb{R}$, any $s\leq u \leq r$ (with $s<r$),
\begin{align*}
 \frac{\|\left( \EE^s| \EE^u \rho_r(x) - \EE^s \rho_r(x)|^m \right)^{\frac{1}{m}} \|_{L^q}}{|r-s|^\alpha}  
 & =  \frac{\|\left( \EE^s| \EE^u \rho_r(x) - \EE^s \EE^u \rho_r(x)|^m \right)^{\frac{1}{m}} \|_{L^q}}{|r-s|^\alpha}   \\
& \leq  2  \frac{\|\left( \EE^s| \EE^u \rho_r(x) -  Y |^m \right)^{\frac{1}{m}} \|_{L^q}}{|r-s|^\alpha} ,
\end{align*}
and by the conditional Jensen's inequality again,
\begin{align}
\frac{\|\left( \EE^s| \EE^u \rho_r(x) - \EE^s \rho_r(x)|^m \right)^{\frac{1}{m}} \|_{L^q}}{|r-s|^\alpha}  &  \leq 2  \frac{\|\left( \EE^s\EE^u |  \rho_r(x) -  Y  |^m \right)^{\frac{1}{m}} \|_{L^q}}{|r-s|^\alpha} \nonumber \\
& \leq 2 \frac{\|\left( \EE^s |  \rho_r(x) -  Y  |^m \right)^{\frac{1}{m}} \|_{L^q}}{|r-s|^\alpha} \label{eq:comp-seminorms} .
\end{align}
Applying \eqref{eq:comp-seminorms} to  $Y=P_{r-s}\rho_s$ and taking supremums gives
\begin{align}
\big( \rho \big)_{\mathcal{C}^{\alpha,0}_{I,x} L^{m,q}} \leq 2 [\rho]_{\mathcal{C}^{\alpha,0}_{I,x} L^{m,q}}  \label{eq:comp-seminorms1} .
\end{align}
Moreover, if $\rho= u^{n,k}-O^n$, then letting $m=q$, $Y = \int_0^s P^n_{(r-\theta)_h} b^k(u^{n,k}_{\theta_h})(x) d \theta$ in \eqref{eq:comp-seminorms} and taking supremums gives
\begin{align}\label{eq:comp-seminorms2-0}
\big( u^{n,k}-O^n \big)_{\mathcal{C}^{\alpha,0}_{I,x} L^{m}} 
\leq 2 \sup_{\substack{(s,r)\in \Delta_{I}\\  x \in \T}}  \frac{\Big\|  { \displaystyle \int_s^{r} } P^n_{(r-\theta)_h} b^k(u^{n,k}_{\theta_h})(x) \, d\theta  \Big\|_{L^m}}{|r-s|^{\alpha}} .
\end{align}
In particular, for the pseudo-norm $\{ u^{n,k} \}_{n,k,m,\alpha}$ introduced in \eqref{eq:discrete-seminorm}, this reads as
\begin{align}\label{eq:comp-seminorms2}
\big( u^{n,k}-O^n \big)_{\mathcal{C}^{\alpha,0}_{[0,1],x} L^{m}} \leq 2 \{ u^{n,k} \}_{n,k,m,\alpha}.  
\end{align}
Finally, for $\psi = u-O$, $\phi=u^{n,k}-O^n$ and $\rho=\psi-\phi$, then letting $Y = P_{r-s}(u_s-O_s)(x)-\int_0^s P^n_{(r-\theta)_h} b^k(u^{n,k}_{\theta_h}) (x) d \theta$ in \eqref{eq:comp-seminorms} with $q=m$ and recalling the definition of $\mathcal{E}^{n,k}$ in \eqref{def:error}, one gets
\begin{align}
\big( \psi-\phi \big)_{\mathcal{C}^{\alpha,0}_{I,x} L^{m}} 
& \leq  2\sup_{\substack{(s,r)\in \Delta_{I}\\  x \in \T}}  \frac{\Big\|  { u_r(x)-O_r(x)-P_{r-s}(u_s-O_s)(x) -\displaystyle \int_s^{r} } P^n_{(r-\theta)_h} b^k(u^{n,k}_{\theta_h})(x) \, d\theta  \Big\|_{L^m}}{|r-s|^{\alpha}} \nonumber \\
& = 2 [ \mathcal{E}^{n,k}]_{\mathcal{C}^{\alpha,0}_{I,x} L^{m}} . \label{eq:comp-seminorms3}
\end{align}

\subsection{Regularisation properties of the Ornstein-Uhlenbeck process}\label{sec:reg-O-1}

The OU process regularises quantities of the form $\int_0^t P_{t-r}f(\phi_r + O_r)(x) \, dr$. This will be useful in Section~\ref{sec:reg-V}.

\begin{prop}\label{prop:drift-approx}
Let $m \in[2, \infty), q \in [m,+\infty]$ and $p \in[q,+\infty]$.

\begin{enumerate}[label=(\alph*)]
 \item\label{item:4.1(a)} \underline{The sub-critical case}:
let $\beta \in\left(-2, 0\right)$ such that $\beta-1 / p>-2$. Let $\tau \in(0,1)$ be such that $\frac{1}{4}(\beta-1 / p-1)+\tau>0$. There exists a constant $C>0$ such that for any $0 \leq S < T\leq 1$ and any $(s,t) \in \Delta_{[S,T]}$, any $f \in \mathcal{C}_b^{\infty}\left(\mathbb{R}, \mathbb{R}\right) \cap \mathcal{B}_p^\beta$, and any $\psi : [0,1] \times \T \times \Omega \rightarrow \mathbb{R}$ such that $\psi$ is adapted to $\mathbb{F}$, there is
\begin{align*}
& \sup_{x \in \mathbb{T}} \left\|\left(\mathbb{E}^s\left|\int_s^t \int_{\T} p_{T-r}(x,y) f\left(O_r(y)+\psi_r(y) \right) d y d r\right|^m\right)^{\frac{1}{m}}\right\|_{L^{q}} \\ 
&\quad  \leq  C\|f\|_{\mathcal{B}_p^\beta}(t-s)^{1+\frac{1}{4}\left(\beta-\frac{1}{p}\right)} 
+C\|f\|_{\mathcal{B}_p^\beta}[\psi]_{\mathcal{C}_{[S, T],x}^{\tau,0} L^{m, q}}(t-s)^{1+\frac{1}{4}\left(\beta-\frac{1}{p}-1\right)+\tau} .
\end{align*}

\item\label{item:4.1(b)} \underline{The limit case}: let $\beta-1 / p=-2$ and $p<+\infty$. There exists a constant $C>0$ such that for any $0 \leq S < T\leq 1$ and any $(s,t) \in \Delta_{[S,T]}$, any $f \in \mathcal{C}_b^{\infty}\left(\mathbb{R}, \mathbb{R} \right) \cap \mathcal{B}_p^{\beta+1}$, and any $\psi : [0,1] \times \T \times \Omega \rightarrow \mathbb{R}$ such that $\psi$ is adapted to $\mathbb{F}$, there is
\begin{align}\label{eq:3.5b}
& \sup_{x \in \mathbb{T}} \left\|\int_{s}^t \int_{\T} p_{T-r}(x,y) f\left(O_r(y)+\psi_r(y) \right) dy d r\right\|_{L^m} \nonumber \\ 
&\quad  \leq C \|f\|_{\mathcal{B}_p^\beta}\Big(1+\Big|\log \frac{\|f\|_{\mathcal{B}_p^\beta}}{\|f\|_{\mathcal{B}_p^{\beta+1}}}\Big|\Big) \Big(1+(\psi)_{\mathcal{C}_{[S,T],x}^{\frac{3}{4},0} L^{m}} \Big) (t-s)^{\frac{1}{2}} .
\end{align}
\end{enumerate}
\end{prop}

\begin{proof}
We refer to \cite[Lemma 5.2$(i)$]{athreya2020well} for a proof of $(a)$. Let us provide the proof of $(b)$.
Let $0 \leq S<T\leq 1$. For $x \in \mathbb{T}$ and $(s,t) \in \Delta_{[S,T]}$, let 
\begin{align} \label{eq:A}
    A_{s,t}:=\int_{s}^{t} \int_{\T} p_{T-r}(x,y) f(O_r(y)+\EE^{s} \psi_{r} (y)) dy dr \ \text{and } \mathcal{A}_{t}:=\int_S^{t} \int_{\T} p_{T-r}(x,y) f(O_r(y)+\psi_r(y)) dy dr.
\end{align}

Assume that $(\psi)_{\mathcal{C}^{3/4,0}_{[S,T],x}L^m}<\infty$, otherwise 
\eqref{eq:3.5b} immediately holds. We aim to apply the Stochastic Sewing Lemma with critical exponent (sse \cite[Theorem 4.5]{athreya2020well}), which is recalled in Remark~\ref{rmk:critical-sewing}. Let $\varepsilon \in (0,1)$ small enough to be specified later.
To show that \eqref{eq:condsew1} and \eqref{eq:condsew2} hold true respectively with $\varepsilon_1=1/4 >0$, and $\varepsilon_2=\varepsilon/2>0$, and that condition~\eqref{eq:sts4} (corresponding to (4.11) in \cite{athreya2020well}) holds with $\Gamma_{4}=0$, we prove that there exists $C>0$ which does not depend on $S,T$ such that for any $(s,t) \in \Delta_{[S,T]}$ and $u=\frac{s+t}{2}$, there is
\begin{enumerate}[label=(\roman*)]
\item\label{en:(1b)}  $\|\EE^{s} \delta A_{{s},u,t}\|_{L^m}\leq C\, \| f \|_{\mathcal{B}_p^{\beta+1}}\, (\psi)_{\mathcal{C}^{\frac{3}{4},0}_{[S,T],x}L^m}\, (t-{s})^{\frac{5}{4}}$;

\myitem{(i$^\prime$)}\label{en:(1'b)} $\|\EE^{s} \delta A_{{s},u,t}\|_{L^m} \leq C\, \|f\|_{\mathcal{B}_p^{\beta}}\, (\psi)_{\mathcal{C}^{\frac{3}{4},0}_{[S,T],x} L^m}\, (t-{s})$ ;

\item\label{en:(2b)} $\| \delta A_{{s},u,t}\|_{L^m} \leq C\, \| f \|_{\mathcal{B}_p^{\beta}}  \Big( 1+  (\psi)_{\mathcal{C}^{\frac{3}{4},0}_{[S,T],x} L^m} \Big) (t-s)^{\frac{1}{2}+ \frac{\varepsilon}{2}}$; 

\item\label{en:(3b)}  If \ref{en:(1b)} and \ref{en:(2b)} are satisfied, \eqref{eq:convAt} gives the 
convergence in probability of $\sum_{i=1}^{N_k-1} A_{t^k_i,t^k_{i+1}}$ on any sequence $\Pi_k=\{t_i^k\}_{i=1}^{N_k}$ of partitions of $[S,t]$ with a mesh that converges to $0$. We will show that the limit is $\mathcal{A}$, the process given in \eqref{eq:A}.
\end{enumerate}
If \ref{en:(1b)}, \ref{en:(1'b)}, \ref{en:(2b)} and \ref{en:(3b)} hold, then the Stochastic Sewing Lemma with critical exponent, whose conclusion is recalled in Remark~\ref{rmk:critical-sewing}, gives that
\begin{align*}
    \Big\| \int_{s}^{t} \int_{\T} p_{T-r}(x,y) f(O_r(y)+\psi_r(y)) \, dy dr \Big\|_{L^m} 
    &\leq C\, \|f\|_{\mathcal{B}_p^{\beta}}\, (\psi)_{\mathcal{C}^{\frac{3}{4},0}_{[S,T],x}L^m} \Big(1+ \Big|\log\frac{\| f \|_{\mathcal{B}_p^{\beta+1}} t^{\varepsilon_1}}{\| f \|_{\mathcal{B}_p^{\beta}}}\Big| \Big)\, (t-{s})  \\ 
    & \quad  + C\, \| f \|_{\mathcal{B}_p^{\beta}}  \Big( 1+  (\psi)_{\mathcal{C}^{\frac{3}{4},0}_{[S,T],x} L^m} \Big) (t-s)^{\frac{1}{2} + \frac{\varepsilon}{2}} \\  
    & \quad + \| A_{{s},t} \|_{L^m} .
\end{align*}
To bound $\| A_{{s},t} \|_{L^m}$, we apply Lemma \ref{lem:6.1athreya} (originally Lemma 6.1 in \cite{athreya2020well}) for $h(\cdot,r,y) = {f(\cdot + \EE^{s} \psi_{r} (y))}$.
Recall that $f$ is smooth and bounded, and by Lemma~\ref{lem:besov-spaces}$(i)$, $ \|f(\cdot +\EE^{s} \psi_{r}(y)) \|_{\mathcal{B}^{\beta}_{p}} \leq \| f\|_{\mathcal{B}^{\beta}_{p}}$, hence the assumptions of Lemma \ref{lem:6.1athreya} are verified. It follows that
\begin{align*}%
\| A_{{s},t} \|_{L^m}
&\leq C\, \|f\|_{\mathcal{B}^{\beta}_p} \, (t-s)^{1+\frac{1}{4}(\beta-\frac{1}{p})} = C\, \|f\|_{\mathcal{B}^{\beta}_p} \, (t-s)^{\frac{1}{2}}.
\end{align*}

~

Let us now verify that \ref{en:(1b)}, \ref{en:(1'b)}, \ref{en:(2b)} and \ref{en:(3b)} are satisfied.

\paragraph{Proof of \ref{en:(1b)} and \ref{en:(1'b)}.} For 
$S\leq s < u < t \leq T$, there is
    $$\delta A_{{s},u,t}= \int_u^{t} \int_{\T} p_{T-r}(x,y) \Big(f(O_{r}(y)+\EE^{s} \psi_{r}(y) )-f(O_{r}(y)+ \EE^{u} \psi_{r}(y)) \Big) \, dy dr.$$ 
Hence the tower property and Fubini's 
Theorem yield
\begin{align*}
    |\EE^{s} \delta A_{{s},u,t}| & =
    \Big|\EE^{s} \int_u^{t} \int_{\T} p_{T-r}(x,y) \EE^u [f(O_{r}(y)+\EE^{s} \psi_{r}(y))-f(O_{r}(y)+\EE^{u} \psi_{r}(y))] \, dy dr \Big| .
    \end{align*}
Consider the $\mathscr{B}(\R)\otimes \mathcal{F}_{u}$-measurable random function $F_{s,u,r,y}(\cdot) := f(\cdot+\EE^{s} \psi_{r}(y))-f(\cdot+\EE^{u} \psi_{r}(y))$. Then in view of Lemma~\ref{lem:reg-O}, one has for $\lambda\in [0,1]$ that
   \begin{align*}
      | \EE^{s} \delta A_{{s},u,t} |
      & \leq  \int_u^{t} \int_{\T} p_{T-r}(x,y) \, \Big| \EE^s \Big[ \| f(\cdot+ \EE^{s} \psi_r(y))- f(\cdot+\EE^{u} \psi_{r}(y)) \|_{\mathcal{B}^{\beta-\lambda}_{p}}\Big] \Big |\,  (r-u)^{\frac{1}{4}(\beta-\lambda-\frac{1}{p})} \, dy dr .
 \end{align*}
   It follows by applying
    Lemma \ref{lem:besov-spaces}$(ii)$ that 
   \begin{align*}
      \| \EE^{s} \delta A_{{s},u,t} \|_{L^m}
      & \leq C \|f\|_{\mathcal{B}_p^{\beta-\lambda+1}} \int_u^{t} \int_{\T} p_{T-r}(x,y) \| \mathbb{E}^{s}[|\EE^{s} \psi_{r}(y) - \EE^u \psi_r(y) |]\|_{L^m}  (r-u)^{\frac{1}{4}(\beta-\lambda-\frac{1}{p})} \, dy dr.
 \end{align*}
By the conditional Jensen's inequality, $\| \EE^{s} X \|_{L^m} \leq \| X \|_{L^m}$ for any integrable random variable $X$. Thus, using $\beta-1/p=-2$, we get
   \begin{align*}
      \| \EE^{s} \delta A_{{s},u,t} \|_{L^m} 
      & \leq C \|f\|_{\mathcal{B}_p^{\beta-\lambda+1}} \int_u^t \int_{\T} p_{T-r}(x,y) \| \EE^{s} \psi_{r}(y) - \EE^u \psi_r(y) \|_{L^m}  (r-u)^{\frac{1}{4}(\beta-\lambda-\frac{1}{p})} \, dy dr \\
      & \leq  C \|f\|_{\mathcal{B}_p^{\beta-\lambda+1}} \ (\psi)_{\mathcal{C}^{\frac{3}{4},0}_{[S,T],x} L^{m}}  (t-s)^{\frac{5}{4}-\frac{\lambda}{4}} .
 \end{align*}
Choosing $\lambda=0$ in the previous inequality yields \ref{en:(1b)}, while choosing $\lambda=0$ yields \ref{en:(1'b)}.

\paragraph{Proof of \ref{en:(2b)}.} Let $\varepsilon\in (0,1)$.
We aim to apply Lemma \ref{lem:6.1athreya} in the interval $[u,T]$ with $\beta-\varepsilon$ and $h(\cdot,r,y) ={f(\cdot + \EE^{s}\psi_{r}(y) )-f(\cdot + \EE^u \psi_{r}(y))}$. Since $f$ is smooth and bounded, Lemma~\ref{lem:besov-spaces}$(i)$ yields
$$ {\|f(\cdot +\EE^{s}\psi_{r}(y))-f(\cdot + \EE^u \psi_{r}(y)) \|_{\mathcal{B}^{\beta-\varepsilon}_{p}}} \leq 2\| f\|_{\mathcal{B}^{\beta-\varepsilon}_{p}}.$$
Hence the assumptions of Lemma \ref{lem:6.1athreya} are verified. It follows that
\begin{align*}
\| \delta A_{s,u,t} \|_{L^m}  
& \leq C\, \sup_{\substack{r \in [s,t] \\ y \in [0,1]}} \| \|f(\cdot +\EE^{s}\psi_{r}(y))-f(\cdot + \EE^u \psi_{r}(y)) \|_{\mathcal{B}^{\beta-\varepsilon}_{p}} \|_{L^m} (t-u)^{1+\frac{1}{4}(\beta-\varepsilon-\frac{1}{p})} \\
& \leq C\, \| f \|_{\mathcal{B}_p^{\beta}} \, \sup_{\substack{r \in [s,t] \\ y \in [0,1]}} \| | \EE^{s}\psi_{r}(y)- \EE^u \psi_{r}(y) |^{\varepsilon} \|_{L^m}  (t-u)^{\frac{1}{2}-\frac{\varepsilon}{4}} ,
\end{align*}
where the last inequality comes from Lemma~\ref{lem:besov-spaces}$(ii)$. Hence by Jensen's inequality,%
\begin{align*}
\| \delta A_{s,u,t} \|_{L^m}   
& \leq C\, \| f \|_{\mathcal{B}_p^{\beta}} \, \sup_{\substack{r \in [s,t] \\ y \in [0,1]}} \| \EE^{s}\psi_{r}(y)- \EE^u \psi_{r}(y) \|_{L^m}^\varepsilon (t-u)^{\frac{1}{2}-\frac{\varepsilon}{4}} \\
& \leq  C\, \| f \|_{\mathcal{B}_p^{\beta}} \,  (\psi)_{\mathcal{C}^{\frac{3}{4},0}_{[S,T],x} L^{m}}^\varepsilon (t-u)^{\frac{1}{2}+\frac{\varepsilon}{2}} \\
 & \leq C\, \| f \|_{\mathcal{B}_p^{\beta}}  \Big( 1+  (\psi)_{\mathcal{C}^{\frac{3}{4},0}_{[S,T],x} L^{m}} \Big) (t-s)^{\frac{1}{2} + \frac{\varepsilon}{2}} . 
\end{align*}
\paragraph{Proof of \ref{en:(3b)}.} Let $(\Pi_k)_{k \in 
\mathbb{N}}$ be  partitions of $[S,t]$ with $\Pi_k=\{t_i^k\}_{i=1}^{N_k}$ and mesh size that 
converges to zero. Then there is
\begin{align*}
    \Big\|\mathcal{A}_{t}-\sum_{i=1}^{N_k-1} A_{t_i^k,t_{i+1}^k}\Big\|_{L^m}
    &\leq \sum_{i=1}^{N_{k}-1} \int_{t_i^k}^{t_{i+1}^k} \int_{\T} p_{T-r}(x,y) \| f(O_r(y) + \psi_r(y))-f(O_r(y)+\EE^{t_i^k} \psi_{r}(y))\|_{L^m} \, dy dr\\
    &\leq \sum_{i=1}^{N_{k}-1} \int_{t_i^k}^{t_{i+1}^k} \|f\|_{\mathcal{C}^1} \|\psi_r(y)-\EE^{t_i^k} \psi_{r}(y) \|_{L^m} \, dy dr.
\end{align*}
Hence in view of \eqref{def:holder-norm-phi}, $\|\psi_r(y)-\EE^{t_i^k} \psi_{r}(y) \|_{L^m} \leq (r-t_{i}^k)^{3/4} (\psi)_{\mathcal{C}^{3/4,0}_{[S,T],x}L^m}$, where the last term was assumed to be finite at the beginning of the proof. Hence $\|\mathcal{A}_{t}-\sum_{i=1}^{N_k-1} A_{t_i^k,t_{i+1}^k}\|_{L^m} \underset{k \rightarrow \infty}{\longrightarrow} 0$.
\end{proof}

\subsection{H\"older regularity of $V^k$}\label{sec:reg-V}

In this section, the regularisation result of Section~\ref{sec:reg-O-1} is used to provide a bound on the H\"older semi-norm of $V^k$ defined in \eqref{def:Vk}.

\begin{corollary}\label{cor:bound-Vk}
Let $m \in[2, \infty)$ and $p \in [m+\infty]$.

\begin{enumerate}[label=(\alph*)]
 \item\label{item:3.5(a)-} \underline{The sub-critical case}:
let $\gamma \in \left(-1, 0\right)$ such that $\gamma-1 / p>-1$. There exists  $C>0$ such that for any $(s, t) \in \Delta_{[0,1]}$ and $k \in \mathbb{N}$, there is
\begin{align*}
 \left[V^k \right]_{\mathcal{C}_{[s, t],x}^{\frac{1}{2},0} L^m} \leq C\left\|b-b^k\right\|_{\mathcal{B}_p^{\gamma-1}} .
\end{align*}

\item\label{item:3.5(b)-} \underline{The limit case}: let $\gamma-1 / p=-1$ and $p<+\infty$. There is $C>0$ such that for any $(s, t) \in \Delta_{[0,1]}$ and $k \in \mathbb{N}$, there is
\begin{align*}
\left[V^k\right]_{\mathcal{C}_{[s, t],x}^{\frac{1}{2},0} L^m} \leq C\left\|b-b^k \right\|_{\mathcal{B}_p^{\gamma-1}}\left(1+\left|\log \left(\left\|b-b^k\right\|_{\mathcal{B}_p^{\gamma-1}}\right)\right|\right) . 
\end{align*}
\end{enumerate}
\end{corollary}

\begin{proof}
The proof follows by applying Proposition \ref{prop:drift-approx} with $f=b^j-b^k$, $\beta=\gamma-1$ and $\psi=v + P \psi_0$. It goes along the same lines as the paragraph ``Bound on $K-K^n$'' in Section 3.2 of \cite{GHR2023}, where one can formally replace $K-K^n$ in \cite{GHR2023} by $V^k$ here. 
\end{proof}

\subsection{Regularisation properties of the Ornstein-Uhlenbeck process on increments}\label{sec:reg-O-2}

In this section, regularisation properties of the process $O$ are proven on increment quantities of the form 
\begin{align}\label{eq:reg-prop-lip}
\int_0^t P_{t-r}\left( f(\phi_r+O_r)- f(\psi_r+O_r) \right) dr,
\end{align} 
which will be bounded by norms of $\phi-\psi$ and the appropriate Besov norm of $f$.

\begin{lemma}\label{lem:ssl-o-on} 
Let $m \in [2,\infty)$, $p\in [m,+\infty]$ and $\beta \in (-2, 0)$. There exists $C>0$ such that for any $0 \leq S < T\leq 1$ and any $(s,t) \in \Delta_{[S,T]}$,  for any $\phi, \psi : [0,1] \times \T \times \Omega \rightarrow \mathbb{R}$ such that for all $r \in [0,1]$ and $x \in \mathbb{T}$, $\phi_r(x)$ and $\psi_r(x)$ are $\mathcal{F}_S$-measurable, and for any $f  \in \mathcal{C}_b^\infty (\R, \R) \cap \mathcal{B}_p^{\beta+1}$, there is
\begin{align*}%
& \sup_{x \in \mathbb{T}} \left\| \int_s^t \int_{\T} p_{T-r}(x,y) \left( f(\psi_r(y)+O_r(y)) - f(\phi_r(y) + O_r(y)) \right) dy dr \right\|_{L^m} \\
&\quad  \leq C \| f \|_{\mathcal{B}_p^{\beta+1}} \  \|\psi-\phi\|_{L^{\infty,\infty}_{[S,T],x} L^m}  (t-s)^{1+ \frac{1}{4} (\beta-\frac{1}{p})} .
\end{align*}
\end{lemma}

\begin{proof}
Let $y \in \mathbb{T}$ and $0 \leq S \leq T \leq 1$. We aim to apply
 Lemma \ref{lem:6.1athreya} with $h(O_r(y),r,y)= f(\psi_r(y)+O_r(y))-f(\phi_r(y)+O_r(y))$, $X_r(y) = p_{T-r}(x,y)$. 
By Lemma~\ref{lem:besov-spaces}$(ii)$, one has
 \begin{align*}
\| \| h(\cdot,r,y) \|_{\mathcal{B}_p^{\beta}} \|_{L^m} & \leq \| f \|_{\mathcal{B}_p^{\beta+1}} \| \psi_r(y)-\phi_r(y) \|_{L^m} \\
&\leq \| f \|_{\mathcal{B}_p^{\beta+1}}  \|\psi-\phi\|_{L^{\infty,\infty}_{[S,T],x} L^m} .
\end{align*}
Thus the assumptions of Lemma \ref{lem:6.1athreya} are verified, which gives the expected result.
\end{proof}

The following lemma will be useful in the proof of Proposition \ref{prop:ssl-o-on-2}, which extends Lemma \ref{lem:ssl-o-on} to the case where the processes $\phi$ and $\psi$ are adapted.

\begin{lemma}\label{cor:ssl-4}
Let $\beta \in\left(-2, 0\right)$, $m \in[2, \infty)$ and $p \in [m,+\infty]$. Let $\lambda, \lambda_1, \lambda_2 \in(0,1]$ be such that $\beta>\max (-2+\lambda,-2+\lambda_1+\lambda_2)$. 
There exists $C>0$ such that
for any $f \in \mathcal{C}_b^{\infty}\left(\mathbb{R}, \mathbb{R}\right) \cap \mathcal{B}_p^\beta$, any $0 \leq s \leq t \leq T\leq 1$, $u=\frac{t+s}{2}$, and any $\mathbb{F}$-adapted processes $\phi, \psi : [0,1] \times \T \times \Omega \rightarrow \mathbb{R}$, there is
\begin{align*}
&\Big\| \int_u^t \int_{\T} p_{T-r}(x,y) \\
&\quad \times\Big(f(\EE^u \psi_r(y)+ O_r(y) )-f(\EE^u \phi_r(y)+O_r(y) )-f(\EE^s \psi_r(y)+ O_r(y))  +f(\EE^s \phi_r(y)+O_r(y))\Big) dydr \Big\|_{L^m} \\
&\quad \leq  C \| f \|_{\mathcal{B}_p^\beta} \, (\psi)_{\mathcal{C}^{\frac{3}{4},0}_{[s,t],x} L^{m,\infty}}^{ \lambda_2} (t-s)^{\frac{3  \lambda_2}{4}} \sup_{\substack{y \in \T\\ r \in [u,T] }}   \| \EE^u \psi_r(y) - \EE^u \phi_r(y) \|_{L^m}^{\lambda_1}  (t-u)^{1+\frac{1}{4}(\beta-\lambda_1-\lambda_2-\frac{1}{p})} \\ 
&\quad  \quad +C  \| f \|_{\mathcal{B}_p^\beta} \sup_{\substack{y \in \T \\ r \in [u,T]}} \| \EE^s \psi_r(y)+\EE^u \phi_r(y)-\EE^u \psi_r(y)-\EE^s \phi_r(y) \|_{L^m}  (t-u)^{1+\frac{1}{4}(\beta-\lambda-\frac{1}{p})} .
\end{align*}

\end{lemma}
\begin{proof}
Let $x,y \in \mathbb{T}$ and $0 \leq s \leq t \leq T \leq 1$. Decompose 
\begin{align*}
&f(\EE^u \psi_r(y)+ O_r(y) )-f(\EE^u \phi_r(y)+O_r(y) )-f(\EE^s \psi_r(y)+ O_r(y))  +f(\EE^s \phi_r(y)+O_r(y)) \\
&\quad = h^1(O_r(y),r,y) - h^2(O_r(y),r,y),
\end{align*}
where 
\begin{align*}
h^1(z,r,y) & := -f\left(\EE^s \psi_r(y)+ z \right)+f\left(\EE^s \phi_r(y)+z \right)+f\left(\EE^u \psi_r(y)+ z \right)  +f\left(\EE^u \psi_r(y)+\EE^s \phi_r(y)-\EE^s \psi_r(y)+z \right),\\
h^2(z,r,y) & := -f\left(\EE^u \psi_r(y)+\EE^s \phi_r(y)-\EE^s \psi_r(y)+z \right) + f\left(\EE^u \phi_r(y)+ z \right).
\end{align*}
We aim to apply again Lemma \ref{lem:6.1athreya} twice in the interval $[u,T]$. First with 
$h^1(O_r(y),r,y)$, and $\beta-\lambda_1-\lambda_2$. Then with $h^2(O_r(y),r,y)$
and $\beta-\lambda$. For the bound on the moment of $\| h^1(\cdot,r,y) \|_{\mathcal{B}_p^{\beta-\lambda_1-\lambda_2}}$,  use Lemma~\ref{lem:besov-spaces}$(iii)$ and Jensen's inequality to get
\begin{align*}
\EE \| h^1(\cdot,r,y) \|_{\mathcal{B}_p^{\beta-\lambda_1-\lambda_2}}^m & =\EE \EE^s \| h^1(\cdot,r,y) \|_{\mathcal{B}_p^{\beta-\lambda_1-\lambda_2}}^m  \\
& \leq \| f \|_{\mathcal{B}_p^\beta}^m\, \EE \Big[ | \EE^s \psi_r(y) - \EE^s \phi_r(y) |^{m \lambda_1}  \EE^s |\EE^u \psi_r(y)-\EE^s \psi_r(y)|^{m\lambda_2} \Big] \\
& \leq \| f \|_{\mathcal{B}_p^\beta}^m\, \| \EE^s | \EE^u \psi_r(y)-\EE^s \psi_r(y)|^{m \lambda_2} \|_{L^\infty} \EE \Big[  | \EE^s \psi_r(y) - \EE^s \phi_r(y)|^{m \lambda_1}  \Big] \\
& \leq  \| f \|_{\mathcal{B}_p^\beta}^m\, \sup_{y \in \T} \left\| \big( \EE^s | \EE^u \psi_r(y) - \EE^s  \psi_r(y) |^m \big)^{\frac{1}{m}} \right\|_{L^\infty}^{m \lambda_2} \sup_{\substack{y \in \T\\ r \in [u,T] }}  \EE \Big[  | \EE^s \psi_r(y) - \EE^s \phi_r(y) |^{m} \Big]^{\lambda_1} \\
& \leq \| f \|_{\mathcal{B}_p^\beta}^m\, (\psi)_{\mathcal{C}^{\frac{3}{4},0}_{[s,t],x} L^{m,\infty}}^{m \lambda_2}  (t-s)^{\frac{3 m \lambda_2}{4}} \sup_{\substack{y \in \T\\ r \in [u,T] }}  \EE \Big[  | \EE^s \psi_r(y) - \EE^s \phi_r(y) |^{m} \Big]^{\lambda_1} . 
\end{align*}
So Lemma \ref{lem:6.1athreya} yields
\begin{align}\label{eq:boundh1}
\Big\|& \int_u^t \int_{\T} p_{T-r}(x,y) h^{1}(O_{r}(y),r,y) \, dydr \Big\|_{L^m} \nonumber\\
&\leq  C \| f \|_{\mathcal{B}_p^\beta}  (\psi)_{\mathcal{C}^{\frac{3}{4},0}_{[s,t],x} L^{m,\infty}}^{ \lambda_2} (t-s)^{\frac{3  \lambda_2}{4}} \sup_{\substack{y \in \T\\ r \in [u,T] }}   \| \EE^s \psi_r(y) - \EE^s \phi_r(y) \|_{L^m}^{\lambda_1}  (t-u)^{1+\frac{1}{4}(\beta-\lambda_1-\lambda_2)} .
\end{align}

Similarly, using Lemma~\ref{lem:besov-spaces}$(ii)$, one gets
\begin{align*}
\| \| h^2(\cdot,r,y) \|_{\mathcal{B}_p^{\beta-\lambda}} \|_{L^m} 
& \leq \| f \|_{\mathcal{B}_p^\beta} \sup_{\substack{y \in \T \\ r \in [u,T]}} \| \EE^s \psi_r(y)+\EE^u \phi_r(y)-\EE^u \psi_r(y)-\EE^s \phi_r(y)\|_{L^m} 
\end{align*}
and applying Lemma \ref{lem:6.1athreya},
\begin{align}\label{eq:boundh2}
\Big\| \int_u^t \int_{\T}& p_{T-r}(x,y) h^{2}(O_{r}(y),r,y) \, dydr \Big\|_{L^m} \nonumber\\
&\leq  C  \| f \|_{\mathcal{B}_p^\beta} \sup_{\substack{y \in \T \\ r \in [u,T]}} \| \EE^s \psi_r(y)+\EE^u \phi_r(y)-\EE^u \psi_r(y)-\EE^s \phi_r(y) \|_{L^m}  (t-u)^{1+\frac{1}{4}(\beta-\lambda-\frac{1}{p})} .
\end{align}
The result follows by summing \eqref{eq:boundh1} and \eqref{eq:boundh2}.
\end{proof}

Finally, we state a general regularisation property for the quantity \eqref{eq:reg-prop-lip}, which will be used in Section~\ref{sec:reg-E1} to bound the error term $\mathcal{E}^{1,n,k}$ in the sub-critical case.

\begin{prop}\label{prop:ssl-o-on-2}
Let $m \in [2,\infty)$, and assume that $p\in [m,\infty]$ and $\gamma-1/p > -1$. There exists a constant $C$ such that for any $f  \in \mathcal{C}_b^\infty (\R, \R) \cap \mathcal{B}_p^\gamma$, any $\mathbb{F}$-adapted processes $\phi, \psi : [0,1] \times \T \times \Omega \rightarrow \mathbb{R}$, any $0 \leq S < T \leq 1$ and any $(s,t) \in \Delta_{[S,T]}$, we have 
\begin{align}\label{eq:ssl-o-on-2}
\begin{split}
& \sup_{x \in \mathbb{T}} \left\| \int_s^t \int_{\T} p_{T-r}(x,y) \left( f(\psi_r(y)+O_r(y)) - f(\phi_r(y) + O_r(y)) \right) dy dr \right\|_{L^m}  \\
&\quad  \leq C  \| f \|_{\mathcal{B}_p^\gamma}  ((\psi)_{\mathcal{C}^{\frac{3}{4},0}_{[S,T],x} L^{m,\infty}}+1) \Big( \| \psi - \phi \|_{L^{\infty,\infty}_{[S,T],x} L^m} + \big( \psi-\phi \big)_{\mathcal{C}^{\frac{1}{2},0}_{[S,T],x} L^m}  \Big) (t-s)^{\frac{3}{4}+\frac{1}{4}(\gamma -\frac{1}{p})}   \\ 
& \quad\quad + C  \| f \|_{\mathcal{B}_p^\gamma} \Big( \big( \psi-\phi \big)_{\mathcal{C}^{\frac{1}{2},0}_{[S,T],x} L^m}   + (\psi)_{\mathcal{C}^{\frac{3}{4},0}_{[S,T],x} L^{m,\infty}}  \| \psi - \phi \|_{L^{\infty,\infty}_{[S,T],x} L^m}  \Big) (t-s)^{\frac{5}{4}+\frac{1}{4}(\gamma-\frac{1}{p})} .
\end{split}
\end{align}
\end{prop}

\begin{proof}
Fix $x \in \mathbb{T}$ and $S \leq T$. Define for $(s,t) \in \Delta_{[S,T]}$,
\begin{align*}
A_{s,t} = \int_s^t \int_{\T} p_{T-r}(x,y) \left( f(\EE^s \psi_r(y) + O_r(y)) - f(\EE^s \phi_r(y) + O_r(y)) \right) dy dr ,
\end{align*}
and
\begin{align}\label{eq:defA-4.5}
\mathcal{A}_t = \int_S^t \int_{\T} p_{T-r}(x,y) \left( f(\psi_r(y)+ O_r(y)) - f(\phi_r(y) + O_r(y)) \right) dy dr.
\end{align}
Assume without loss of generality that the quantities that appear in the right-hand side of \eqref{eq:ssl-o-on-2} are all finite, otherwise the result is obvious. We check the assumptions in order to apply the Stochastic Sewing Lemma (Lemma \ref{lem:SSL} with $q=m$). To prove that \eqref{eq:condsew1} and \eqref{eq:condsew2} hold true with $\varepsilon_1=\varepsilon_2=\frac{1}{4}+\frac{1}{4}(\gamma-1/p) >0$, we show that there exists $C>0$ which does not depend on $s,t,S$ and $T$ such that
\begin{enumerate}[label=(\roman*)]
\item\label{en:(1b-)}  
$\| \delta A_{s,u,t} \|_{L^m} \leq  C  \| f \|_{\mathcal{B}_p^\gamma}  ((\psi)_{\mathcal{C}^{\frac{3}{4},0}_{[S,T],x} L^{m,\infty}}+1) \Big( \| \psi - \phi \|_{L^{\infty,\infty}_{[S,T],x} L^m} + \big( \psi-\phi \big)_{\mathcal{C}^{\frac{1}{2},0}_{[S,T],x} L^m}  \Big) (t-s)^{\frac{3}{4}+\frac{1}{4}(\gamma-\frac{1}{p})}$;

\item\label{en:(2b-)} $\| \EE^s \delta A_{s,u,t} \|_{L^m} \leq C  \| f \|_{\mathcal{B}_p^\gamma} \Big(  \big( \psi-\phi \big)_{\mathcal{C}^{\frac{1}{2},0}_{[S,T],x} L^m} + (\psi)_{\mathcal{C}^{\frac{3}{4},0}_{[S,T],x} L^{m,\infty}}  \| \psi - \phi \|_{L^{\infty,\infty}_{[S,T],x} L^m}  \Big) (t-s)^{\frac{5}{4}+\frac{1}{4}(\gamma-\frac{1}{p})}$; 

\item\label{en:(3b-)}  If \ref{en:(1b-)} and \ref{en:(2b-)} are satisfied, \eqref{eq:convAt} gives the 
convergence in probability of $\sum_{i=1}^{N_k-1} A_{t^k_i,t^k_{i+1}}$ on any sequence $\Pi_k=\{t_i^k\}_{i=1}^{N_k}$ of partitions of $[S,t]$ with a mesh that converges to $0$. We will show 
that the limit is $\mathcal{A}$, as given in \eqref{eq:defA-4.5}.
\end{enumerate}

If \ref{en:(1b-)}, \ref{en:(2b-)} and \ref{en:(3b-)} are satisfied, then Lemma \ref{lem:SSL} implies that
\begin{align*}
& \Big\| \int_s^t \int_{\T} p_{T-r}(x,y) \left( f(\psi_r(y)+O_r(y)) - f(\phi_r(y) + O_r(y)) \right) dy dr \Big\|_{L^m}  \\
& \leq C  \| f \|_{\mathcal{B}_p^\gamma}  ((\psi)_{\mathcal{C}^{\frac{3}{4},0}_{[S,T],x} L^{m,\infty}}+1) \Big( \| \psi - \phi \|_{L^{\infty,\infty}_{[S,T],x} L^m}       +   \big( \psi-\phi \big)_{\mathcal{C}^{\frac{1}{2},0}_{[S,T],x} L^m}  \Big) (t-s)^{\frac{3}{4}+\frac{1}{4}(\gamma-\frac{1}{p})} \\ 
& \quad + C \| f \|_{\mathcal{B}_p^\gamma} \Big( ( \psi-\phi )_{\mathcal{C}^{\frac{1}{2},0}_{[S,T],x} L^m} + (\psi)_{\mathcal{C}^{\frac{3}{4},0}_{[S,T],x} L^{m,\infty}}  \| \psi - \phi \|_{L^{\infty,\infty}_{[S,T],x} L^m}  \Big) (t-s)^{\frac{5}{4}+\frac{1}{4}(\gamma-\frac{1}{p})}  + \| A_{s,t} \|_{L^m}.
\end{align*}
To bound $\| A_{{s},{t}} \|_{L^m}$, apply Lemma~\ref{lem:ssl-o-on} for $\EE^s \psi_r $ and $ \EE^s \phi_r$ in place of $\psi_r$ and $\phi_r$, to get that
\begin{align*}
\| A_{s,t} \|_{L^m} & \leq C \| f \|_{\mathcal{B}_p^\gamma}  \|\EE^s \psi_\cdot - \EE^s \phi_\cdot \|_{L^{\infty,\infty}_{[S,T],x} L^m}  (t-s)^{\frac{3}{4}+\frac{1}{4}(\gamma-\frac{1}{p})} \\
&\leq C \| f \|_{\mathcal{B}_p^\gamma} \| \psi - \phi \|_{L^{\infty,\infty}_{[S,T],x} L^m}  (t-s)^{\frac{3}{4}+\frac{1}{4}(\gamma-\frac{1}{p})} ,
\end{align*}
and taking the supremum over $x \in \T$ finally yields \eqref{eq:ssl-o-on-2}.

~

Let us now verify that  \ref{en:(1b-)}, \ref{en:(2b-)} and \ref{en:(3b-)} are satisfied.
\paragraph{Proof of \ref{en:(1b-)}:}
For $u = \frac{s+t}{2}$, there is
\begin{align*}
\| \delta A_{s,u,t} \|_{L^m} & = \Big\| \int_u^t \int_{\T}  p_{T-r}(x,y) \Big( f( \EE^s \psi_r(y) + O_r(y)) - f(\EE^s \phi_r(y)  + O_r(y)) \\ 
&\quad - f(\EE^u \psi_r(y) +O_r(y)) + f(\EE^u \phi_r(y) +O_r (y)) \Big)\, dy dr \Big\|_{L^m}.
\end{align*}
Let $\varepsilon>0$. By the Lemma \ref{cor:ssl-4} applied with $\lambda_1=1,~\lambda_2=\varepsilon$, one has
\begin{align*}
\| \delta A_{s,u,t} \|_{L^m}  
& \leq C \| f \|_{\mathcal{B}_p^\gamma}  (\psi)_{\mathcal{C}^{\frac{3}{4},0}_{[s,t],x} L^{m,\infty}}^{ \varepsilon} (t-s)^{\frac{3  \varepsilon}{4}} \sup_{\substack{y \in \T\\ r \in [u,T] }}   \|  \EE^u \psi_r(y) - \EE^u \phi_r(y) \|_{L^m}  (t-u)^{1+\frac{1}{4}(\gamma-1-\varepsilon-\frac{1}{p})}  \\ 
& \quad +C  \| f \|_{\mathcal{B}_p^\gamma} \sup_{\substack{y \in \T \\ r \in [u,T]}} \frac{\| \EE^u( \psi_r(y) -\phi_r(y)) -\EE^s (\psi_r(y) - \phi_r(y)) \|_{L^m}}{|r-s|^{\frac{1}{2}}} (u-s)^{\frac{1}{2}} (t-u)^{1+\frac{1}{4}(\gamma-1-\frac{1}{p})}  \\
& \leq C  \| f \|_{\mathcal{B}_p^\gamma}  ((\psi)_{\mathcal{C}^{\frac{3}{4},0}_{[s,t],x} L^{m,\infty}}+1) \Big( \| \psi - \phi \|_{L^{\infty,\infty}_{[S,T],x} L^m} (t-s)^{1+\frac{1}{4}(\gamma-1+2\varepsilon-\frac{1}{p})}   \\ 
& \quad +  \| f \|_{\mathcal{B}_p^\gamma} \big( \psi-\phi \big)_{\mathcal{C}^{\frac{1}{2},0}_{[S,T],x} L^m} (t-s)^{1+\frac{1}{4}(\gamma+1-\frac{1}{p})}  \Big) .
\end{align*}
Therefore,
\begin{align}
\label{eq:bound-delta} 
\| \delta A_{s,u,t} \|_{L^m}  &\leq  C  \| f \|_{\mathcal{B}_p^\gamma}  ((\psi)_{\mathcal{C}^{\frac{3}{4},0}_{[s,t],x} L^{m,\infty}}+1) \Big( \| \psi - \phi \|_{L^{\infty,\infty}_{[S,T],x} L^m}       +  \big( \psi-\phi \big)_{\mathcal{C}^{\frac{1}{2},0}_{[S,T],x} L^m}  \Big) (t-s)^{\frac{3}{4}+\frac{1}{4}(\gamma-\frac{1}{p})} .
\end{align}

\paragraph{Proof of \ref{en:(2b-)}:}
 Let us now bound $\mathbb{E}^s \delta A_{s,u,t}$ for $u=\frac{s+t}{2}$. There is
\begin{align*}
| \mathbb{E}^s \delta A_{s,u,t} | 
& \leq \mathbb{E}^s\int_u^t \int_{\T} p_{T-r}(x,y) \Big| \mathbb{E}^u \Big(f(\EE^s \psi_r(y)+ O_r(y)) - f(\EE^s \phi_r(y) + O_r(y)) \\ 
&\hspace{1cm} - f(\EE^u \psi_r(y)+O_r(y)) + f(\EE^u \phi_r(y) +O_r (y)) \Big)\Big| dy dr\\
& \leq J_1 + J_2 ,
\end{align*}
where 
\begin{align*}
J_1 & = \mathbb{E}^s\int_u^t \int_{\T} p_{T-r}(x,y) \Big| \EE^u f(\EE^s \psi_r(y) + O_r(y)) 
- \EE^u f(\EE^s \phi_r(y)  + P_{r-u}O_u(y)) \\  
& \quad - \EE^u f(\EE^u \psi_r(y)+O_r(y)) 
+ \EE^u f(\EE^s \phi_r(y) +\EE^u \psi_r(y) -\EE^s \psi_r(y) +O_r(y)) \Big| dy dr,
\end{align*}
and
\begin{align*}
J_2 & = \mathbb{E}^s\int_u^t \int_{\T} p_{T-r}(x,y) \Big| \EE^u f(\EE^u \phi_r(y)+O_r(y)) \\ 
& \quad - \EE^u f(\EE^s \phi_r(y) +\EE^u \psi_r(y) -\EE^s \psi_r(y) +O_r(y)) \Big| dy dr.
\end{align*}
Using Lemma \ref{lem:reg-O}, it comes
\begin{align*}
J_1 & \leq \EE^s \int_u^t \int_{\T} p_{T-r}(x,y) \, \Big\| f(\EE^s \psi_r(y) + \cdot) - f(\EE^s \phi_r(y)  + \cdot) - f(\EE^u \psi_r(y)+\cdot) \\ 
& \hspace{1cm} + f(\EE^s \phi_r(y) +\EE^u \psi_r(y) -\EE^s \psi_r(y) +\cdot) \Big\|_{\mathcal{B}_p^{\gamma-2}} (r-u)^{\frac{1}{4}(\gamma-2-\frac{1}{p})} dy dr.
\end{align*}
Moreover, by Lemma \ref{lem:besov-spaces}$(iii)$,
\begin{align*}
 J_1 & \leq C \| f \|_{\mathcal{B}_p^\gamma} \int_u^t \int_{\T} p_{T-r}(x,y) \, \EE^s \Big[  | \EE^s \psi_r(y)  - \EE^s \phi_r(y) |   |\EE^u \psi_r(y) -\EE^s \psi_r(y) | \Big] (r-u)^{\frac{1}{4}(\gamma-2-\frac{1}{p})} dy dr \\
 & \leq C \| f \|_{\mathcal{B}_p^\gamma} \int_u^t \int_{\T} p_{T-r}(x,y) \, (\psi)_{\mathcal{C}^{\frac{3}{4}}_{[S,T]} L^{1,\infty}} (r-s)^{\frac{3}{4}} \, | \EE^s \psi_r(y)  - \EE^s \phi_r(y) | \, (r-u)^{\frac{1}{4}(\gamma-2-\frac{1}{p})} dy dr .
\end{align*}
Taking the $L^m$ norm and using the conditional Jensen's inequality gives
\begin{align}\label{eq:bound-cond-exp-1}
\|  J_1 \|_{L^m} & \leq C  \| f \|_{\mathcal{B}_p^\gamma} (\psi)_{\mathcal{C}^{\frac{3}{4}}_{[S,T]} L^{m,\infty}}  \| \psi - \phi \|_{L^{\infty,\infty}_{[S,T],x} L^m}   (t-s)^{1+\frac{3}{4}+\frac{1}{4}(\gamma-2-\frac{1}{p})}  .
\end{align}
With similar arguments (involving Lemma~\ref{lem:reg-O} and Lemma~\ref{lem:besov-spaces}$(ii)$), we have for $J_2$ that
\begin{align}
| J_2 |
& \leq C  \| f \|_{\mathcal{B}_p^\gamma} \int_u^t \int_{\T} p_{T-r}(x,y) | \EE^u [\psi_r(y)-\phi_r(y)] - \EE^s [\psi_r(y)-\phi_r(y)] |\, (r-u)^{\frac{1}{4}(\gamma-1-\frac{1}{p})} dy dr \nonumber \\ 
\| J_2 \|_{L^m} & \leq C \| f \|_{\mathcal{B}_p^\gamma} \big( \psi-\phi \big)_{\mathcal{C}^{\frac{1}{2},0}_{[S,T],x} L^m}  (t-s)^{1+\frac{1}{2}+\frac{1}{4}(\gamma-1-\frac{1}{p})} .\label{eq:bound-cond-exp-2}
\end{align}
Combining \eqref{eq:bound-cond-exp-1} and \eqref{eq:bound-cond-exp-2} yields \ref{en:(2b-)}.

\paragraph{Proof of \ref{en:(3b-)}:} Finally, for $(\Pi_k)_{k \in 
\mathbb{N}}$ partitions of $[S,t]$ with $\Pi_k=\{t_i^k\}_{i=1}^{N_k}$ and mesh size that 
converges to zero, we have using the $\mathcal{C}^1$ norm of $f$ that

\begin{align*}
\Big\| \mathcal{A}_t - \sum_{i=1}^{N_k-1} A_{t_i^k,t_{i+1}^k} \Big\|_{L^1} 
& \leq \sum_{i=1}^{N_k-1} \int_{t_i^k}^{t_{i+1}^k} \int_{\T} p_{T-r}(x,y) \, \| f \|_{\mathcal{C}^1} \Big( \| \psi_r(y)-\EE^{t_i^k} \psi_r(y) \|_{L^m}\\ 
&\quad + \| \psi_r(y) - \phi_r(y) - \EE^{t_i^k} \psi_r(y) + \EE^{t_i^k} \phi_r(y) \|_{L^m} \Big) dy dr\\
&\leq  C \| f \|_{\mathcal{C}^1} | \Pi_{k} |^{\frac{1}{2}} \Big( (\psi)_{\mathcal{C}^{\frac{3}{4},0}_{[S,T],x} L^m} +\big( \psi-\phi \big)_{\mathcal{C}^{\frac{1}{2},0}_{[S,T],x} L^m} \Big).
\end{align*}
Thus $\mathcal{A}_t$ is the limit in probability of $\sum_{i=1}^{N_k-1} A_{t_i^k,t_{i+1}^k} $ as the mesh size goes to zero. 
\end{proof}

\subsection{H\"older regularity of $\mathcal{E}^{1,n,k}$}\label{sec:reg-E1}

In this section, we prove H\"older bounds on the process $\mathcal{E}^{1,n,k}$ defined in \eqref{def:E1}. First, in the sub-critical case, the regularisation result of Proposition \ref{prop:ssl-o-on-2} is used to quantify the H\"older semi-norm of $\mathcal{E}^{1,n,k}$, see Corollary \ref{cor:bound-E1}. Then in the limit case, a H\"older estimate is established in Proposition \ref{prop:bound-E1-critic} using the critical-exponent version of the Stochastic Sewing Lemma recalled in Remark \ref{rmk:critical-sewing}.

\begin{corollary}\label{cor:bound-E1}
Recall the definition of $\epsilon(n,k)$ in \eqref{eq:defepsilonhn}.
Let $0 \leq S < T \leq 1$, $\gamma >-1$, $m \in [2,+\infty)$ and $p \in [m, +\infty]$ such that $\gamma-1/p \in (-1,0)$. 
There exists $C$ which depends only on $\gamma$, $m$ and $p$ such that for any $n\in \N^*$ and $k \in \mathbb{N}$,
\begin{align*}%
[\mathcal{E}^{1,n,k}]_{\mathcal{C}^{\frac{1}{2},0}_{[S,T],x} L^m}  
& \leq C  \| b \|_{\mathcal{B}_p^\gamma} \Big( \| \mathcal{E}^{n,k}_{0,\cdot} \|_{L^{\infty,\infty}_{[S,T],x} L^m}  + [\mathcal{E}^{n,k}]_{\mathcal{C}^{\frac{1}{2},0}_{[S,T],x} L^m} \nonumber+ \epsilon(n,k) \Big) (T-S)^{\frac{1}{4}+\frac{1}{4}(\gamma-\frac{1}{p})} ,
\end{align*}
where we recall the process $\mathcal{E}^{n,k}$ from \eqref{def:error}.
\end{corollary}

\begin{proof}
Let $0<S\leq T\leq 1$ and $(s,t)\in \Delta_{[S,T]}$. Applying Proposition~\ref{prop:ssl-o-on-2} with $\psi=  P \psi_0+v = u-O$, $\phi = P^n \psi_0 + v^{n,k}$ (for $v^{n,k}$ defined in \eqref{def:v-hk}) and $f=b^k$ reads
\begin{align}\label{eq:boundE1nk}
\sup_{x\in \T} \|\mathcal{E}^{1,n,k}_{s,t}\|_{L^m} \leq C  \| b^k\|_{\mathcal{B}_p^\gamma}  \Big((\psi)_{\mathcal{C}^{\frac{3}{4},0}_{[S,T],x} L^{m,\infty}}+1\Big) \Big( \| \psi - \phi \|_{L^{\infty,\infty}_{[S,T],x} L^m} + \big( \psi-\phi \big)_{\mathcal{C}^{\frac{1}{2},0}_{[S,T],x} L^m}  \Big) (t-s)^{\frac{3}{4}+\frac{1}{4}(\gamma-\frac{1}{p})}.
\end{align}
In view of \eqref{eq:comp-seminorms1}, the term $(\psi)_{\mathcal{C}^{3/4,0}_{[S,T],x} L^{m,\infty}}$ is bounded by $2 [\psi]_{\mathcal{C}^{3/4,0}_{[S,T],x} L^{m,\infty}} = 2[u-O]_{\mathcal{C}^{3/4,0}_{[S,T],x} L^{m,\infty}}$, which is exactly what appears in \eqref{eq:holder-reg-sol} and is thus finite. Moreover, we bound $\| \psi - \phi \|_{L^{\infty,\infty}_{[S,T],x} L^m}$ by
\begin{align*}
 \| \psi -  \phi  \|_{L^{\infty,\infty}_{[S,T],x} L^m} & \leq \| \mathcal{E}^{n,k}_{0,\cdot} \|_{L^{\infty,\infty}_{[S,T],x} L^m} + \sup_{\substack{t \in [0,1] \\x \in \T}} | P_t \psi_0(x) - P^n_t \psi_0(x) | \nonumber \\
 & \leq \| \mathcal{E}^{n,k}_{0,\cdot} \|_{L^{\infty,\infty}_{[S,T],x} L^m} + \epsilon(n,k) ,
\end{align*}
and use \eqref{eq:comp-seminorms3} to get $\big( \psi-\phi \big)_{\mathcal{C}^{1/2,0}_{[S,T],x} L^m} \leq 2 [\mathcal{E}^{n,k}]_{\mathcal{C}^{1/2,0}_{[S,T],x} L^m}$.
Hence, plugging the previous inequalities in \eqref{eq:boundE1nk}, and using that $ \|b^k\|_{\mathcal{B}^\gamma_{p}}\leq  \|b\|_{\mathcal{B}^\gamma_{p}}$, one gets
\begin{align*}
\sup_{x\in \T} \|\mathcal{E}^{1,n,k}_{s,t}\|_{L^m} 
\leq C  \|b\|_{\mathcal{B}_p^\gamma} \Big( \| \mathcal{E}^{n,k}_{0,\cdot} \|_{L^{\infty,\infty}_{[S,T],x} L^m} + \epsilon(n,k)
+  [\mathcal{E}^{n,k}]_{\mathcal{C}^{\frac{1}{2},0}_{[S,T],x} L^m} \Big) (t-s)^{\frac{3}{4}+\frac{1}{4}(\gamma-\frac{1}{p})}.
\end{align*}
Divide now by $(t-s)^{1/2}$ and take the supremum over $(s,t) \in \Delta_{[S,T]}$ to get the desired result.
\end{proof}

We now proceed with the H\"older bound on $\mathcal{E}^{1,n,k}$ in the limit case.
\begin{prop}\label{prop:bound-E1-critic}
Recall that $\{ u^{n,k} \}_{n,k,m, \alpha}$ was defined in \eqref{eq:discrete-seminorm}. Let $m\in [2,+\infty)$, $p\in [m,+\infty)$, $\gamma-1 / p=-1$ and $\eta \in (0,(\gamma+1) \wedge \frac{1}{4})$. 
There are constants $\mathbf{M}>0$ and $\ell_0 > 0$ such that for any $\zeta \in (0,1 / 2)$, any $0 \leq S<T \leq 1$ satisfying $T-S \leq \ell_0$, any $(s,t) \in \Delta_{[S,T]}$, any $n\in \N^*$ and $k \in \mathbb{N}$,
\begin{align*}
\sup_{x \in \mathbb{T}} \|\mathcal{E}^{1,n,k}_{s,t}(x) \|_{L^m} 
& \leq \mathbf{M}\left(1+\Big|\log \frac{T^\eta (1+\big\{ u^{n,k} \big\}_{n,k,m, \frac{1}{2}+\eta}}{\big\| \mathcal{E}^{1,n,k}_{0,\cdot}  \big\|_{L_{[S, T],x}^{\infty,\infty} L^m}+\epsilon(n, k)}\Big|\right)
\left(\big\| \mathcal{E}^{1,n,k}_{0,\cdot} \big\|_{L_{[S, T],x}^{\infty,\infty} L^m}+\epsilon(n, k)\right)(t-s)
\\ & \quad +\mathbf{M}\left(\big\| \mathcal{E}^{1,n,k}_{0,\cdot} \big\|_{L_{[S, T],x}^{\infty,\infty} L^m}+ \left[ \mathcal{E}^{1,n,k} \right]_{\mathcal{C}_{[S, T],x}^{\frac{1}{2}-\zeta,0} L^m} + \epsilon(n, k) \right)(t-s)^{\frac{1}{2}}.
\end{align*}
\end{prop}

\begin{proof} 
Let us denote $\psi=P \psi_0 + v$ and $\phi= P^n \psi_0 + v^{n,k}$.
For any $(s,t)\in \Delta_{[s,t]}$ and $x\in \T$, let 
\begin{equation*} %
\begin{split}
R_{s,t}(x) &=\int_s^{t} \int_{\T} p_{t-r}(x,y) \left(b^k(\psi_r(y)+ O_r(y)) - b^k(\phi_r(y)+ O_r(y))\right) dy dr  \\ 
&\quad - \int_{s}^{t} \int_{\T} p_{t-r}(x,y) \left( b^k(\EE^s \psi_r(y)+ O_r(y)) - b^k(\EE^s \phi_r(y)+ O_r(y))\right) dy dr .
\end{split}
\end{equation*}
Notice that for $\tau>0$ (to be precised later),
\begin{align*}
\| \mathcal{E}^{1,n,k}_{s,t}(x) \|_{L^m} & =  \Big\| \int_{s}^{t} \int_{\T} p_{t-r}(x,y) f( \psi_r(y)+ O_r(y)) - f( \phi_r(y)+  O_r(y)) \, dy dr \Big\|_{L^m}  \\
& \leq \| R_{s,t}(x) \|_{L^m}  + \Big\| \int_{s}^{t} \int_{\T} p_{t-r}(x,y) f(\EE^s \psi_r(y)+ O_r(y)) - f(\EE^s \phi_r(y)+  O_r(y)) \, dy dr \Big\|_{L^m} \\
& \leq  [R]_{\mathcal{C}^{\tau,0}_{[s,t],x} L^{m}} (t-s)^{\tau} +  \Big\| \int_{s}^{t} \int_{\T} p_{t-r}(x,y) f(\EE^s \psi_r(y)+ O_r(y)) - f(\EE^s \phi_r(y)+  O_r(y)) \, dy dr\Big\|_{L^m}.
\end{align*} 
Apply Lemma \ref{lem:ssl-o-on} with $\beta=\gamma-1$ to the second quantity on the right-hand side of the previous inequality to get
\begin{align*}
\| \mathcal{E}^{1,n,k}_{s,t}(x) \|_{L^m}  &\leq  [R]_{\mathcal{C}^{\tau,0}_{[s,t],x} L^{m}} (t-s)^{\tau}  +  C \| f \|_{\mathcal{B}_p^{\gamma}} \,  \|\EE^s \psi- \EE^s \phi\|_{L^{\infty,\infty}_{[s,t],x} L^m}  (t-s)^{\frac{1}{2}} .
\end{align*}
By Jensen's inequality for conditional expectation, one gets $ \|\EE^s \psi- \EE^s \phi\|_{L^{\infty,\infty}_{[s,t],x} L^m} \leq \| \psi- \phi\|_{L^{\infty,\infty}_{[s,t],x} L^m} \leq \| \psi- \phi\|_{L^{\infty,\infty}_{[S,T],x} L^m}$. Moreover
\begin{align*}
\|  \psi -\phi \|_{L_{[S,T],x}^{\infty,\infty} L^m}  = \| v -v^{n,k} + P \psi_0 -P^n \psi_0 \|_{L_{[S,T],x}^{\infty,\infty} L^m} 
  &\leq \| \mathcal{E}^{n,k} \|_{L_{[S,T],x}^{\infty,\infty} L^m} + \sup_{\substack{t \in [0,1] \\ x \in \T}}| P_t \psi_0(x) - P^n_t \psi_0(x) |  \\
& \leq \| \mathcal{E}^{1,n,k} \|_{L_{[S,T],x}^{\infty,\infty} L^m} +\epsilon(n,k) .%
\end{align*}
Hence
\begin{align}\label{eq:E1-R-link}
\| \mathcal{E}^{1,n,k}_{s,t}(x) \|_{L^m} 
& \leq  [R]_{\mathcal{C}^{\tau,0}_{[S,T],x} L^{m}} (t-s)^{\tau} +  C \| f \|_{\mathcal{B}_p^{\gamma}} \left(\| \mathcal{E}^{1,n,k} \|_{L_{[S,T],x}^{\infty,\infty} L^m} +\epsilon(n,k) \right) (t-s)^{\frac{1}{2}} .
\end{align}

~

The goal now is to use stochastic sewing techniques to bound the term $ [R]_{\mathcal{C}^{\tau,0}_{[S,T],x} L^{m}} $. 
Fix $x \in \mathbb{T}$ and $(\tilde{S},\tilde{T}) \in \Delta_{[S,T]}$. For any $(s,t) \in \Delta_{\tilde{S}, \tilde{T}}$, let $A_{s,t}$ and $\mathcal{A}_{t}$ be defined as follows:
\begin{equation} \label{eq:Prop51A-critic}
\begin{split}
A_{s,t} &= \int_{s}^{t} \int_{\T} p_{\tilde{T}-r}(x,y) \left( b^k(\EE^s \psi_r(y)+ O_r(y)) - b^k(\EE^s \phi_r(y)+  O_r(y)) \right) dy dr ,\\
\mathcal{A}_{t} &= \int_{\tilde{S}}^{t} \int_{\T} p_{\tilde{T}-r}(x,y) \left( b^k(\psi_r(y)+ O_r(y)) - b^k( \phi_r(y)+ O_r(y))\right) dy dr  .
\end{split}
\end{equation}
Let $\eta \in (0,(\gamma+1) \wedge \frac{1}{4})$ and set $\tau=1/2+\eta/2 \wedge (1/2-\zeta)$. In the following, we aim to apply the Stochastic Sewing Lemma with critical exponent \cite[Theorem 4.5]{athreya2020well} (see Remark \ref{rmk:critical-sewing}). To show that \eqref{eq:condsew1}, \eqref{eq:condsew2} and \eqref{eq:sts4} hold true with $\varepsilon_{1} =\eta$, $\varepsilon_2 = \eta/2 \wedge (1/2-\zeta)>0$ and $\varepsilon_4 = \eta/2 \wedge (1/2-\zeta)$, we prove that there exists $C>0$ which does not depend on $s,t,S, \tilde{S}, \tilde{T}$ and $T$ such that for $u = (s+t)/2$,
\begin{enumerate}[label=(\roman*)]
\item \label{item47(1)} 
$\|\EE^{s} \delta A_{{s},u,{t}} \|_{L^m} \leq  C\, ( 1+ \big\{ u^{n,k} \big\}_{n,k,m,\frac{1}{2}+\eta} )\, (t-s)^{1+\eta}$;

\myitem{(i')}\label{item47(1')}
$
\|\EE^{s} \delta A_{{s},u,{t}}\|_{L^m}  \leq  C \, [R]_{\mathcal{C}^{\tau,0}_{[S,T],x} L^{m}} \, (t-s)^{\frac{1}{2}+\tau} +C \Big( \| \mathcal{E}^{1,n,k}_{0,\cdot} \|_{L^{\infty,\infty}_{[S,T],x} L^m}  + \epsilon(n,k)\Big) {(t-s)};$

\item \label{item47(2)} 
$\| \delta A_{{s},u,{t}}\|_{L^m} \leq C\, \Big( \| \mathcal{E}^{1,n,k}_{0,\cdot} \|_{L^{\infty,\infty}_{[S,T],x} L^m}  + [\mathcal{E}^{1,n,k}]_{\mathcal{C}^{\frac{1}{2}-\zeta,0}_{[S,T],x} L^m}+\epsilon
(n,k)  \Big)   \,   (t-s)^{\frac{1}{2}+\varepsilon_2} $;

\item \label{item47(3)} If \ref{item47(1)} and \ref{item47(2)} are satisfied, \eqref{eq:convAt} implies the 
convergence in probability of $\sum_{i=1}^{N_k-1} A_{t^k_i,t^k_{i+1}}$ on any sequence $\Pi_k=\{t_i^k\}_{i=1}^{N_k}$ of partitions of $[\tilde{S},t]$ with a mesh that converges to $0$. We will show 
that the limit is $\mathcal{A}$, given in \eqref{eq:Prop51A-critic}.
\end{enumerate}

Then from \eqref{eq:sts-critic}, it comes
\begin{align*}%
 \| \mathcal{A}_{\tilde{T}} - \mathcal{A}_{\tilde{S}}- A_{\tilde{S},\tilde{T}} \|_{L^m} 
 &\leq  C\, \bigg(1+ \bigg|\log \frac{T^\eta (1+ \big\{ u^{n,k} \big\}_{n,k,m,\frac{1}{2}+\eta})}{ \| \mathcal{E}^{1,n,k}_{0,\cdot} \|_{L^{\infty,\infty}_{[S,T],x} L^m}  +\epsilon(n,k)} \bigg| \bigg) \, \Big( \| \mathcal{E}^{1,n,k}_{0,\cdot} \|_{L^{\infty,\infty}_{[S,T],x} L^m}  + \epsilon(n,k)\Big) \, (\tilde{T}-\tilde{S}) \nonumber\\
 &\quad+ C\, \Big( \| \mathcal{E}^{1,n,k}_{0,\cdot} \|_{L^{\infty,\infty}_{[S,T],x} L^m}+ [\mathcal{E}^{1,n,k}]_{\mathcal{C}^{\frac{1}{2}-\zeta,0}_{[S,T],x} L^m} + \epsilon(n,k) \Big) \,   (\tilde{T}-\tilde{S})^{\frac{1}{2}+\varepsilon_2} \\ 
 & \quad +C \, [R]_{\mathcal{C}^{\tau,0}_{[S,T],x} L^{m}} \, (\tilde{T}-\tilde{S})^{\frac{1}{2}+\tau} .
\end{align*}
Notice that $\mathcal{A}_{\tilde{S}}=0$ and $ \mathcal{A}_{\tilde{T}} - \mathcal{A}_{\tilde{S}}- A_{\tilde{S},\tilde{T}}=R_{\tilde{S},\tilde{T}}(x)$. Now divide on both sides by $(\tilde{T}-\tilde{S})^\tau$, and take the supremum over $(\tilde{S},\tilde{T}) \in \Delta_{[S,T]}$ and $x \in \T$. For $S\leq T$ such that $T-S\leq (2C)^{-2}=:\ell_0$, we get that $C \, [R]_{\mathcal{C}^{\tau,0}_{[S,T],x} L^{m}} \, (\tilde{T}-\tilde{S})^{\frac{1}{2}} \leq \frac{1}{2} [R]_{\mathcal{C}^{\tau,0}_{[S,T],x} L^{m}}$. Then subtract by this quantity on both sides to get
\begin{align}\label{eq:seminorm-R}
 [R]_{\mathcal{C}^{\tau,0}_{[S,T],x} L^{m}}
 &\leq  C\,  \bigg(1+ \bigg|\log \frac{T^\eta (1+ \big\{ u^{n,k} \big\}_{n,k,m,\frac{1}{2}+\eta})}{ \| \mathcal{E}^{1,n,k}_{0,\cdot} \|_{L^{\infty,\infty}_{[S,T],x} L^m}  +\epsilon(n,k)}\bigg| \bigg) \, \Big( \| \mathcal{E}^{1,n,k}_{0,\cdot} \|_{L^{\infty,\infty}_{[S,T],x} L^m}  + \epsilon(n,k) \Big) (T-S)^{1-\tau} \nonumber\\
 &\quad + C\, \Big( \| \mathcal{E}^{1,n,k}_{0,\cdot} \|_{L^{\infty,\infty}_{[S,T],x} L^m}  + [\mathcal{E}^{1,n,k}]_{\mathcal{C}^{\frac{1}{2}-\zeta,0}_{[S,T],x} L^m} +\epsilon(n,k)  \Big) .
\end{align}
We conclude by injecting \eqref{eq:seminorm-R} in \eqref{eq:E1-R-link}.

Finally, call $\M$ the constant $C$ that appears in the previous bound on $ \|\mathcal{E}^{1,n,k}_{s,t}(x) \|_{L^m}$, so as to keep track of it in Theorem \ref{thm:main-SPDE}$(b)$.

~

Let us now verify that \ref{item47(1)}, \ref{item47(1')}, \ref{item47(2)} and \ref{item47(3)} are satisfied.

\paragraph{Proof of \ref{item47(1)}:}  For $u \in [s,t]$, the tower property and Fubini's theorem give
\begin{align}\label{eq:deltaAst-decomp-SDE-critic}
\mathbb{E}^{s} \delta A_{s,u,t} & =  \int_u^{t} \int_{\T} p_{\tilde{T}-r}(x,y)\, \mathbb{E}^{s}\, \mathbb{E}^u \Big[ b^k(\EE^s \psi_r(y) +O_r(y)) - b^k(\EE^s \phi_r(y)+ O_r(y)) \nonumber \\ 
& \hspace{0.8cm} - b^k(\EE^u \psi_r(y) + O_r(y)) + b^k(\EE^u \phi_r(y)+ O_r(y)) \Big] \, dy dr  \nonumber \\ 
 & =: \int_u^{t} \int_{\T} p_{\tilde{T}-r}(x,y)\, \mathbb{E}^{s}\, \mathbb{E}^u \left[ F_{s,u,r,y}(O_r(y)) + \tilde{F}_{s,u,r,y}(O_r(y)) \right]  dy dr  , 
\end{align}
where $F_{s,u,r,y}, \tilde{F}_{s,u,r,y}: \R \times \Omega \rightarrow \R$ are defined by
\begin{align*}
F_{s,u,r,y}(\cdot) & =b^k(\EE^s \psi_r(y) +\cdot) - b^k(\EE^s \phi_r(y)+ \cdot) - b^k(\EE^u \psi_r(y) + \cdot) + b^k(\EE^u \psi_r(y) + \EE^s \phi_r(y) - \EE^s \psi_r(y)+ \cdot)  , \\
\tilde{F}_{s,u,r,y}(\cdot)& = b^k(\EE^u \phi_r(y) +\cdot) - b^k\left(\EE^u \psi_r(y) + \EE^s \phi_r(y) - \EE^s \psi_r(y)+ \cdot\right).
\end{align*}
Since $F_{s,u,r,y}$ is an $\mathcal{F}_u$-measurable function, Lemma \ref{lem:reg-O} gives that for $\lambda\in [0,1]$,
\begin{align*}
 |\mathbb{E}^u F_{s,u,r,y}(O_r(y)) | \leq C\, \| F_{s,u,r,y} \|_{\mathcal{B}_p^{\gamma-1-\lambda}} \, (r-u)^{\frac{1}{4}(\gamma-1-\lambda-\frac{1}{p})} .
\end{align*}
Moreover, by Lemma \ref{lem:besov-spaces}$(iii)$, it comes that
\begin{align*}
\| F_{s,u,r,y} \|_{\mathcal{B}_p^{\gamma-1-\lambda}} 
& \leq C\, \|b^k \|_{\mathcal{B}_p^\gamma} | \EE^s \psi_r(y)-\EE^u \psi_r(y) |\, | \EE^s \psi_r(y)-\EE^s \phi_r(y) |^\lambda .
\end{align*}
Therefore, using that  $(\psi)_{\mathcal{C}^{3/4,0}_{[S,T],x} L^{1,\infty}} \leq  (\psi)_{\mathcal{C}^{3/4,0}_{[S,T],x} L^{m,\infty}}$ yields
\begin{align*}
 | \EE^s \mathbb{E}^u F_{s,u,r,y}(O_r(y)) |
& \leq C\, \|b^k\|_{\mathcal{B}_p^\gamma}\,  (\psi)_{\mathcal{C}^{\frac{3}{4},0}_{[S,T],x}L^{m,\infty}}  | \EE^s \psi_r(y)-\EE^s \phi_r(y)|^\lambda \,  (t-s)^{\frac{3}{4}}\, (r-u)^{\frac{1}{4}(\gamma-1-\lambda-\frac{1}{p})}.
\end{align*}
Then from Jensen's inequality, \eqref{eq:comp-seminorms1} with $q=\infty$ and using $\gamma-1/p=-1$,
\begin{align}\label{eq:boundcondF}
 \| \EE^s \mathbb{E}^u F_{s,u,r,y}(O_r(y)) \|_{L^m}
& \leq C\, \|b^k\|_{\mathcal{B}_p^\gamma}\, [\psi]_{\mathcal{C}^{\frac{3}{4},0}_{[S,T],x} L^{m,\infty}}   \| \EE^s \psi_r(y)-\EE^s \phi_r(y) \|_{L^m}^\lambda\,  \,  (t-s)^{\frac{3}{4}}\, (r-u)^{\frac{1}{4}(-2-\lambda)}.
\end{align}
As for $\tilde{F}_{s,u,r,y}$, since it is also an $\mathcal{F}_u$-measurable function, we obtain, in the same fashion using Lemma~\ref{lem:reg-O} and Lemma~\ref{lem:besov-spaces}$(ii)$, that
\begin{align}
 \| \EE^s \mathbb{E}^u \tilde{F}_{s,u,r,y}(O_r(y)) \|_{L^m} 
&\leq C\, \|b^k\|_{\mathcal{B}_p^\gamma}\, \| \EE^u \psi_r(y) + \EE^s \phi_r(y) - \EE^s \psi_r(y)- \EE^u \phi_r(y) \|_{L^m}\, (r-u)^{-\frac{1}{2}}  \label{eq:tildeF-bound} \\
  &\leq C\, \|b^k\|_{\mathcal{B}_p^\gamma}\, (\psi-\phi)_{\mathcal{C}^{\frac{1}{2}+\eta,0}_{[S,T],x}L^m} (u-s)^{\frac{1}{2}+\eta} \, (r-u)^{-\frac{1}{2}} \nonumber \\
 & \leq C\, \|b^k\|_{\mathcal{B}_p^\gamma}\, ( (\phi)_{\mathcal{C}^{\frac{1}{2}+\eta,0}_{[S,T],x}L^m} +  (\psi)_{\mathcal{C}^{\frac{1}{2}+\eta,0}_{[S,T],x}L^m} ) \, (u-s)^{\frac{1}{2}+\eta} (r-u)^{-\frac{1}{2}}  \nonumber .
 \end{align}
 Using \eqref{eq:comp-seminorms1} with $m=q$, it follows that
 \begin{align}\label{eq:boundFtilde}
  \| \EE^s \mathbb{E}^u \tilde{F}_{s,u,r,y}(O_r(y)) \|_{L^m}
  &\leq C\, \|b^k \|_{\mathcal{B}_p^\gamma}\, ( (\phi)_{\mathcal{C}^{\frac{1}{2}+\eta,0}_{[S,T],x}L^m} + [\psi]_{\mathcal{C}^{\frac{1}{2}+\eta,0}_{[S,T],x}L^m} ) (u-s)^{\frac{1}{2}+\eta}\, (r-u)^{-\frac{1}{2}} .
\end{align}
Choosing $\lambda=0$ and recalling that $\eta < 1/4$, plug \eqref{eq:boundcondF} and \eqref{eq:boundFtilde} in \eqref{eq:deltaAst-decomp-SDE-critic} to obtain 
\begin{align*}
\|\EE^s  \delta A_{s,u,t} \|_{L^m} 
&\leq C\, \|b^k \|_{\mathcal{B}_p^\gamma}\,  [\psi]_{\mathcal{C}^{\frac{3}{4},0}_{[S,T],x} L^{m,\infty}} \, (t-s)^{\frac{3}{4}}\, (t-u)^{\frac{1}{2}} \\
&\quad +  C\, \|b^k \|_{\mathcal{B}_p^\gamma}\, ( (\phi)_{\mathcal{C}^{\frac{1}{2}+\eta,0}_{[S,T],x}L^m} + [\psi]_{\mathcal{C}^{\frac{3}{4},0}_{[S,T],x}L^m} ) (u-s)^{\frac{1}{2}+\eta}\, (t-u)^{\frac{1}{2}}\\
&\leq C\, \|b^k\|_{\mathcal{B}_p^\gamma}\, ( (\phi)_{\mathcal{C}^{\frac{1}{2}+\eta,0}_{[S,T],x}L^m} + [\psi]_{\mathcal{C}^{\frac{3}{4},0}_{[S,T],x}L^{m,\infty}} ) (t-s)^{1+\eta}  .
\end{align*}
Notice that $(\phi)_{\mathcal{C}^{\frac{1}{2}+\eta,0}_{[S,T],x}L^m}  = (u^{n,k}-O^n)_{\mathcal{C}^{\frac{1}{2}+\eta,0}_{[S,T],x}L^m}$, use \eqref{eq:comp-seminorms2} and \eqref{eq:holder-reg-sol} to deduce~\ref{item47(1)}.

\paragraph{Proof of \ref{item47(1')}:}  Relying on \eqref{eq:deltaAst-decomp-SDE-critic},
now use \eqref{eq:boundcondF} with $\lambda=1$ and \eqref{eq:tildeF-bound} to get
\begin{align}
\begin{split}
\|\mathbb{E}^{s} \delta A_{s,u,t} \|_{L^m} 
&\leq  C\, \|b^k\|_{\mathcal{B}_p^\gamma}\,  \| \EE^s \psi_r(y)-\EE^s \phi_r(y) \|_{L^m}\,  [\psi]_{\mathcal{C}^{\frac{3}{4},0}_{[S,T],x} L^{m,\infty}} \,  (t-s) \,  \\ & \quad + C\, \|b^k\|_{\mathcal{B}_p^\gamma}\, \| \EE^u \psi_r(y) + \EE^s \phi_r(y) - \EE^s \psi_r(y)- \EE^u \phi_r(y) \|_{L^m} \, (t-s)^{\frac{1}{2}}  . \label{eq:general-delta-critic}
\end{split}
\end{align}
Let us now bound the terms that appear in the right-hand side. Using the conditional Jensen's inequality and recalling \eqref{def:error}, 
\begin{align}
\| \EE^s \psi_r(y) - \EE^s \phi_r(y) \|_{L^m}  \leq \| \psi_r - \phi_r \|_{L^m} 
& \leq \| v_r(y) - v^{n,k}_r(y)  + P_r \psi_0(y) - P^n \psi_0(y) \|_{L^m} \nonumber \\ & \leq   \| \mathcal{E}^{n,k}_{0,\cdot} \|_{L^{\infty,\infty}_{[S,T],x} L^m}  + \sup_{\substack{t \in [0,1] \\ x \in \T}}| P_t \psi_0(x) - P^n_t \psi_0(x) | \nonumber \\
& \leq   \| \mathcal{E}^{n,k}_{0,\cdot} \|_{L^{\infty,\infty}_{[S,T],x} L^m}  + \epsilon(n,k) .  \label{eq:general-delta-critic1}
\end{align}
Consider now $\| \EE^u \psi_r(y) + \EE^s \phi_r(y) - \EE^s \psi_r(y)- \EE^u \phi_r(y) \|_{L^m}$. Introduce the pivot term $v^k_{u}-v^k_{s}$ to get 
\begin{align*}
&\| \EE^u \psi_r(y) + \EE^s \phi_r(y) - \EE^s \psi_r(y)- \EE^u \phi_r(y)\|_{L^m} \\
&\quad  \leq (v-v^k)_{\mathcal{C}^{\frac{1}{2},0}_{[S,T],x}L^m} (u-s)^{\frac{1}{2}} + \|\EE^u [v^k_r(y)-\phi_r(y)] - \EE^s [v^k_r(y) -\phi_r(y)] \|_{L^m}.
\end{align*}
Now use \eqref{eq:comp-seminorms1} with $m=q$,  \eqref{eq:comp-seminorms} with $\rho=v^k-\phi$ and $Y=\int_0^s P_{r-\theta} b^k (u_\theta)(y) - P^n_{(r-\theta)_h} b^k(u^{n,k}_{\theta_h})(y) \, d \theta $ and recall the definitions of $\mathcal{E}^{2,n,k},\mathcal{E}^{3,n,k}, V^k$ in \eqref{def:E2}, \eqref{def:E4} and \eqref{def:Vk} to get
\begin{align*}
&\| \EE^u \psi_r(y) + \EE^s \phi_r(y) - \EE^s \psi_r(y)- \EE^u \phi_r(y)\|_{L^m} \\
&\quad  \leq  (v-v^k)_{\mathcal{C}^{\frac{1}{2},0}_{[S,T],x}L^m} (r-s)^{\frac{1}{2}} + 2\Big\| \int_s^r P_{r-\theta} b^k (u_\theta)(y) - P^n_{(r-\theta)_h} b^k(u^{n,k}_{\theta_h})(y) \, d \theta \Big\|_{L^m} \\
& \leq [V^k]_{\mathcal{C}^{\frac{1}{2},0}_{[S,T],x}L^m} (r-s)^{\frac{1}{2}}  + 2 \| \mathcal{E}^{2,n,k}_{s,r} + \mathcal{E}^{3,n,k}_{s,r} \|_{L^m} \\ & \quad +  2 \Big\| \int_s^r P_{r-\theta} \left(b^k (P_{\theta} \psi_0 + v_\theta +O_\theta ) -b^k(P^n_{\theta} \psi_0+v^{n,k}_\theta +O^n_{\theta}) \right)(y) \, d\theta \Big\|_{L^m} \\
& \leq\epsilon(n,k) (r-s)^{\frac{1}{2}} +[R]_{\mathcal{C}^{\tau,0}_{[S,T],x} L^m} (r-s)^{\tau}  \\ & \quad +  2 \Big\| \int_s^r P_{r-\theta} \left(b^k (\EE^s [P_{\theta} \psi_0 + v_\theta] +O_\theta ) -b^k(\EE^s[P^n_{\theta} \psi_0+v^{n,k}_\theta] +O^n_{\theta}) \right)(y) \, d\theta \Big\|_{L^m} .
\end{align*}
Apply Lemma \ref{lem:ssl-o-on} for the last term in the right-hand side of the previous inequality (with $\beta = \gamma-1$), use that 
$\| \EE^s [P_{\theta} \psi_0(y) + v_\theta(y)]- \EE^s[P^n_{\theta} \psi_0(y)+v^{n,k}_\theta(y)]\|_{L^m} \leq   \| \mathcal{E}^{n,k}_{0,\cdot} \|_{L^{\infty,\infty}_{[S,T],x} L^m}  +\epsilon(n,k)$ and that $ \|b^k\|_{\mathcal{B}^\gamma_{p}} \leq \|b\|_{\mathcal{B}^\gamma_{p}}<\infty$ to get
\begin{align}
\begin{split}
&\| \EE^u \psi_r(y) + \EE^s \phi_r(y) - \EE^s \psi_r(y)- \EE^u \phi_r(y)\|_{L^m} \\ 
 & \leq C \Big(  \| \mathcal{E}^{n,k}_{0,\cdot} \|_{L^{\infty,\infty}_{[S,T],x} L^m}  + \epsilon(n,k) \Big) (r-s)^{\frac{1}{2}} +  [R]_{\mathcal{C}^{\tau,0}_{[S,T],x} L^m} (r-s)^{\tau}  . \label{eq:general-delta-critic2}
 \end{split}
\end{align}
Plugging \eqref{eq:general-delta-critic1} and \eqref{eq:general-delta-critic2} in \eqref{eq:general-delta-critic}, it comes
\begin{align}\label{eq:bound-E-delta-critic}
\|\mathbb{E}^{s} \delta A_{s,u,t} \|_{L^m} 
&\leq C\, (1+ [\psi]_{\mathcal{C}^{\frac{3}{4},0}_{[S,T],x} L^{m,\infty}})  \Big( \| \mathcal{E}^{n,k}_{0,\cdot} \|_{L^{\infty,\infty}_{[S,T],x} L^m}  + \epsilon(n,k) \Big)\, (t-s)   + C [R]_{\mathcal{C}^{\tau,0}_{[S,T],x} L^m} (t-s)^{\frac{1}{2}+\tau} .
\end{align}
Now we wish to recover an upper bound with $\mathcal{E}^{1,n,k}$ instead of $\mathcal{E}^{n,k}$, so recalling \eqref{eq:defepsilonhn}, we write
\begin{align} \label{eq:comparEandE1}
\|\mathcal{E}^{n,k}_{0,\cdot} \|_{L^{\infty,\infty}_{[S,T],x} L^{m}}  & \leq \|\mathcal{E}^{1,n,k}_{0,\cdot} \|_{L^{\infty,\infty}_{[S,T],x} L^{m}}  +\|V^k_{0,\cdot} \|_{L^{\infty,\infty}_{[S,T],x} L^{m}} + \|\mathcal{E}^{2,n,k}_{0,\cdot} \|_{L^{\infty,\infty}_{[S,T],x} L^{m}}  + \|\mathcal{E}^{3,n,k}_{0,\cdot} \|_{L^{\infty,\infty}_{[S,T],x} L^{m}}  \nonumber \\
& \leq \|\mathcal{E}^{1,n,k}_{0,\cdot} \|_{L^{\infty,\infty}_{[S,T],x} L^{m}}   + \epsilon(n,k).
\end{align}
Using \eqref{eq:comparEandE1} and that $[\psi]_{\mathcal{C}^{3/4,0}_{[S,T],x} L^{,\infty}}<\infty$ in \eqref{eq:bound-E-delta-critic} permits to deduce~\ref{item47(1')}.

\paragraph{Proof of \ref{item47(2)}:} Apply Lemma \ref{cor:ssl-4} with $\lambda_2=\eta$, $\lambda_1=\lambda=1$ and $\beta=\gamma$, which satisfies $\gamma>-1+\eta$ by definition of $\eta$, to get
\begin{align*}
 \| \delta A_{s,u,t} \|_{L^m}  
 & \leq C  \| b^k \|_{\mathcal{B}_p^\gamma} \, (\psi)_{\mathcal{C}^{\frac{3}{4},0}_{[s,t],x} L^{m,\infty}}^\eta (t-s)^{\frac{3\eta}{4}}  \sup_{\substack{y \in \T \\ r \in [u,T]}} \| \EE^u \psi_r(y) - \EE^u \phi_r(y) \|_{L^m} (t-u)^{\frac{1}{2}-\frac{\eta}{4}} \\ 
  & \quad + C \| b^k\|_{\mathcal{B}_p^\gamma} \sup_{\substack{y \in \T \\ r \in [u,T]}} \| \EE^s \psi_r(y)+\EE^u \phi_r(y)-\EE^u \psi_r(y)-\EE^s \phi_r(y) \|_{L^m} (t-u)^{\frac{1}{2}} .
\end{align*}
Then use \eqref{eq:comp-seminorms1} with $q=\infty$ and \eqref{eq:general-delta-critic1} to get
 \begin{align*}
 \| \delta A_{s,u,t} \|_{L^m}  & \leq C  \| b^k \|_{\mathcal{B}_p^\gamma}  ([\psi]_{\mathcal{C}^{\frac{3}{4},0}_{[s,t],x} L^{m,\infty}}+1) \Big(  \| \mathcal{E}^{n,k}_{0,\cdot} \|_{L^{\infty,\infty}_{[S,T],x} L^m} + \epsilon(n,k) \Big)   (t-u)^{\frac{1}{2}+\frac{\eta}{2}}   \\ 
  & \quad + C \| b^k\|_{\mathcal{B}_p^\gamma} (\psi-\phi)_{\mathcal{C}^{\frac{1}{2}-\zeta,0}_{[S,T],x} L^m} (t-u)^{1-\zeta} .
  \end{align*}
  Using \eqref{eq:comp-seminorms3} gives
   \begin{align}\label{eq:bound-delta-critic} 
   \begin{split}
  \| \delta A_{s,u,t} \|_{L^m} & \leq C  \| b^k\|_{\mathcal{B}_p^\gamma}  ([\psi]_{\mathcal{C}^{\frac{3}{4},0}_{[s,t],x} L^{m,\infty}}+1) \Big(  \| \mathcal{E}^{n,k}_{0,\cdot} \|_{L^{\infty,\infty}_{[S,T],x} L^m} + \epsilon(n,k) \Big)   (t-u)^{\frac{1}{2}+\frac{\eta}{2}} \\ 
  & \quad + C \|b^k \|_{\mathcal{B}_p^\gamma} [\mathcal{E}^{n,k}]_{\mathcal{C}^{\frac{1}{2}-\zeta,0}_{[S,T],x} L^m} (t-u)^{1-\zeta}  .
  \end{split}
  \end{align}
As in \eqref{eq:comparEandE1}, it holds that $[\mathcal{E}^{n,k} ]_{\mathcal{C}^{1/2-\zeta,0}_{[S,T],x} L^m} \leq  [\mathcal{E}^{1,n,k} ]_{\mathcal{C}^{1/2-\zeta,0}_{[S,T],x} L^m}   + \epsilon(n,k)$. Hence,
using \eqref{eq:comparEandE1} and that $[\psi]_{\mathcal{C}^{3/4,0}_{[S,T],x}L^{m,\infty}}$ is finite, we obtain \ref{item47(2)} from \eqref{eq:bound-delta-critic}.

\paragraph{Proof of \ref{item47(3)}:} The proof is identical to the last point in the proof of Proposition~\ref{prop:ssl-o-on-2}.
\end{proof}

\section{Regularisation effect of the discrete Ornstein-Uhlenbeck process, and H\"older regularity of $u^{n,k}$ and $\mathcal{E}^{3,n,k}$}\label{sec:reg-On-1}

In Section~\ref{sec:reg-O}, it was crucial in the sewing strategy to compute precisely the conditional expectation of functions of the OU process, which is achieved through the relation \eqref{eq:condexpf(OU)}. In this equality, the conditional expectation of the OU process is involved, which satisfies $\EE^{u} O_{t}(x) = \int_{0}^u\int_{\T} p_{t-s}(x,y)\, \xi(ds,dy) = P_{t-u} O_{u}(x)$. Similar tools will be needed for the discrete OU process $O^n$ (recall its definition and its variance $Q^n$ from \eqref{def:discrete-OU}) to establish regularisation properties by such a process. But due to the discretisation in time, $\EE^{u} O^n_{t}(x)$ cannot be written using $P^n$. Instead, for $(u,t) \in \Delta_{[0,1]}$ and $x \in \T$,  we introduce:
\begin{align}\label{def:Onst}
\widehat{O}^n_{u,t}(x) := \int_0^u \int_\T p^n_{(t-r)_h}(x,y) \xi(dy,dr) .
\end{align}
A relation similar to \eqref{eq:condexpf(OU)} is achieved using $\widehat{O}^n$, namely that $\EE^s f\left(O^n_t(x)\right)=G_{Q^n(t-s)} f\big(\widehat{O}^n_{s,t}(x)\big)$, see Lemma~\ref{lem:reg-On}. This will be particularly useful in this section.

Because of the discretisation in time, upper bounds on quantities such as $\int_0^t f(\phi_r + O^n_r)(x)\, dr$ will involve at best the supremum norm of $f$. But herein, $f$ plays the role of a sequence that converges to a distribution. So while having $\| f \|_\infty$ in an upper bound could be troublesome, such quantities will always be compensated by negative powers of $n$. We use these properties to study the regularity of the numerical scheme $u^{n,k}$.

\subsection{Regularisation properties of the discrete Ornstein-Uhlenbeck process}\label{subsec:reg-discreteOU}

In this subsection, we state general regularisation properties of $O^n$. First, we proceed as before by freezing the process $\psi$ into an $\mathcal{F}_{S}$-measurable process (see Lemma~\ref{lem:bound-Khn}). Then the main result, which plays the same role as Lemma~\ref{lem:6.1athreya} in the case of the continuous OU process, is given in Proposition~\ref{prop:bound-Khn}.

\begin{lemma}\label{lem:bound-Khn}
Let $m \in [2, \infty)$, $p \in [m,\infty]$, and assume that $\gamma-1/p \ge -1$. There exists a constant $C>0$ such that for any $0 \leq S < T \leq 1$, any $f \in \mathcal{C}^\infty_{b}(\mathbb{R}, \mathbb{R}) \cap \mathcal{B}_p^\gamma$, any $n \in \mathbb{N}^*$, any
 $(s, t)\in \Delta_{[S,T]}$ and any $\psi : [0,1] \times \T \times \Omega \to \mathbb{R}$ such that for all $r \in [0,1]$ and $x \in \mathbb{T}$, $\psi_r(x)$ is an $\mathbb{R}$-valued $\mathcal{F}_S$-measurable random variable, there is
\begin{align*}%
\sup_{x \in \mathbb{T}} \Big\| \Big( \EE^S \Big| \int_s^t \int_{\T} p_{T-r}(x,y)  f(\psi_{r_h}(y_n) + O^n_{r_h}(y_n)) dy d r  \Big|^m \Big)^{\frac{1}{m}} \Big\|_{L^\infty} \nonumber \leq C\, \Big( \|f\|_\infty\,  n^{-\frac{1}{2}} + \| f \|_{\mathcal{B}_p^\gamma} \Big) (t-s)^{\frac{3}{4}}  .
\end{align*}
\end{lemma}

\begin{proof}
Fix $x \in \mathbb{T}$. We will check the assumptions in order to apply the Stochastic Sewing Lemma (see Lemma~\ref{lem:SSL}). For $(s,t) \in \Delta_{[S,T]}$, let 
\begin{align*}
A_{s,t} & = \mathbb{E}^s \int_s^t \int_{\T} p_{T-r}(x,y)  f(\psi_{r_h}(y_n) + O^n_{r_h}(y_n)) \, d y d r  \\
\mathcal{A}_t & = \int_S^t  \int_{\T} p_{T-r}(x,y)  f(\psi_{r_h}(y_n)+ O^n_{r_h}(y_n)) \, dy d r .
\end{align*}
Let $u\in [s,t]$ and observe that $\mathbb{E}^s \delta A_{s,u,t}=0$, so \eqref{eq:condsew1} holds with $\Gamma_1=0$. We will prove that \eqref{eq:condsew2} holds with
\begin{equation}\label{eq:Gamma2}
\Gamma_2 =  C \|f\|_\infty   n^{-\frac{1}{2}} + C \| f \|_{\mathcal{B}_p^\gamma} .
\end{equation}

\paragraph{The case $t-s\leq 2h$.} In this case one has
\begin{align}\label{eq:Ast<h-Khn}
|A_{s,t}| \leq \|f\|_\infty (t-s) \leq  \|f\|_\infty h^{\frac{1}{4}} (t-s)^{\frac{3}{4}} .
\end{align}

\paragraph{The case $t-s>2h$.} Here  split $A_{s,t}$ in two to get
\begin{align*}
A_{s,t} &= \mathbb{E}^s \int_s^{s+2h}  \int_{\T} p_{T-r}(x,y)  f(\psi_{r_h} (y_n)+ O^n_{r_h}(y_n)) dy d r + \mathbb{E}^s \int_{s+2h}^t \int_{\T} p_{T-r}(x,y)  f(\psi_{r_h} (y_n)+ O^n_{r_h}(y_n)) dy d r\\
&=: I + J .
\end{align*}
For the first part, there is
\begin{equation*}
|I| = \Big|\mathbb{E}^s \int_s^{s+2h}  \int_{\T} p_{T-r}(x,y)  f(\psi_{r_h} (y_n)+ O^n_{r_h}(y_n)) \, dy d r  \Big| \leq 2h\, \|f\|_\infty \leq C\, \|f\|_\infty h^{\frac{1}{4}} (t-s)^{\frac{3}{4}}.
\end{equation*}
Then using Lemma \ref{lem:reg-On}, one gets
\begin{align*}%
\| J \|_{L^m} &  \leq C \int_{s+2h}^t  \| f \|_{\mathcal{B}_p^\gamma} (r_h-s)^{\frac{1}{4}(\gamma-\frac{1}{p})}  \, dr .
\end{align*}
Using that $2(r_h-s) \ge (r-s)$ and $0>\gamma-1/p\geq-1$, we deduce that for $t-s >2h$,
\begin{align}\label{eq:Ast>h-Khn}
\|J \|_{L^m} 
& \leq C \| f \|_{\mathcal{B}_p^\gamma} (t-s)^{\frac{3}{4}} .
\end{align}

Hence, combining \eqref{eq:Ast<h-Khn} and \eqref{eq:Ast>h-Khn}, there is for all $s \leq t$:
\begin{align*}
\| A_{s,t}\|_{L^{m}} & \leq C \Big(   \|f\|_{\infty} h^{\frac{1}{4}} (t-s)^{\frac{3}{4}} + \| f \|_{\mathcal{B}_p^\gamma} (t-s)^{\frac{3}{4}} \Big) .
\end{align*}
Thus for any $u\in [s,t]$,
\begin{align*}
\| \delta A_{s,u,t}\|_{L^m} & \leq \| A_{s,t}\|_{L^{m}}+\| A_{s,u}\|_{L^{m}}+\| A_{u,t}\|_{L^{m}}\leq  C \Big(   \|f\|_{\infty} h^{\frac{1}{4}} + \| f \|_{\mathcal{B}_p^\gamma} \Big)  (t-s)^{\frac{3}{4}} .
\end{align*}
Finally, recall that $h=c(2n)^{-2}$ to deduce that \eqref{eq:condsew2} holds with $\Gamma_{2}$ given by \eqref{eq:Gamma2}.

\paragraph{Convergence in probability.}
For $(\Pi_k)_{k \in \mathbb{N}}$ partitions of $[S,t]$ with $\Pi_k=\{t_i^k\}_{i=1}^{N_k}$ and mesh size which converges to zero, there is
\begin{align*}
\Big\| \mathcal{A}_t -  \sum_{i=1}^{N_{k}-1} A_{t_i^k,t_{i+1}^k} \Big\|_{L^1} & \leq \sum_{i=1}^{N_{k}-1} \int_{t_i^k}^{t_{i+1}^k} \int_{\T} p_{T-r}(x,y) \mathbb{E}   \Big|  f(\psi_{r_h}(y_n)+ O^n_{r_h}(y_n))  -\mathbb{E}^{t_i^k}[  f(\psi_{r_h}(y_n)+ O^n_{r_h}(y_n))] \Big|  \diff r\\
& =: I_2 .
\end{align*}
Note that if $r_h \leq t_i^k$, then $ \mathbb{E} \big|  f(\psi_{r_h}(y_n)+ O^n_{r_h}(y_n))  -\mathbb{E}^{t_i^k}[  f(\psi_{r_h}(y_n) + O^n_{r_h}(y_n))] \big| = 0$. On the other hand when $r_h \in (t_i^k,t_{i+1}^k]$, define the $\mathscr{B}(\R)\otimes \mathcal{F}_{S}$-measurable random function $F_{r_{h},y_{n}}(\cdot) := f(\psi_{r_h}(y_n) +\cdot)$. Then
Lemma~\ref{lem:reg-On} gives that
\begin{align*}
&\mathbb{E} \left|  f(\psi_{r_h}(y_n)+ O^n_{r_h}(y_n)) - \mathbb{E}^{t_i^k}[  f(\psi_{r_h}(y_n) + O^n_{r_h}(y_n))] \right| \\
 &\quad =  \mathbb{E} \left| F_{r_{h},y_{n}}(O^n_{r_h}(y_n)) - G_{Q^n(r_h-t_i^k)}F_{r_{h},y_{n}}(\widehat{O}^n_{t_i^k,r}(y_n)) \right| \\
 &\quad \leq \mathbb{E} \left| F_{r_{h},y_{n}}(O^n_{r_h}(y_n)) - F_{r_{h},y_{n}}(\widehat{O}^n_{t_i^k,r}(y_n)) \right| +  \mathbb{E} \left| (I - G_{Q^n(r_h-t_i^k)})F_{r_{h},y_{n}}(\widehat{O}^n_{t_i^k,r}(y_n)) \right| \\
 &\quad \leq C \| f \|_{\mathcal{C}^1} \EE|O^n_{r_h}(y_n) - \widehat{O}^n_{t_i^k,r}(y_n)| +  C \| f \|_{\mathcal{C}^1} Q^n(r_h-t_i^k)^{\frac{1}{2}},
\end{align*}
where the last inequality comes from Lemma~\ref{lem:reg-Pnh} applied with $\alpha=\beta=1$. Now use $\EE|O^n_{r_h}(y_n) - \widehat{O}^n_{t_i^k,r}(y_n)| \leq (\EE|O^n_{r_h}(y_n) - \widehat{O}^n_{t_i^k,r}(y_n)|^2)^{1/2} = Q^n(r_h-t_i^k)^{\frac{1}{2}}$ and $Q^n(r_h-t_i^k) \lesssim (r_h-t_i^k)^{1/2}$ (Lemma~\ref{lem:bound-Qn}$(iii)$) to deduce that
\begin{align*}
\mathbb{E} \left|  f(\psi_{r_h}(y_n)+ O^n_{r_h}(y_n)) - \mathbb{E}^{t_i^k}[  f(\psi_{r_h}(y_n) + O^n_{r_h}(y_n))] \right| \leq C\, \| f \|_{\mathcal{C}^1}\, (r_h-t_i^k)^{\frac{1}{4}}.
\end{align*}
It follows that $I_2  \leq \sum_{i=1}^{N_{k}-1} \int_{t_i^k}^{t_{i+1}^k} \| f \|_{\mathcal{C}^1} |\Pi_{k}|^{\frac{1}{4}} \, dr$,
and therefore $\sum_{i=1}^{N_{k}-1} A_{t_{i}^k, t_{i+1}^k}$ converges in probability to $\mathcal{A}_{t}$ as $k\to +\infty$.
The conclusion follows by applying Lemma~\ref{lem:SSL} with $\varepsilon_1>0$, $\varepsilon_2 = \frac{1}{4}$.
\end{proof}

\begin{proposition}\label{prop:bound-Khn}
Let $\eta \in (0,\frac{1}{2})$, $m \in [2, \infty)$ and assume that $\gamma-1/p \ge -1$.
There exists a constant $C>0$ such that for any $0 \leq S < T \leq 1$, any  $f \in \mathcal{C}^\infty_{b}(\mathbb{R}, \mathbb{R}) \cap \mathcal{B}_p^\gamma$, any $(s, t)\in \Delta_{[S,T]}$, any $n \in \mathbb{N}^*$ and any $\mathbb{F}$-adapted process $\psi : [0,1] \times \T \times \Omega \to \mathbb{R}$, there is
\begin{align}\label{eq:bound-Khn-general}
\begin{split}
 & \Big\| \Big( \EE^s \Big| \int_s^t \int_\T  p_{T-r}(x,y) f(\psi_{r_h} (y_n)+ O^n_{r_h}(y_n)) \, dy dr \Big|^m \Big)^{\frac{1}{m}}  \Big\|_{L^{\infty}}  \\  
 &\quad \leq  C\, \Big( \|f\|_\infty\,  n^{-\frac{1}{2}} + \| f \|_{\mathcal{B}_p^\gamma} \Big) (t-s)^{\frac{3}{4}}    + C\,  ( \psi )_{\mathcal{C}^{\frac{1}{2}+\eta,0}_{[S,T],x} L^{m,\infty}}  \, \Big( \| f \|_{\mathcal{B}_p^{\gamma}} + \| f \|_{\mathcal{C}^1} n^{-1} \Big)   \, (t-s)^{1+\eta} .
\end{split}
\end{align}
\end{proposition}

\begin{proof}
Fix $x \in \mathbb{T}$ and $0 \leq S < T \leq 1$. We will check the assumptions in order to apply the Stochastic Sewing Lemma (Lemma~\ref{lem:SSL} with $q=\infty$). For any $(s,t) \in \Delta_{[S,T]}$, define 
\begin{align*}
A_{s,t} & = \int_s^t \int_\T p_{T-r}(x,y) f(\EE^s [\psi_{r_h}(y_n)]+ O^n_{r_h}(y_n)) \, dy dr  \\
\mathcal{A}_t & =\int_S^t  \int_\T p_{T-r}(x,y) f( \psi_{r_h}(y_n)+ O^n_{r_h}(y_n)) \, dy dr.
\end{align*}
Recall that $h=c(2n)^{-2}$. To show that \eqref{eq:condsew1} and \eqref{eq:condsew2} hold true with $\varepsilon_1= \eta$, $\varepsilon_2=\frac{1}{4}$, we prove that there exists $C>0$ which does not depend on $s,t,S$ and $T$ such that for $u = (s+t)/2$,
\begin{enumerate}[label=(\roman*)]
\item \label{item56(1)} 
$\|\EE^{{s}} [\delta A_{s,u,t}]\|_{L^\infty} \leq  C\,  (\psi)_{\mathcal{C}^{\frac{1}{2}+\eta,0}_{[S,T],x} L^{m,\infty}}   \, \Big(   \| f \|_{\mathcal{B}_p^{\gamma}} + \| f \|_{\mathcal{C}^1} n^{-1} \Big)   \, (t-s)^{1+\eta} $;

\item \label{item56(2)} 
$\Big\| \Big( \EE^S | \delta A_{s,u,t} |^m \Big)^{\frac{1}{m}} \Big\|_{L^\infty} \leq C\, \Big( \|f\|_{\infty}\, n^{-\frac{1}{2}} + \| f \|_{\mathcal{B}_p^\gamma} \Big) \, ({t}-{s})^{\frac{3}{4}}$;

\item \label{item56(3)} If \ref{item56(1)} and \ref{item56(2)} are satisfied, \eqref{eq:convAt} gives the 
convergence in probability of $\sum_{i=1}^{N_k-1} A_{t^k_i,t^k_{i+1}}$ on any sequence of 
partitions $\Pi_k=\{t_i^k\}_{i=1}^{N_k}$ of $[S,t]$ with a mesh that converges to $0$. We will show 
that the limit is $\mathcal{A}$.
\end{enumerate}

Assuming for now that \ref{item56(1)}, \ref{item56(2)} and \ref{item56(3)} hold, Lemma~\ref{lem:SSL} gives that
\begin{align*}
 & \Big\| \Big( \EE^s \Big| \int_s^t \int_\T p_{T-r}(x,y) f(\psi_{r_h}(y_n) + O_{r_h}(y_n)) \, dy d r \Big|^m \Big)^{\frac{1}{m}}  \Big\|_{L^{\infty}} \\
   & \leq C\, \Big( \|f\|_{\infty}\, n^{-\frac{1}{2}} + \| f \|_{\mathcal{B}_p^\gamma} \Big) \, ({t}-{s})^{\frac{3}{4}}  + C\,  ( \psi )_{\mathcal{C}^{\frac{1}{2}+\eta,0}_{[S,T],x} L^{m,\infty}}  \, \Big( \| f \|_{\mathcal{B}_p^{\gamma}} + \| f \|_{\mathcal{C}^1} n^{-1} \Big)   \, (t-s)^{1+\eta} +\| A_{s,t}\|_{L^m}.
\end{align*}
Apply Lemma \ref{lem:bound-Khn} for $\EE^s \psi_r $ in place of $\psi_r$ and get that $\|A_{s,t}\|_{L^m} \leq C\, ( \|f\|_{\infty} n^{-\frac{1}{2}} + \| f \|_{\mathcal{B}_p^\gamma}) \, ({t}-{s})^{\frac{3}{4}}$. Taking the supremum over $x \in \T$, one deduces \eqref{eq:bound-Khn-general}.

~

Let us now verify that \ref{item56(1)}, \ref{item56(2)} and \ref{item56(3)}  are satisfied.

\paragraph{Proof of \ref{item56(1)}:}  For $u \in [s,{t}]$, the tower property gives
\begin{align}\label{eq:condexpdeltaA}
\mathbb{E}^s \delta A_{s,u,t} & = \mathbb{E}^s \int_u^t \int_\T  p_{T-r}(x,y) \mathbb{E}^u\Big[ f(\EE^s[\psi_{r_h}(y_n)]+ O^n_{r_h}(y_n))- f(\EE^u[\psi_{r_h}(y_n)]+ O^n_{r_h}(y_n)) \Big]\, dy dr  .
\end{align}

\paragraph{The case $u-s \leq 2h$.}
In this case, the Lipschitz property of $f$ gives
\begin{align*}
\| \mathbb{E}^s \delta A_{s,u,t} \|_{L^\infty} & \leq  \|f\|_{\mathcal{C}^1} \int_u^t \int_{\T}  p_{T-r} (x,y) \, \EE^s|\EE^s \psi_{r_h}(y_n)-\EE^u \psi_{r_h}(y_n)| \, dy dr.
\end{align*}
If $r_h \leq s$, then $\EE^s |\EE^s \psi_{r_h}-\EE^u \psi_{r_h}|=0$. Assume now that $r_h > s$. If $r_h \ge u$, we have
$$\EE^s | \EE^s \psi_{r_h}-\EE^u \psi_{r_h} |\leq ( \psi )_{\mathcal{C}^{\frac{1}{2}+\eta,0}_{[S,T],x} L^{1,\infty}}  (r_h-s)^{\frac{1}{2}+\eta} .$$
If $r_h \in (s,u)$, then 
$\EE^s| \EE^s \psi_{r_h}-\EE^u \psi_{r_h} |  \leq ( \psi )_{\mathcal{C}^{1/2+\eta,0}_{[S,T],x} L^{1,\infty}}  (r_h-s)^{\frac{1}{2}+\eta} \leq ( \psi )_{\mathcal{C}^{1/2+\eta,0}_{[S,T],x} L^{m,\infty}} (r_h-s)^{\eta} h^{\frac{1}{2}} $.
Therefore, using $h=c(2n)^{-2}$,
\begin{align}\label{eq:J0Prop56}
\| \mathbb{E}^s \delta A_{s,u,t} \|_{L^\infty} & \leq 
  C \|f\|_{\mathcal{C}^1}  (\psi )_{\mathcal{C}^{\frac{1}{2}+\eta,0}_{[s,t],x} L^{m,\infty}}  (t-s)^{1+\eta} n^{-1} .
\end{align}

\paragraph{The case $u-s > 2h$.} We split the integral in \eqref{eq:condexpdeltaA} into two parts, the first one which integrates between $u$ and $u+2h$, the second one integrating between $u+2h$ and $t$, so that $\EE^s\delta A_{s,u,t} =: J_1 + J_2$. For $J_1$, we obtain as before that
\begin{align*}
| J_1 |  & \leq  \|f\|_{\mathcal{C}^1} \int_u^{u+2h} \int_\T p_{T-r}(x,y) \EE^s |\EE^s \psi_{r_h}(y_n)-\EE^u \psi_{r_h}(y_n)|\, dy dr,
\end{align*}
and using similar arguments and the fact that $u-s > 2h$, it follows that
\begin{align}\label{eq:J1Prop56}
\| J_{1}\|_{L^\infty} = \|\mathbb{E}^s \delta A_{s,u,u+2h}\|_{L^\infty}
& \leq  C  \|f\|_{\mathcal{C}^1} ( \psi )_{\mathcal{C}^{\frac{1}{2}+\eta,0}_{[S,T],x} L^{m,\infty}} (t-s)^{1+\eta} h^{\frac{1}{2}} \nonumber\\
& \leq  C  \|f\|_{\mathcal{C}^1} (\psi)_{\mathcal{C}^{\frac{1}{2}+\eta,0}_{[S,T],x} L^{m,\infty}} (t-s)^{1+\eta} n^{-1}.
\end{align}
As for $J_2$, write
\begin{align*}
 J_2 &= \mathbb{E}^s \int_{u+2h}^t \int_{\T} p_{T-r}(x,y)\, \mathbb{E}^u \Big[f(\EE^s [\psi_{r_h}(y_n)]+ O^n_{r_h}(y_n)) - f(\EE^u [\psi_{r_h}(y_n)]+ O^n_{r_h}(y_n)) \Big] \, dy dr .
\end{align*}
Lemma \ref{lem:reg-On} gives that%
\begin{align*}
| J_2| & \leq C \mathbb{E}^s \int_{u+2h}^t \int_{\T} p_{T-r}(x,y)  \|  f(\EE^s \psi_{r_h}(y_n)+ \cdot) - f(\EE^u \psi_{r_h}(y_n)+ \cdot)    \|_{\mathcal{B}_p^{\gamma-1}} (r_h-u)^{\frac{1}{4}(\gamma-\frac{1}{p}-1)} dy dr 
\end{align*}
and with the fact that $2(r_h-u) \ge (r-u)$ and $\gamma-1/p\geq -1$, one gets
\begin{align}\label{eq:J2Prop56}
\| J_2 \|_{L^\infty}
& \leq C\, \|f \|_{\mathcal{B}_p^{\gamma}} ( \psi )_{\mathcal{C}^{\frac{1}{2}+\eta,0}_{[S,T],x} L^{m,\infty}} \, \int_{u+2h}^t (r-u)^{-\frac{1}{2}} (r-s)^{\frac{1}{2}+\eta}\, dr \nonumber\\
& \leq C\, \|f \|_{\mathcal{B}_p^{\gamma}} ( \psi )_{\mathcal{C}^{\frac{1}{2}+\eta,0}_{[S,T],x} L^m} (t-s)^{1+\eta} .
\end{align}
In view of the inequalities \eqref{eq:J0Prop56}, \eqref{eq:J1Prop56} and \eqref{eq:J2Prop56}, we deduce \ref{item56(1)}.

\paragraph{Proof of \ref{item56(2)}:} Write
\begin{equation*}
\Big\| \Big( \EE^S | \delta A_{{s},u,{t}} |^m \Big)^{\frac{1}{m}} \Big\|_{L^\infty} \leq \Big\| \Big( \EE^S | \delta A_{{s},{t}} |^m \Big)^{\frac{1}{m}} \Big\|_{L^\infty} + \Big\| \Big( \EE^S | \delta A_{{s},u} |^m \Big)^{\frac{1}{m}} \Big\|_{L^\infty} + \Big\| \Big( \EE^S | \delta A_{u,{t}} |^m \Big)^{\frac{1}{m}} \Big\|_{L^\infty} .
\end{equation*}
To obtain \ref{item56(2)}, it remains to apply Lemma~\ref{lem:bound-Khn} to each of the above terms, for $\EE^s \psi_{r}$, $\EE^s \psi_{r}$ again and $\EE^u \psi_{r}$ in place of $\psi_r$. 

\paragraph{Proof of \ref{item56(3)}:} The proof is identical to the last point in the proof of Proposition~\ref{prop:ssl-o-on-2}.
\end{proof}

\subsection{Regularisation properties of the discrete Ornstein-Uhlenbeck process on increments}\label{ec:reg-On-3}

Similarly to Section~\ref{sec:reg-O-2}, we show here regularisation properties of the process $O^n$ on quantities of the form
\begin{align*}
\int_0^t \int_{\T} p_{t-r}(x,y) \left( f(\phi_r(y)+O^n_{r_{h}}(y))- f(\phi_{r_{h}}(y_n)+O^n_{r_{h}}(y_{n})) \right) dr.
\end{align*} 
We proceed again with a $2$-step approach, first by forcing $\phi$ to be $\mathcal{F}_{S}$-measurable, for some fixed $S$, then proceeding with the general case of an $\mathbb{F}$-adapted process $\phi$.
The main result of this subsection, Proposition~\ref{prop:ssl-onh-final}, will be used in Section~\ref{sec:reg-On}.

For the sake of readability, we introduce the projection operator $\pi_{n}$ on $\Lambda_{h}$ and $\T_{n}$ as follows: for $\phi : [0,1] \times \T \times \Omega \rightarrow \mathbb{R}$,
\begin{equation}\label{eq:defprojpi}
(\pi_{n} \phi)_{r}(y) := \phi_{r_{h}}(y_{n}), \quad r\in [0,1], \ y\in \T.
\end{equation}

\begin{lemma}\label{sec:ssl-onh}
Let $m \in [2,+\infty)$ and $\varepsilon \in (0,\frac{1}{2})$, and assume that $\gamma-1/p \ge -1$. There exists a constant $C$ such that for any $0 \leq S < T\leq 1$, any $(s,t)\in \Delta_{[S,T]}$, any $f  \in \mathcal{C}_b^\infty (\R, \R)\cap \mathcal{B}^\gamma_{p}$, any $n\in \N^*$ and any $\phi : [0,1] \times \T \times \Omega \rightarrow \mathbb{R}$ such that for all $x \in \mathbb{T}$ and $r \in [0,1]$, $\phi_r(x)$ is $\mathcal{F}_S$-measurable, there is
\begin{align*}%
& \sup_{x \in \mathbb{T}} \Big\| \int_s^t \int_{\T} p_{T-r}(x,y) \Big(f(\phi_r(y)+O^n_{r_h}(y)) - f(\phi_{r_h}(y_n) + O_{r_h}^n(y_n)) \Big)\, dy dr \Big\|_{L^m}  \\ 
& \quad \leq C \Big( \| f \|_{\infty} n^{-1+2\varepsilon} + n^{2\varepsilon} \| f \|_{\mathcal{B}^{\gamma}_p} \left( \| \phi-\pi_n\phi \|_{L^{\infty,\infty}_{[h,1],x} L^m} + n^{-\frac{1}{2}} \right) \Big) (t-s)^{\frac{1}{2}+\varepsilon} .
\end{align*}
\end{lemma}

\begin{proof}
Fix $x \in \mathbb{T}$ and $0 \leq S < T \leq 1$. We will apply the Stochastic Sewing Lemma (Lemma~\ref{lem:SSL}). For $(s,t) \in \Delta_{[S,T]}$, define
\begin{align*}
    \mathcal{A}_t:=\int_S^t p_{T-r}(x,y) \Big( f(\phi_r(y)+O^n_{r_h}(y)) - f(\phi_{r_h}(y_n) + O_{r_h}^n(y_n)) \Big)  \, dy dr  
    ~~\text{and}~~ A_{s,t}:=\EE^s[ \mathcal{A}_t - \mathcal{A}_s].
\end{align*}
Observe that $\EE^s[\delta A_{s,u,t}]=0$, so \eqref{eq:condsew1} immediately holds. 
To show \eqref{eq:condsew2}, we shall prove that
\begin{align} \label{(4.8)-critic-}
    \|\delta A_{s,u,t}\|_{L^m}\leq C \Big( \| f \|_{\infty} n^{-1+2\varepsilon} + n^{2\varepsilon} \| f \|_{\mathcal{B}^{\gamma}_p} \left( \| \phi-\pi_n\phi \|_{L^{\infty,\infty}_{[h,1],x} L^m} + n^{-\frac{1}{2}} \right) \Big) \, (t-s)^{\frac{1}{2}+\varepsilon} .
\end{align} 
For $u=(s+t)/2$, the triangle inequality and the conditional Jensen inequality give that
\begin{align}\label{eq:backtoAst-}
    \|\delta A_{s,u,t}\|_{L^m}
        &\leq \int_u^t \int_{\T} p_{T-r}(x,y) \Big(\left\|\EE^s \left[f(\phi_r(y)+O^n_{r_h}(y)) - f(\phi_{r_h}(y_n) + O_{r_h}^n(y_n)) \right]  \right\|_{L^m} \nonumber \\
         & \quad +\|\EE^u [f(\phi_r(y)+O^n_{r_h}(y)) - f(\phi_{r_h}(y_n) + O_{r_h}^n(y_n)) ]  \|_{L^m}\Big) dy  dr \nonumber \\
            &\leq 2 \int_u^t \int_{\T} p_{T-r}(x,y) \left\|\EE^u  \left[f(\phi_r(y)+O^n_{r_h}(y)) - f(\phi_{r_h}(y_n) + O_{r_h}^n(y_n)) \right] \right\|_{L^m} \, dy dr .
\end{align}
\paragraph{The case $t-u \leq 2h$.} Using $h=c(2n)^{-2}$, we simply have that
\begin{align}\label{eq:Ast----<h}
  \|\delta A_{s,u,t}\|_{L^m}  \leq C \| f \|_{\infty} (t-s) &\leq C \| f \|_{\infty} h^{\frac{1}{2}-\varepsilon} (t-s)^{\frac{1}{2}+\varepsilon} \nonumber\\
& \leq C \| f \|_{\infty} n^{-1+2\varepsilon} (t-s)^{\frac{1}{2}+\varepsilon}  .
\end{align}
\paragraph{The case $t-u \ge 2h$.}We split the integral in the right-hand side of \eqref{eq:backtoAst-} first between $u$ and $u+2h$, and then between $u+2h$ and $t$. We denote the two respective integrals $J_1$ and $J_2$. For $J_1$, we have the simple bound $J_{1}\leq C \| f \|_\infty h \leq C \| f \|_{\infty} h^{\frac{1}{2}-\varepsilon} (t-s)^{\frac{1}{2}+\varepsilon}$. So using $h=c(2n)^{-2}$, $J_{1}\leq C \| f \|_{\infty} n^{-1+2\varepsilon} (t-s)^{\frac{1}{2}+\varepsilon}$. As for $J_2$, for $r \in [0,1]$ and $y \in \T$, consider the $\mathscr{B}(\R)\otimes \mathcal{F}_{S}$-measurable random function $F_{\rho,\zeta}(\cdot) := f(\phi_\rho(\zeta)+\cdot)$ for $\rho\in \{r,r_{h}\}, \zeta\in \{y,y_{n}\}$.
Then we have
\begin{align*}
J_2 & \leq C \int_{u+2h}^t \int_{\T} p_{T-r}(x,y) \left\|\EE^u  \left[F_{r,y}(O^n_{r_h}(y)) - F_{r_{h},y_{n}}(O_{r_h}^n(y_n)) \right] \right\|_{L^m} \, dy dr .
\end{align*}
Using Lemma~\ref{lem:reg-On}, we get
\begin{align*}
 J_2 & \leq C  \int_{u+2h}^t \int_{\T} p_{T-r}(x,y) \Big\| G_{Q^n(r_h-u)} F_{r,y}(\widehat{O}^n_{u,r_h}(y)) - G_{Q^n(r_h-u)} F_{r_{h},y_{n}}(\widehat{O}_{u,r_h}^n(y_n)) \Big\|_{L^m} \, dy dr \\
 & \leq C  \int_{u+2h}^t \int_{\T} p_{T-r}(x,y) \Bigg( \Big\| G_{Q^n(r_h-u)} F_{r,y}(\widehat{O}^n_{u,r_h}(y)) - G_{Q^n(r_h-u)} F_{r,y}(\widehat{O}_{u,r_h}^n(y_n)) \Big\|_{L^m} \\ 
 & \quad  + \Big\| G_{Q^n(r_h-u)} (F_{r,y} - F_{r_{h},y_{n}})(\widehat{O}^n_{u,r_h}(y_n)) \Big\|_{L^m} \, \Bigg) dy dr .
 \end{align*}
 Applying Lemma~\ref{eq:reg-S}$(iii)$, Lemma~\ref{lem:besov-spaces}$(i)$ and Lemma~\ref{lem:bound-Qn}$(iii)$ gives 
 \begin{align*}
 &\|G_{Q^n(r_h-u)} F_{r,y}(\widehat{O}^n_{u,r_h}(y)) - G_{Q^n(r_h-u)} F_{r,y}(\widehat{O}_{u,r_h}^n(y_n)) \|_{L^m} \\
 &\quad \leq C \| F_{r,y}\|_{\mathcal{B}^\gamma_{p}} Q^n(r_h-u)^{\frac{1}{2}(\gamma-\frac{1}{p}-1)} \, \|\widehat{O}^n_{u,r_h}(y)) - \widehat{O}_{u,r_h}^n(y_n)) \|_{L^m}\\
 &\quad \leq C \| f\|_{\mathcal{B}^\gamma_{p}} (r_h-u)^{\frac{1}{4}(\gamma-\frac{1}{p}-1)} \, \|\widehat{O}^n_{u,r_h}(y)) - \widehat{O}_{u,r_h}^n(y_n)) \|_{L^m}.
 \end{align*}
 Applying Lemma~\ref{eq:reg-S}$(i)$, Lemma~\ref{lem:besov-spaces}$(ii)$ and Lemma~\ref{lem:bound-Qn}$(iii)$ gives 
 \begin{align*}
 &\|G_{Q^n(r_h-u)} (F_{r,y} - F_{r_{h},y_{n}})(\widehat{O}^n_{u,r_h}(y_n))\|_{L^m} \\
 &\quad \leq C \| F_{r,y} - F_{r_{h},y_{n}}\|_{\mathcal{B}^{\gamma-1}_{p}}\, Q^n(r_h-u)^{\frac{1}{2}(\gamma-\frac{1}{p}-1)} \\
 &\quad \leq C \| f\|_{\mathcal{B}^\gamma_{p}} (r_h-u)^{\frac{1}{4}(\gamma-\frac{1}{p}-1)} \, \|\phi_{r}(y) - \phi_{r_h}(y_n)) \|_{L^m}.
 \end{align*}
Now use the two previous inequalities, Lemma~\ref{lem:diffhatOn}, $(r_h-u)\ge (r-u)/2$ and $\gamma-1/p\geq -1$ to get that
\begin{align*}
J_2 
& \leq C \| f \|_{\mathcal{B}^{\gamma}_p} \left( \| \phi-\pi_{n}\phi \|_{L^{\infty,\infty}_{[h,1],x} L^m} + n^{-\frac{1}{2}} \right) \int_{u+2h}^t \int_{\T} p_{T-r}(x,y)\, (r-u)^{\frac{1}{4}(\gamma-1-\frac{1}{p})}\, dy d r \\
& \leq C \| f \|_{\mathcal{B}^{\gamma}_p} \left( \| \phi-\pi_n\phi \|_{L^{\infty,\infty}_{[h,1],x} L^m} +n^{-\frac{1}{2}} \right) (t-s)^{\frac{1}{2}} .
\end{align*}
Using that $t-s \geq 2h$ and $h=c(2n)^{-2}$, one gets $(t-s)^{\frac{1}{2}} \leq C n^{2\varepsilon} (t-s)^{\frac{1}{2}+\varepsilon}$.
Hence when $t-u \ge 2h$, the bounds on $J_1$ and $J_2$ yield
\begin{align}\label{eq:Ast---->h}
 \|\delta A_{s,u,t}\|_{L^m} 
 & \leq C \Big( \| f \|_{\infty} n^{-1+2\varepsilon} +n^{2\varepsilon} \| f \|_{\mathcal{B}^{\gamma}_p} \left( \| \phi-\pi_n\phi \|_{L^{\infty,\infty}_{[h,1],x} L^m} +n^{-\frac{1}{2}} \right) \Big) (t-s)^{\frac{1}{2}+\varepsilon} .
\end{align}
Combining \eqref{eq:Ast----<h} and \eqref{eq:Ast---->h}, we deduce \eqref{(4.8)-critic-}.

\paragraph{Convergence in probability.} We omit the proof of \eqref{eq:convAt} as it follows the same lines as the paragraph ``Convergence in probability" in the proof of Lemma \ref{lem:bound-Khn}.

\smallskip

Applying Lemma~\ref{lem:SSL}, it thus comes
\begin{align*}%
   \Big\| \Big( \EE^S | \mathcal{A}_t-\mathcal{A}_s|^m \Big)^{\frac{1}{m}} \Big\|_{L^m} 
   \leq \|A_{s,t}\|_{L^m} + C \Big( \| f \|_{\infty} n^{-1+2\varepsilon} + n^{2\varepsilon}\| f \|_{\mathcal{B}^{\gamma}_p} \left( \| \phi-\pi_n\phi \|_{L^{\infty,\infty}_{[h,1],x} L^m}  + n^{-\frac{1}{2}} \right) \Big) (t-s)^{\frac{1}{2}+\varepsilon}  .
\end{align*}
Repeating the same arguments used to obtain \eqref{(4.8)-critic-} produces the same bound on $\|A_{s,t}\|_{L^m}$, namely:
\begin{align*}
  \|A_{s,t}\|_{L^m} \leq C \Big( \| f \|_{\infty} n^{-1+2\varepsilon} + n^{2\varepsilon}\| f \|_{\mathcal{B}^{\gamma}_p} \left( \| \phi-\pi_n\phi \|_{L^{\infty,\infty}_{[h,1],x} L^m}  + n^{-\frac{1}{2}} \right) \Big) (t-s)^{\frac{1}{2}+\varepsilon} ,
\end{align*}
from which the conclusion follows.
\end{proof}

Next, we extend the previous lemma to the case when the process $\phi$ is adapted.
\begin{prop}\label{prop:ssl-onh-final}
Let $m \in [2,+\infty)$, $\varepsilon \in (0, \frac{1}{2})$ and assume that $\gamma-1/p \ge -1$. There exists a constant $C$ such that for any $0 \leq S < T\leq 1$, any $(s,t)\in \Delta_{[S,T]}$, any $f \in \mathcal{C}_b^\infty (\R, \R)\cap \mathcal{B}^\gamma_{p}$, any $n\in \N^*$ and any $\mathbb{F}$-adapted process $\phi : [0,1] \times \T \times \Omega \rightarrow \mathbb{R}$, there is
\begin{align}\label{eq:ssl-onh-final}
\begin{split}
& \sup_{x \in \mathbb{T}} \Big\| \int_s^t \int_{\T} p_{T-r}(x,y) \left(f(\phi_r(y)+O^n_{r_h}(y)) - f(\phi_{r_h}(y_n) + O_{r_h}^n(y_n)) \right) dy dr \Big\|_{L^m}  \\
&\quad \leq C \Big( \| f \|_{\infty} n^{-1+2\varepsilon} + n^{2\varepsilon}\| f \|_{\mathcal{B}^{\gamma}_p} ( \| \phi-\pi_n\phi \|_{L^{\infty,\infty}_{[h,1],x} L^m} +n^{-\frac{1}{2}} ) \Big) (t-s)^{\frac{1}{2}+\varepsilon}  \\ 
&\quad\quad +  C \Bigg( (\phi)_{\mathcal{C}^{1,0}_{[0,1],x} L^{1,\infty}} \Big( \| f \|_{\mathcal{B}_p^\gamma} ( \| \phi- \pi_{n}\phi\|_{L^{\infty,\infty}_{[h,1],x} L^m}+n^{-\frac{1}{2}} ) +  \| f \|_{\mathcal{C}^1} n^{-2+2\varepsilon} \Big) \\ 
& \hspace{2cm} +  \| f \|_{\mathcal{B}_p^\gamma} \big(\phi-\pi_n\phi\big)_{\mathcal{C}^{\frac{1}{2}+\varepsilon,0}_{[0,1],x} L^{1,m}} \Bigg) (t-s)^{1+\varepsilon \wedge \frac{1}{4}} ,
\end{split}
\end{align}
where we recall that $\pi_{n}\phi$ is defined in \eqref{eq:defprojpi}.
\end{prop}

\begin{proof}
Fix $x \in \T$ and $0 \leq S < T\leq 1$. We will apply again Lemma~\ref{lem:SSL}. Define for $(s,t) \in \Delta_{[S,T]}$,
\begin{align*}
    &\mathcal{A}_t:=\int_S^t p_{T-r}(x,y) \left( f(\phi_{r}(y)+O^n_{r_h}(y)) - f(\phi_{r_h}(y_n) + O_{r_h}^n(y_n)) \right) dy dr \\
    &A_{s,t}:=\int_s^t p_{T-r}(x,y) \left( f(\EE^s [\phi_r(y)]+O^n_{r_h}(y)) - f(\EE^s [\phi_{r_h}(y_n)] + O_{r_h}^n(y_n)) \right) dy dr .
\end{align*}
Assume without loss of generality that the quantities in the r.h.s of \eqref{eq:ssl-onh-final} are finite, otherwise the result is obvious. We check the assumptions in order to apply Lemma \ref{lem:SSL} with $q=m$. To show that \eqref{eq:condsew1} and \eqref{eq:condsew2} hold true with $\varepsilon_1=\varepsilon \wedge 1/4$ and $\varepsilon_2=\varepsilon$, we show that there is a constant $C>0$ which does not depend on $s,t,S$ and $T$ such that
\begin{enumerate}[label=(\roman*)]
\item\label{en:(1bx)}  $\| \delta A_{s,u,t} \|_{L^m}  \leq  C \Big( \| f \|_{\infty} n^{-1+2\varepsilon} + n^{2\varepsilon} \| f \|_{\mathcal{B}^{\gamma}_p} \left( \| \phi-\pi_n\phi \|_{L^{\infty,\infty}_{[h,1],x} L^m} +n^{-\frac{1}{2}} \right) \Big) (t-s)^{\frac{1}{2}+\varepsilon} $;

\item\label{en:(2bx)} $\|\EE^s \delta A_{s,u,t}\|_{L^m} \\ 
\leq C \Big( (\phi)_{\mathcal{C}^{1,0}_{[0,1],x} L^{1,\infty}} \Big( \| f \|_{\mathcal{B}_p^\gamma} \| \phi- \pi_n\phi\|_{L^{\infty,\infty}_{[h,1],x} L^m} +  \| f \|_{\mathcal{C}^1} n^{-2+2\varepsilon} \Big) +  \| f \|_{\mathcal{B}_p^\gamma} \left(\phi-\pi_n\phi\right)_{\mathcal{C}^{\frac{1}{2}+\varepsilon,0}_{[0,1],x} L^{1,m}} \Big)\\
 \times (t-s)^{1+\varepsilon \wedge \frac{1}{4}} $; 

\item\label{en:(3bx)}  If \ref{en:(1bx)} and \ref{en:(2bx)} are satisfied, \eqref{eq:convAt} gives the 
convergence in probability of $\sum_{i=1}^{N_k-1} A_{t^k_i,t^k_{i+1}}$ on any sequence $\Pi_k=\{t_i^k\}_{i=1}^{N_k}$ of partitions of $[S,t]$ with a mesh that converges to $0$. We will show that the limit is $\mathcal{A}_{t}$.
\end{enumerate}

If \ref{en:(1bx)}, \ref{en:(2bx)} and \ref{en:(3bx)} are satisfied,  then Lemma \ref{lem:SSL} gives
\begin{align*}
&\Big\|  \int_s^t \int_{\T} p_{T-r}(x,y) \left( f(\phi_r(y)+O^n_{r_h}(y)) - f(\phi_{r_h}(y) + O_{r_h}^n(y_n)) \right) dy dr\Big\|_{L^m}  \nonumber \\
& \leq  \| A_{s,t} \|_{L^m} + C \Big( \| f \|_{\infty} n^{-1+2\varepsilon} + n^{2\varepsilon}\| f \|_{\mathcal{B}^{\gamma}_p} ( \| \phi-\pi_n\phi \|_{L^{\infty,\infty}_{[h,1],x} L^m} +n^{-\frac{1}{2}} ) \Big) (t-s)^{\frac{1}{2}+\varepsilon}  \nonumber \\ 
& \quad + C \Big( (\phi)_{\mathcal{C}^{1,0}_{[0,1],x} L^{1,\infty}} \Big( \| f \|_{\mathcal{B}_p^\gamma} \left( \| \phi- \pi_n\phi \|_{L^{\infty,\infty}_{[h,1],x} L^m} + n^{-\frac{1}{2}} \right) +  \| f \|_{\mathcal{C}^1} n^{-2+2\varepsilon} \Big) \\ 
& \quad +  \| f \|_{\mathcal{B}_p^\gamma} (\phi-\pi_n\phi)_{\mathcal{C}^{\frac{1}{2}+\varepsilon,0}_{[0,1],x} L^{1,m}} \Big) (t-s)^{1+\varepsilon \wedge \frac{1}{4}} .
\end{align*}

To bound $\| A_{{s},{t}} \|_{L^m}$, we apply  Lemma \ref{sec:ssl-onh} with $\EE^s \phi_r$ in place of $\phi_r$ to get that
\begin{align}\label{eq:bound-Lp-norm-Ax}
\| A_{s,t} \|_{L^m} 
& \leq C \Big( \| f \|_{\infty} n^{-1+2\varepsilon} + n^{2\varepsilon}\| f \|_{\mathcal{B}^{\gamma}_p} ( \| \phi-\pi_n\phi \|_{L^{\infty,\infty}_{[h,1],x} L^m} +n^{-\frac{1}{2}} ) \Big) (t-s)^{\frac{1}{2}+\varepsilon} . 
\end{align}
Thus, taking the supremum over $x \in \T$, we get \eqref{eq:ssl-o-on-2}.

~

Let us now verify that \ref{en:(1bx)}, \ref{en:(2bx)} and \ref{en:(3bx)} are satisfied.

\paragraph{Proof of \ref{en:(1bx)}:} For $u=(s+t)/2$, we have 
$ \|\delta A_{s,u,t}\|_{L^m} \leq  \|A_{s,t}\|_{L^m} +\|A_{s,u}\|_{L^m}  + \|A_{u,t}\|_{L^m}$.
Using \eqref{eq:bound-Lp-norm-Ax} for each term, it follows that 
\begin{align*}
 \|\delta A_{s,u,t}\|_{L^m} & \leq C \Big( \| f \|_{\infty} n^{-1+2\varepsilon} + n^{2\varepsilon}\| f \|_{\mathcal{B}^{\gamma}_p} \left( \| \phi-\pi_n\phi \|_{L^{\infty,\infty}_{[h,1],x} L^m} +n^{-\frac{1}{2}} \right) \Big) (t-s)^{\frac{1}{2}+\varepsilon}   .
\end{align*}

\paragraph{Proof of \ref{en:(2bx)}:}
Let $u=(s+t)/2$. We have
\begin{equation}\label{eq:deltaAOn}
\begin{split}
 \delta A_{s,u,t}& =  \int_u^t \int_{\T} p_{T-r}(x,y) \Big( f(\EE^s [\phi_r(y)]+O^n_{r_{h}}(y)) - f(\EE^u [\phi_r(y)]+O^n_{r_{h}}(y)) \\ 
 & \quad\quad - f(\EE^s [\phi_{r_{h}}(y_n)]+O^n_{r_{h}}(y_n)) + f(\EE^u [\phi_{r_{h}}(y_n)]+O^n_{r_{h}}(y_n)) \Big)\, dy dr.
\end{split}
\end{equation}
If $t-u \leq 2h$, then recalling the pseudo-norm defined in \eqref{def:holder-norm-phi}, the following simple bound holds:
\begin{align}\label{eq:Ast---<h}
  \| \EE^s  \delta A_{s,u,t}\|_{L^m}  
  \leq 2 \| f \|_{\mathcal{C}^1}   \big( \phi \big)_{\mathcal{C}^{1,0}_{[0,1],x} L^{1,m}} (t-s)^2 
  &\leq 2 \| f \|_{\mathcal{C}^1}   \big( \phi \big)_{\mathcal{C}^{1,0}_{[0,1],x} L^{1,\infty}} h^{1-\varepsilon} (t-s)^{1+\varepsilon} \nonumber\\
  &\leq C \| f \|_{\mathcal{C}^1} \big( \phi \big)_{\mathcal{C}^{1,0}_{[0,1],x} L^{1,\infty}} n^{-2+2\varepsilon} (t-s)^{1+\varepsilon},
\end{align}
using  $h=c(2n)^{-2}$ in the last inequality. 

Now if $t-u \ge 2h$, split the expression of $\delta A_{s,u,t}$ from \eqref{eq:deltaAOn} as 
the sum of an integral between $u$ and $u+2h$ and of an integral between $u+2h$ and $t$. Denote the two respective integrals by $J_1$ and $J_2$.

For $J_{1}$, proceed as in the case $t-u \leq 2h$ to get that
\begin{align}\label{eq:J1-On}
\|\EE^s J_{1}\|_{L^m} \leq C \| f \|_{\mathcal{C}^1}   \big( \phi \big)_{\mathcal{C}^{1,0}_{[0,1],x} L^{1,m}} h^2 \leq C \| f \|_{\mathcal{C}^1} \big( \phi \big)_{\mathcal{C}^{1,0}_{[0,1],x} L^{1,\infty}} n^{-2+2\varepsilon} (t-s)^{1+\varepsilon}.
\end{align}

For $J_{2}$, the tower property and the triangle inequality give
\begin{align*}
    \| \EE^s J_{2}\|_{L^m}
        &\leq  \int_{u+2h}^t \int_{\T} p_{T-r}(x,y) \Big\| \EE^s \EE^u \Big[f(\EE^s [\phi_r(y)]+O^n_{r_{h}}(y)) -f(\EE^u [\phi_r(y)]+O^n_{r_{h}}(y))  \\ 
        & \quad -f(\EE^s[ \phi_{r_{h}}(y_n)]+ O^n_{r_{h}}(y_n)) + f(\EE^u [\phi_{r_{h}}(y_n)] + O^n_{r_h}(y_n)) \Big]  \Big\|_{L^m}  dy dr .
\end{align*}
Now introduce the pivot term $ f(\EE^s[ \phi_{r_h}(y)]+\EE^u [\phi_r(y)] -\EE^s [\phi_r(y)] + O_{r_h}^n(y_n))$ in the previous expression:
\begin{align*}
\| \EE^s J_2 \|_{L^m}  
& \leq   \int_{u+2h}^t \int_{\T} p_{T-r}(x,y)  \Big\| \EE^s \EE^u  \Big[f(\EE^s [\phi_r(y)]+O^n_{r_h}(y)) -f(\EE^u [\phi_r(y)]+O^n_{r_h}(y) ) \\ 
& \quad -f(\EE^s [\phi_{r_h}(y_n)] +O^n_{r_h}(y_n)) + f(\EE^s [\phi_{r_h}(y_n)]+\EE^u [\phi_r(y)] -\EE^s [\phi_r(y)]+ O^n_{r_h} (y_n))  \Big]  \Big\|_{L^m}   \,  dy dr \\ 
& \quad +  \int_{u+2h}^t \int_{\T} p_{T-r}(x,y) \Big\| \EE^s \EE^u  \Big[f(\EE^u [\phi_{r_h}(y_n)]+ O^n_{r_h}(y_n)) - f(\EE^s [\phi_{r_h}(y_n)]+\EE^u [\phi_r(y)] \\ 
& \quad -\EE^s [\phi_r(y)]+ O^n_{r_h}(y_n)) \Big] \Big\|_{L^m} \,  dy dr  \\
& =: J_{21} + J_{22}. 
\end{align*}
For $r \in [0,1]$ and $y \in \T$, introduce the $\mathscr{B}(\R)\otimes \mathcal{F}_{u}$-measurable random functions $F_{s,r,u,y}$ and $\tilde{F}_{s,r,u,y,n}$ defined as follows:
\begin{align*}
F_{s,r,u,y} (\cdot) & := f(\EE^s [\phi_r(y)]+\cdot) -f(\EE^u [\phi_r(y)]+\cdot) \\ 
\tilde{F}_{s,r,u,y,n} (\cdot) &  := f(\EE^s [\phi_{r_h}(y_n)] +\cdot) - f(\EE^s [\phi_{r_h}(y_n)]+\EE^u [\phi_r(y)] -\EE^s [\phi_r(y)]+ \cdot).
\end{align*}
Following the same steps used in the proof of Lemma \ref{sec:ssl-onh} for $F_{\rho,\zeta}$ and using the same arguments (Lemma~\ref{lem:reg-On} and Lemma~\ref{eq:reg-S}), we deduce that
\begin{align*}
J_{21} & \leq C   \int_{u+2h}^t \int_{\T} p_{T-r}(x,y) \Big( \Big\| \EE^s  \|F_{s,r,u,y} - \tilde{F}_{s,r,u,y,n} \|_{\mathcal{B}_p^{\gamma-2}} \Big\|_{L^m} \\ 
& \quad +\left\| \EE^s \|  F_{s,r,u,y} \|_{\mathcal{B}_p^{\gamma-1}}  | \widehat{O}^n_{u,r_h}(y)-\widehat{O}^n_{u,r_h}(y_n) | \right\|_{L^m}  \Big) \Big] (r_h-u)^{\frac{1}{4}(\gamma-2-\frac{1}{p})} \, dy dr .
\end{align*}
Using Lemma~\ref{lem:besov-spaces}$(ii)$ and $(iii)$, there is
\begin{align*}
J_{21} & \leq C \| f \|_{\mathcal{B}_p^\gamma}   \int_{u+2h}^t \int_{\T} p_{T-r}(x,y) \Big( \| \EE^s | \EE^s  \phi_r(y)- \EE^u \phi_r(y)| \, | \EE^s \phi_r(y) -\EE^s \phi_{r_h}(y_n)  \|_{L^m} \\ & \quad +\|  \EE^s | \EE^s  \phi_r(y)- \EE^u \phi_r(y)|  \, | \widehat{O}^n_{u,r_h}(y)-\widehat{O}^n_{u,r_h}(y_n) | \|_{L^m}  \Big) \Big] (r_h-u)^{\frac{1}{4}(\gamma-2-\frac{1}{p})} \, dy dr  .
\end{align*}
It comes by Lemma~\ref{lem:diffhatOn} and $(r_h-u) \ge (r-u)/2$ that
\begin{align}\label{eq:boundJ21}
J_{21} & \leq   C \| f \|_{\mathcal{B}_p^\gamma} \int_{u+2h}^t \int_{\T} p_{T-r}(x,y) (\phi)_{\mathcal{C}^{1,0}_{[0,1],x} L^{1,\infty}} \left( \| \phi- \pi_n\phi \|_{L^{\infty,\infty}_{[h,1],x} L^m}+n^{-\frac{1}{2}} \right) (r-s)  (r-u)^{\frac{1}{4}(\gamma-2-\frac{1}{p})}  \, dy dr .
\end{align}
For $J_{22}$, we apply Lemma~\ref{lem:reg-On} to get
\begin{align*}
J_{22} & \leq C \int_{u+2h}^t \int_{\T} p_{T-r}(x,y) \Big\| \EE^s \Big[ \| f\left(\EE^u [\phi_{r_h}(y_n)]+ O^n_{r_h}(y_n)\right) \\
&\hspace{2cm}- f\left(\EE^s [\phi_{r_h}(y_n)]+\EE^u [\phi_r(y)] -\EE^s [\phi_r(y)]+ O^n_{r_h}(y_n)\right) \Big] \Big\|_{\mathcal{B}_p^{\gamma-1}} \|_{L^m} (r_h-u)^{\frac{1}{4}(\gamma-1-\frac{1}{p})} \, dy dr .
\end{align*}
It comes by applying Lemma~\ref{lem:besov-spaces}$(ii)$ that
\begin{align}
J_{22} &  \leq  C \| f \|_{\mathcal{B}_p^\gamma}  \int_{u+2h}^t \int_{\T} p_{T-r}(x,y) \Big\| \EE^s\Big[ |\EE^u \phi_{r_h}(y_n)  - \EE^s \phi_{r_h}(y_{n}) - \EE^u \phi_r(y) + \EE^s \phi_r(y) | \Big] \Big\|_{L^m} (r_h-u)^{\frac{1}{4}(\gamma-1-\frac{1}{p})} \Big) \, dy dr \nonumber \\ 
& \leq C \| f \|_{\mathcal{B}_p^\gamma}  \int_{u+2h}^t \int_{\T} p_{T-r}(x,y) (\phi-\pi_n\phi)_{\mathcal{C}^{\frac{1}{2}+\varepsilon,0}_{[0,1],x} L^{1,m}} (r-s)^{\frac{1}{2}+\varepsilon}  (r-u)^{\frac{1}{4}(\gamma-1-\frac{1}{p})} \Big) \, dy dr .\label{eq:boundJ22}
\end{align}
Combining \eqref{eq:boundJ21} and \eqref{eq:boundJ22}, we get
\begin{align}\label{eq:Ast--->h}
\| \EE^s J_2 \|_{L^m}  & \leq C  \| f \|_{\mathcal{B}_p^\gamma} \Big( (\phi)_{\mathcal{C}^{1,0}_{[0,1],x} L^{1,\infty}} \left( \| \phi- \pi_n\phi \|_{L^{\infty,\infty}_{[h,1],x} L^m} +n^{-\frac{1}{2}} \right) + (\phi-\pi_n\phi )_{\mathcal{C}^{\frac{1}{2}+\varepsilon,0}_{[0,1],x} L^{1,m}} \Big) (t-s)^{1+\varepsilon \wedge \frac{1}{4}} .
\end{align}
Thus putting together \eqref{eq:Ast---<h}, \eqref{eq:J1-On} and \eqref{eq:Ast--->h}, we deduce that for all $s \leq t$,
\begin{align*}%
& \|\EE^s \delta A_{s,u,t}\|_{L^m} \\
 & \leq C \Big( (\phi)_{\mathcal{C}^{1,0}_{[0,1],x} L^{1,\infty}} \Big( \| f \|_{\mathcal{B}_p^\gamma} \left(\| \phi- \pi_n\phi\|_{L^{\infty,\infty}_{[h,1],x} L^m}+n^{-\frac{1}{2}} \right) +  \| f \|_{\mathcal{C}^1} n^{-2+2\varepsilon} \Big)  \\ 
 & \quad +  \| f \|_{\mathcal{B}_p^\gamma} (\phi-\pi_n\phi )_{\mathcal{C}^{\frac{1}{2}+\varepsilon,0}_{[0,1],x} L^{1,m}} \Big) (t-s)^{1+\varepsilon \wedge \frac{1}{4}} ,
\end{align*}
which proves \ref{en:(2bx)}.

\paragraph{Proof of \ref{en:(3bx)}:}
Let $t\in [S,T]$. Let $(\Pi_k)_{k \in \mathbb{N}}$, with $\Pi_k=\{t_i^k\}_{i=1}^{N_k}$, be a 
sequence of partitions of $[S,t]$ with mesh size that converges to zero. Using the $\mathcal{C}^1$ norm of $f$, we get
\begin{align*}
    \Big\|\mathcal{A}_t-\sum_{i=1}^{N_{k}-1} A_{t^k_i,t^k_{i+1}} \Big\|_{L^1}
    & \leq\sum_{i=1}^{N_{k}-1} \int_{t^k_i}^{t^k_{i+1}} \int_{\T} p_{T-r}(x,y) \Big( \|f(\phi_r(y)+O^n_{r_h}(y))-f(\EE^{r_i^k}[ \phi_r(y)]+O^n_{r_h}(y)) \|_{L^1} \\ 
    & \quad + \|f(\phi_{r_h}(y_n)+O^n_{r_h}(y_n))- f(\EE^{t_i^k} [\phi_{r_h}(y_n)]+O^n_{r_h}(y_n)) \|_{L^1}\Big)\,  dy dr\\
    & \leq  2 \sum_{i=1}^{N_{k}-1} \int_{t^k_i}^{t^k_{i+1}} \int_{\T} p_{T-r}(x,y) \| f \|_{\mathcal{C}^1}  \big( \phi \big)_{\mathcal{C}^{1,0}_{[0,1],x} L^1}  (r-t_i^k) \, dy dr\\
    & \leq C \| f \|_{\mathcal{C}^1}  \big( \phi \big)_{\mathcal{C}^{1,0}_{[0,1],x} L^{1,\infty}} | \Pi_k | (t-S) \longrightarrow 0.
\end{align*}
\end{proof}

\subsection{H\"older regularity of $\mathcal{E}^{3,n,k}$}\label{app:E3}

In this section, we prove an upper bound on $\mathcal{E}^{3,n,k}$ (see \eqref{def:E4}). The following lemma is similar to \cite[Eq. (3.34)]{butkovsky2021optimal}, the only difference being the powers in $n$ and $(t-s)$ in the upper bound.
\begin{lemma}\label{lem:bound-E4}
Let $m \ge 2$ and $\varepsilon \in (0,\frac{1}{2})$. There exists a constant $C$ such that for all $0 \leq s \leq t \leq 1$ and $n\in \N^*, k \in \mathbb{N}$, one has the bound
\begin{align*}
 \sup_{x \in \mathbb{T}} \Big\| \int_{s}^{t} \int_\T (p_{t-r}-p_{(t-r)_h}^{n})(x,y)\, b^k \left(v_{r_h}^{n,k} (y_n)+ P^n_{r_h}(y_n) + O^n_{r_h} (y_n) \right) dy dr \Big\|_{L^m} 
&   \leq C \|b^k\|_\infty n^{-\frac{1}{2}+\varepsilon} |t-s|^{\frac{1}{2}+\frac{\varepsilon}{2}}  .
\end{align*}
\end{lemma}
\begin{proof}
First, for $|t-s| \leq h$, one can bound the quantity on the left-hand side by $C \|b^k \|_\infty |t-s| \leq  \|b^k \|_\infty n^{-1+\varepsilon} |t-s|^{1/2+\varepsilon}$.
Then, for $|t-s| \ge h$, we split the integral in two, first between $t-h$ and $t$ and then between $s$ and $t-h$. For the first integral, one has again the bound $  C\, \| b^k\|_\infty  n^{-2}  \leq C n^{-1+\varepsilon} |t-s|^{1/2+\varepsilon}$.
Between $s$ and $t-h$, we write
 \begin{align*}
 & \left| \int_{s}^{t-h}  \int_\T (p_{t-r}-p_{(t-r)_h}^{n})(x,y) \, b^k \left(v_{r_h}^{n,k} (y_n)+ P^n_{r_h}(y_n) + O^n_{r_h} (y_n) \right) dy dr  \right|  \\  
 & \quad \leq  C\, \| b^k \|_\infty  \int_s^{t-h} \int_{\T} |p_{t-r}(x,y) - p^n_{(t-r)_h}(x,y) | \, dy dr  \\
 & \quad \leq  C\, \| b^k \|_\infty  \int_s^{t-h} \left(\int_{\T} |p_{t-r}(x,y) - p^n_{(t-r)_h}(x,y) |^2 \, dy\right)^{\frac{1}{2}} dr  \\
& \quad \leq C\, \| b^k \|_\infty  n^{-\frac{1}{2}+\varepsilon} |t-s|^{\frac{1}{2}+\frac{\varepsilon}{2}} \ , 
\end{align*}
using Lemma~\ref{lem:P-Pn}$(i)$ with $\alpha= 1-2 \varepsilon$ in the last inequality. 
\end{proof}

\subsection{H\"older regularity of the numerical scheme}\label{subsec:reg-scheme}

We now use the regularisation result of Proposition \ref{prop:bound-Khn} to conclude on the regularity of the numerical scheme. 
\begin{proposition}\label{cor:bound-Khn}
Recall the functionals \eqref{def:holder-norm-phi} and \eqref{eq:discrete-seminorm}. Let $m \in [2, \infty)$ and assume that \eqref{eq:cond-gamma-p-H} holds.
Let $\mathcal{D}$ be a sub-domain of $\N^2$ satisfying \eqref{eq:assump-bn-bounded} for some parameter $\varepsilon \in (0, \frac{1}{2})$. Then
\begin{align}\label{eq:bound-Khn}
\sup_{(n,k) \in \mathcal{D}} ( u^{n,k}-O^n )_{\mathcal{C}^{\frac{1}{2}+\frac{\varepsilon}{2},0}_{[0,1],x} L^{m,\infty}}  \leq 2 \sup_{(n,k) \in \mathcal{D}} \big\{ u^{n,k} \big\}_{n,k,m, \frac{1}{2}+\frac{\varepsilon}{2}} < \infty .
\end{align}
\end{proposition}
\begin{proof}
The first inequality of \eqref{eq:bound-Khn} always holds by \eqref{eq:comp-seminorms2}. 
Applying Proposition \ref{prop:bound-Khn} with $f=b^n$, $\psi_{r}=P^n_{r} \psi_0+v^{n,k}_{r}=u^{n,k}_{r}-O^n_{r}$ and $\eta=\varepsilon/2$, there exists a constant $C$ such that for any $k,n \in \N^*$ and any $(s,t)\in \Delta_{[0,1]}$, we have
\begin{align*} 
&\Big\|  \Big( \EE^s \Big| \int_s^t \int_\T p_{t-r}(x,y) b^k (P_{r_h} \psi_0(y_n)+v^{n,k}_{r_h} (y_n)+ O^n_{r_h}(y_n)) \, dy dr \Big|^m  \Big)^{\frac{1}{m}} \Big\|_{L^{\infty}} \\
&\quad  \leq  C\, \Big( \|b^k \|_\infty\,  n^{-\frac{1}{2}} + \| b \|_{\mathcal{B}_p^\gamma} \Big) (t-s)^{\frac{3}{4}} + C\,  ( u^{n,k}-O^n )_{\mathcal{C}^{\frac{1}{2}+\frac{\varepsilon}{2},0}_{[s,t],x} L^{m,\infty}}  \, \Big( \| b \|_{\mathcal{B}_p^{\gamma}} + \| b^k \|_{\mathcal{C}^1} n^{-1} \Big)   \, (t-s)^{1+\frac{\varepsilon}{2}} .
\end{align*}
Moreover, by Lemma \ref{lem:bound-E4}, we know that
\begin{align*}
& \Big\|  \Big( \EE^s \Big| \int_s^t \int_\T (p^n_{(t-r)_h}-p_{t-r})(x,y)\, b^k (P_{r_h} \psi_0(y_n)+v^{n,k}_{r_h}(y_n) + O^n_{r_h}(y_n)) \, dy dr \Big|^m  \Big)^{\frac{1}{m}} \Big\|_{L^{\infty}} \\ 
& \quad \leq  C \| b^k \|_{\infty} n^{-\frac{1}{2}+\varepsilon} (t-s)^{\frac{1}{2}+\frac{\varepsilon}{2}} .
\end{align*}
By the triangle inequality, summing the two previous bounds yields
\begin{align*} 
& \Big\|  \Big( \EE^s \Big| \int_s^t \int_\T p^n_{(t-r)_h}(x,y) \, b^k (P_{r_h} \psi_0(y_n)+v^{n,k}_{r_h}(y_n) + O^n_{r_h}(y_n)) \, dy dr \Big|^m  \Big)^{\frac{1}{m}} \Big\|_{L^{\infty}} \\
 &\quad \leq  C\, \Big( \|b^k \|_\infty\,  n^{-\frac{1}{2}} + \| b \|_{\mathcal{B}_p^\gamma} \Big) (t-s)^{\frac{3}{4}}  + C\, ( u^{n,k}-O^n )_{\mathcal{C}^{\frac{1}{2}+\frac{\varepsilon}{2},0}_{[s,t],x} L^{m,\infty}}  \, \Big( \| b \|_{\mathcal{B}_p^{\gamma}} + \| b^k \|_{\mathcal{C}^1} n^{-1} \Big)   \, (t-s)^{1+\frac{\varepsilon}{2}}  \\ 
 & \quad \quad +C \| b^k \|_{\infty} n^{-\frac{1}{2}+\varepsilon} (t-s)^{\frac{1}{2}+\frac{\varepsilon}{2}} .
\end{align*}
Let $(S,T)\in \Delta_{[0,1]}$. Under \eqref{eq:assump-bn-bounded}, we have
$ \sup_{(h,k) \in \mathcal{D}} \| b^{k} \|_\infty n^{-\frac{1}{2}+\varepsilon} < \infty$ and $\sup_{(h,k) \in \mathcal{D}} \| b^{k} \|_{\mathcal{C}^1} n^{-1} < \infty$. 
Using these bounds, then dividing by $(t-s)^{\frac{1}{2}+\frac{\varepsilon}{2}}$ and taking the supremum over $(s,t)\in \Delta_{[S,T]}$ and $x\in \T$ yields
\begin{align*}
&\sup_{\substack{(s,t)\in \Delta_{[S,T]}\\  x \in \T}} \frac{\Big\|  \Big( \EE^s \Big| \int_s^t \int_\T p^n_{(t-r)_h}(x,y) \, b^k (u^{n,k}_{r_h}(y_n)) \, dy dr \Big|^m  \Big)^{\frac{1}{m}} \Big\|_{L^{\infty}}}{(t-s)^{\frac{1}{2}+\frac{\varepsilon}{2}}} \\
 &\quad \leq  C\, (1 + \| b \|_{\mathcal{B}_p^\gamma} )  + C\, ( u^{n,k}-O^n )_{\mathcal{C}^{\frac{1}{2}+\frac{\varepsilon}{2},0}_{[S,T],x} L^{m,\infty}} \, ( \| b \|_{\mathcal{B}_p^{\gamma}} + 1)   \, (T-S)^{\frac{1}{2}} \\
 &\quad \leq  C\, (1 + \| b \|_{\mathcal{B}_p^\gamma} )  \\
 &\quad\quad+ C\, \sup_{\substack{(s,t)\in \Delta_{[S,T]}\\  x \in \T}} \frac{\Big\|  \Big( \EE^s \Big| \int_s^t \int_\T p^n_{(t-r)_h}(x,y) \, b^k (u^{n,k}_{r_h}(y_n)) \, dy dr \Big|^m  \Big)^{\frac{1}{m}} \Big\|_{L^{\infty}}}{(t-s)^{\frac{1}{2}+\frac{\varepsilon}{2}}}  ( \| b \|_{\mathcal{B}_p^{\gamma}} + 1)   \, (T-S)^{\frac{1}{2}},
\end{align*}
using \eqref{eq:comp-seminorms2-0} in the last inequality.
Let $\ell = \big( \frac{1}{2C (\| b \|_{\mathcal{B}_p^{\gamma}}+1) } \big)^{2}$. Then for $T-S \leq \ell$, we have
\begin{align*}
\sup_{\substack{(s,t)\in \Delta_{[S,T]}\\  x \in \T}} \frac{\Big\|  \Big( \EE^s \Big| \int_s^t \int_\T p^n_{(t-r)_h}(x,y) \, b^k (u^{n,k}_{r_h}(y_n)) \, dy dr \Big|^m  \Big)^{\frac{1}{m}} \Big\|_{L^{\infty}}}{(t-s)^{\frac{1}{2}+\frac{\varepsilon}{2}}} \leq 2 C\, \Big(1 + \| b \|_{\mathcal{B}_p^\gamma} \Big) .
\end{align*}
Since $\ell$ does not depend on $n$ nor $k$, one can apply the previous inequality on $[0,\ell]$, then $[\ell,2\ell]$, etc. Iterating the argument, one gets
$ \{ u^{n,k}\}_{n,k,m,\frac{1}{2}+\frac{\varepsilon}{2}}  \leq \frac{2C}{\ell}\, (1 + \| b \|_{\mathcal{B}_p^\gamma} ) < \infty$.
\end{proof}

\subsection{Time and space regularity of the drift of the numerical scheme}\label{app:morereg-vhk}

We provide more results on the regularity of the drift part of the numerical scheme, i.e. $v^{n,k}$ defined in \eqref{def:v-hk}. These results are useful in the proof of the upper bound on the H\"older semi-norm of $\mathcal{E}^{2,h,k}$ in Corollary \ref{cor:bound-E2}, and also complement Proposition~\ref{cor:bound-Khn}.%

\begin{lemma}\label{lem:reg-v-t}
Let $m \in [2, \infty)$. Then for any $n \in \N^*$, $k \in \mathbb{N}$, we have 
\begin{align*}%
\| v^{n,k} - \pi_{n} v^{n,k} \|_{L^{\infty,\infty}_{[0,1],x} L^m} \leq C  \| b^k \|_{\infty} n^{-1+2 \varepsilon} .
\end{align*}
where we recall that $\pi_{n}$ is the projection operator on $\Lambda_{h}$ and $\T_{n}$ defined in \eqref{eq:defprojpi}.
\end{lemma}

\begin{proof}
Let $t \in [0,1]$ and $x \in \mathbb{T}$. Then
\begin{align*}
| v^{n,k}_t (x) - \pi_{n}(v^{n,k})_{t}(x)  | & = \Big| \int_{t_h}^t P^n_{(t-r)_h} b^k(u^{n,k}_{r_h})(x) \, dr \\ 
& \quad + \int_0^{t_h}   \left( P^n_{(t-r)_h} b^k(u^{n,k}_{r_h})(x) - P^n_{(t_h-r)_h} b^k(u^{n,k}_{r_h})(x_n) \right) dr \Big|.
\end{align*}
Using Lemma \ref{lem:reg-Pnh}$(iii)$ with $\alpha=0$, the first integral above $|\int_{t_h}^t P^n_{(t-r)_h} b^k(u^{n,k}_{r_h})(x) dr|$ is bounded by $C \| b^k \|_\infty (t-t_h)$. Hence, recalling that $h=c(2n)^{-2}$, $|\int_{t_h}^t P^n_{(t-r)_h} b^k(u^{n,k}_{r_h})(x) dr|\leq  C \| b^k \|_\infty n^{-2}$. 

As for the second integral, if $t_h \leq 4h$, then by the same arguments, $|\int_0^{t_h}   ( P^n_{(t-r)_h} b^k(u^{n,k}_{r_h})(x) - P^n_{(t_h-r)_h} b^k(u^{n,k}_{r_h})(x_n) ) dr| \leq C \| b^k\|_{\infty} h\leq C \| b^k \|_\infty n^{-2}$. If $t_h \ge 4h$, then we split the integral first between $0$ and $t_h-4h$ and then between $t_h-4h$ and $t_h$. For the integral between $t_h-4h$ and $t_h$, the same arguments yield again 
$|\int_{t_h-4h}^{t_h} (P^n_{(t-r)_h} b^k(u^{n,k}_{r_h})(x) - P^n_{(t_h-r)_h} b^k(u^{n,k}_{r_h})(x_n)) dr| \leq C \| b^k \|_\infty n^{-2}$.
 
 Denote by $J = |\int_0^{t_h-4h} (P^n_{(t-r)_h} b^k(u^{n,k}_{r_h})(x) - P^n_{(t_h-r)_h} b^k(u^{n,k}_{r_h})(x_n) ) dr|$ the remaining term. 
Let $f \in L^{\infty}(\T)$. By Lemma \ref{lem:reg-Pnh}$(ii)$ (with $\beta=1/2$), we have that for any $u \in [h,1]$ and $\delta \in [0,h)$,
\begin{align*}
| P^n_{u_h+\delta} f(x)- P^n_{u_h}f(x_n) | \leq C \| f \|_{\infty} n^{-\frac{1}{2}+\varepsilon} (u_h+\delta)^{-\frac{1}{4}} .
\end{align*}
Letting $\delta$ go to $h$ and using the continuity in time of $P^n$ (see equations (2.9) and (2.11) in \cite{butkovsky2021optimal}), we deduce that for any $u \in [h,1]$ and $x\in \T$,
\begin{align}\label{eq:contPnh}
| P^n_{u_h+h} f(x)- P^n_{u_h}f(x_n) | \leq C \| f \|_{\infty} n^{-\frac{1}{2}+\varepsilon} (u_h+h)^{-\frac{1}{4}} .
\end{align}
Observe that either $(t-r)_h=(t_h-r)_h$ or $(t-r)_h= (t_h-r)_h + h$. Hence 
 using 
\eqref{eq:contPnh} and Lemma \ref{lem:reg-Pnh}$(ii)$ with $\beta=1/2$, we deduce that
\begin{align*}
J & \leq C \int_0^{t_h-4h} \| b^k \|_{\infty} n^{-\frac{1}{2}+\varepsilon} (t_h-r)_h^{-\frac{1}{4}} \, dr. 
\end{align*}
Since $r < t_h-4h$, we have $(t_h-r)_h \ge (t-r)/2$. It follows that  $J  \leq C \| b^k \|_{\infty} n^{-\frac{1}{2}+\varepsilon}$.
\end{proof}

Now we investigate the H\"older regularity in time of $v^{n,k}-\pi_{n} v^{n,k}$.
\begin{lemma}\label{lem:reg-v-t-2}
We have for any  $n \in \N^*$, $k \in \N$,
\begin{align*}%
\big( v^{n,k}-\pi_{n} v^{n,k} \big)_{\mathcal{C}^{\frac{3}{4},0}_{[0,1],x} L^\infty} \leq C \| b^k \|_{\infty} n^{-\frac{1}{2}+\varepsilon} .
\end{align*}
\end{lemma}

\begin{proof}
Let $x \in \T$ and $0 \leq s \leq u < t \leq 1$. Using \eqref{eq:comp-seminorms} with $m=1,q=\infty$, $\rho_{t}=v^{n,k}_{t}-\pi_{n} (v^{n,k})_{t}$ and $Y = \int_0^s (P^n_{(t-r)_h}b^k(u^{n,k}_{r_{h}})(x)-P^n_{(t_h-r)_h}b^k(u^{n,k}_{r_{h}})(x_n)) \, dr$, we get
\begin{align*}
& \| \EE^s [v^{n,k}_t(x)-v^{n,k}_{t_h}(x_n)]- \EE^u [v^{n,k}_t(x)-v^{n,k}_{t_h}(x_n)] \|_{L^\infty} \\
& \leq 2 \Big\| \int_s^t  (P^n_{(t-r)_h} b^k(u^{n,k}_{r_h})(x)-P^n_{(t_h-r)_h}  b^k(u^{n,k}_{r_h})(x_n)) \, dr \Big\|_{L^\infty} .
\end{align*}
If $t-s \leq 4h$, we use Lemma \ref{lem:reg-Pnh}$(iii)$ with $\alpha=0$ to bound the previous quantity by $C\| b^k \|_\infty (t-s) \leq C \|b^k \|_\infty h^{\frac{1}{4}} (t-s)^{\frac{3}{4}} \leq C \|b^k \|_\infty n^{-\frac{1}{2}} (t-s)^{\frac{3}{4}}$. If $t-s > 4h$, we split the integral first between $s$ and $t-4h$, then between $t-4h$ and $t$. For the latter integral, we have again the bound $C \| b^k \|_{\infty} (2h) \lesssim \|b^k \|_\infty h^{\frac{1}{4}} (t-s)^{\frac{3}{4}}\leq C\|b^k \|_\infty n^{-\frac{1}{2}} (t-s)^{\frac{3}{4}}$.
We denote the first integral, on the interval $[s, t-4h]$ by $\tilde J$. For $r \in (s,t-4h)$, we have $t-r \ge t_h-r \ge h$. Moreover, we either have $(t-r)_h-(t_h-r)_h=h$ or $(t-r)_h=(t_h-r)_h$. Therefore applying Lemma \ref{lem:reg-Pnh}$(ii)$ with $\beta=1/2$ and \eqref{eq:contPnh}, we deduce that
\begin{align*}
J & \leq C \| b^k \|_{\infty} \int_s^{t-4h}  n^{-\frac{1}{2}+\varepsilon} (t_h-r)_h^{-\frac{1}{4}}\, dr.
\end{align*}
Finally, since $ (t_h-r)_h \ge (t-r)/2 \ge h$, we have $J \leq  C \| b^k \|_{\infty}  n^{-1/2+\varepsilon} (t-s)^{3/4}$, which concludes the proof.
\end{proof}

\section{Combined regularisation by the continuous and discrete Ornstein-Uhlenbeck processes, and regularity of $\mathcal{E}^{2,n,k}$}\label{sec:reg-On}

In this section, we show regularisation properties when both $O$ and $O^n$ appear, with a view towards controlling the H\"older regularity of $\mathcal{E}^{2,n,k}$ (defined in \eqref{def:E2}). Since $O^n$ is discretised in time, upper bounds on quantities such as $\int_0^t P_{t-r}\big( f(\phi_r + O_r)-f(\phi_r + O^n_r)\big)(x)\, dr$ will involve $\|f\|_{\infty}$, as seen in the previous section. Recall that $f$ plays the role of a sequence that converges to a distribution, so $\| f \|_\infty$ will have to be compensated by negative powers of $n$.

 \subsection{Combined regularisation by the continuous and discrete Ornstein-Uhlenbeck processes}\label{sec:reg-On-2}

We first study the joint regularisation effect of $O$ and $O^n$ when integrating an $\mathcal{F}_{S}$-measurable process, for fixed time $S$.
\begin{lemma}\label{sec:ssl-on}
Let $m \in [2,+\infty)$ and $\varepsilon \in (0, \frac{1}{2})$. There exists  $C>0$ such that for any $0 \leq S \leq T$, any $f  \in \mathcal{C}_b^\infty (\R, \R)$, any $(s,t) \in \Delta_{[S,T]}$, any $n\in \N^*$ and any $\phi : [0,1] \times \T \times \Omega \rightarrow \mathbb{R}$ such that for all $r \in [0,1]$ and $x \in \mathbb{T}$, $\phi_r(x)$ is $\mathcal{F}_S$-measurable, there is
\begin{align*}%
& \sup_{x \in \mathbb{T}} \Big\| \int_s^t \int_{\T} p_{T-r}(x,y) \Big( f(\phi_r(y)+O_r(y)) - f(\phi_r(y) + O_{r_h}^n(y)) \Big) dy dr\Big\|_{L^m}  \leq C \| f \|_\infty n^{-\frac{1}{2}+\varepsilon}\,  (t-s)^{\frac{1}{2} + \varepsilon}.
\end{align*}
\end{lemma}

\begin{proof}
Fix $x \in \mathbb{T}$ and $0 \leq S \leq T \leq 1$. We will apply the Stochastic Sewing Lemma (see Lemma~\ref{lem:SSL}). For any $(s,t) \in \Delta_{[S,T]}$, define
\begin{align*}
    \mathcal{A}_t:=\int_S^t p_{T-r}(x,y) \Big( f(\phi_r(y)+O_r(y)) - f(\phi_r(y) + O_{r_h}^n(y)) \Big)  \, dy dr  ~~\text{and}~~ A_{s,t}:=\EE^s[ \mathcal{A}_t - \mathcal{A}_s].
\end{align*}
Let $u=(s+t)/2$. Observe that $\EE^s[\delta A_{s,u,t}]=0$, so \eqref{eq:condsew1} holds.
To show \eqref{eq:condsew2}, let us prove that for some $\varepsilon_{2}>0$,
\begin{align} \label{(4.8)-critic}
    \|\delta A_{s,u,t}\|_{L^m}\leq  C \, \| f \|_\infty n^{-\frac{1}{2}+\varepsilon} \, (t-s)^{\frac{1}{2}+\varepsilon_2} .
\end{align} 
As in \eqref{eq:backtoAst-}, the triangle inequality and the conditional Jensen inequality give that
\begin{align}\label{eq:backtoAst}
    \|\delta A_{s,u,t}\|_{L^m} 
            &\leq 2 \int_u^t \int_{\T} p_{T-r}(x,y) \left\|\EE^u \left[f(\phi_r(y)+O_r(y)) - f(\phi_r(y) + O_{r_h}^n(y)) \right] \right\|_{L^m} \, dy dr .
\end{align}
\paragraph{The case $t-u \leq 2h$.} In this case, there is simply
\begin{align}\label{eq:Ast-<h}
  \|\delta A_{s,u,t}\|_{L^m} \leq C \| f \|_{\infty} (t-u) 
  \leq C \| f \|_{\infty} h^{\frac{1}{2}-\varepsilon} (t-s)^{\frac{1}{2}+\varepsilon} \leq  C \| f \|_{\infty} n^{-1+2 \varepsilon} (t-s)^{\frac{1}{2}+\varepsilon},
\end{align}
using $h=c(2n)^{-2}$.
\paragraph{The case $t-u \ge 2h$.} Split the integral in \eqref{eq:backtoAst} first between $u$ and $u+2h$, then between $u+2h$ and $t$ and call the two respective integrals $J_1$ and $J_2$. We bound $J_1$ as in \eqref{eq:Ast-<h}, so there is $J_{1}\leq C \| f \|_{\infty} n^{-1+2 \varepsilon} (t-s)^{1/2+\varepsilon}$. 

As for $J_2$, for $r \in [0,1]$ and $y \in \T$, consider the $\mathscr{B}(\R) \otimes \mathcal{F}_{S}$-measurable random function $F_{r,y}(\cdot) := f(\phi_r(y)+\cdot)$. Then we have
\begin{align*}
J_2 \leq C \int_{u+2h}^t \int_{\T} p_{T-r}(x,y) \left\|\EE^u  \left[F_{r,y}(O_{r}(y)) - F_{r,y}(O_{r_h}^n(y_n)) \right] \right\|_{L^m} \, dy dr .
\end{align*}
Lemma \ref{lem:reg-O} and Lemma \ref{lem:reg-On} give
\begin{align*}
 J_2 & =\int_{u+2h}^t  \int_{\T} p_{T-r}(x,y) \Big\|G_{Q(r-u)} F_{r,y}(P_{r-u} O_u(y)) - G_{Q^n(r_h-s)} F_{r,y}(\widehat{O}^n_{u,r_h}(y)) \Big\|_{L^m} \, dy dr  .
\end{align*}
Now Gaussian kernel estimates (see Lemma~\ref{lem:reg-Pnh}$(i)$ with $\alpha=0$, $\beta=1$) yield
\begin{align}\label{eq:J2lem61}
\begin{split} 
 J_2  & \leq C \| f \|_\infty  \int_{u+2h}^t \int_{\T} p_{T-r}(x,y) \Big( \| P_{r-u} O_u(y) - \widehat{O}^n_{u,r_h}(y) \|_{L^m} + | Q(r-u) - Q^n(r_h-u) |^{\frac{1}{2}} \Big) \\ 
 & \hspace{2.5cm} \times \Big( Q^n(r_h-u) \wedge Q(r-u) \Big)^{-\frac{1}{2}} dy d r .
 \end{split}
\end{align}
Let us now provide a bound on the different quantities that appear on the right-hand side of \eqref{eq:J2lem61}. First, $Q(r-u) \ge C (r-u)^{\frac{1}{2}}$ (see Lemma~\ref{lem:bound-Qn}$(i)$) and since $r_{h}-u\geq h$, Lemma~\ref{lem:bound-Qn}$(iii)$ states that $Q^n(r_h-u) \ge C (r_h-u)^{\frac{1}{2}} \ge \frac{C}{\sqrt{2}} (r-u)^{\frac{1}{2}} $.
Moreover, by Lemma~\ref{lem:reg-Pnh}$(i)$ (with $\alpha=\beta=(1-\varepsilon)/2$) and \cite[Corollary 2.3.2]{butkovsky2021optimal} (with $\beta=1$), one gets
\begin{align}\label{eq:O-On} 
\| P_{r-u} O_u(y) - \widehat{O}^n_{u,r_h}(y) \|_{L^m}  & \leq  \| P_{r-u} O_u(y)- P_{r_h-u} O_u(y) \|_{L^m} + \| P_{r_h-u} O_u(y) - \widehat{O}^n_{u,r_h}(y) \|_{L^m}  \nonumber \\
& \leq C \| O \|_{\mathcal{C}^{\frac{1}{2}-\frac{\varepsilon}{2},0}_{[0,1],x}L^m} n^{-\frac{1}{2}+\varepsilon} + C n^{-\frac{1}{2}+\varepsilon} \leq C n^{-\frac{1}{2}+\varepsilon} .
\end{align}
Moreover, we have
\begin{align}\label{eq:Q-Qn}
| Q(r-u) - Q^n(r_h-u) | & \leq | Q(r-u) - Q(r_h-u) | + | Q(r_h-u) - Q^n(r_h-u) | \nonumber \\
& \leq \int_{r_h-u}^{r-u} \int_{\T} |p_{2\theta}(y,0)|^2\, d y d\theta + | Q(r_h-u) - Q^n(r_h-u) |  \nonumber\\
& \leq C (\sqrt{r-u}-\sqrt{r_h-u}) + C n^{-1+2\varepsilon} \nonumber \\
& \leq C \sqrt{r-r_h} + C n^{-1+2\varepsilon} \leq C n^{-1+2\varepsilon} ,
\end{align}
where we used the fact that $\int_{\T} |p_{2\theta}(y,0)|^2\, d y \leq C \theta^{-\frac{1}{2}}$, and Lemma~\ref{lem:bound-Qn}$(ii)$ with $\alpha=2-4\varepsilon$. Going back to \eqref{eq:J2lem61}, we get
\begin{align*}
J_2 
& \leq C  \| f \|_{\infty} n^{-\frac{1}{2}+\varepsilon} \int_{u+2h}^t  (r-u)^{-\frac{1}{4}} d r  \leq C  n^{-\frac{1}{2}+\varepsilon} (t-s)^{\frac{3}{4}} .
\end{align*}
Hence, combining the bounds on $J_1$ and $J_2$, we have
\begin{align}\label{eq:Ast->h}
 \|\delta A_{s,u,t}\|_{L^m} & \leq C \| f \|_{\infty} n^{-\frac{1}{2}+\varepsilon} (t-s)^{\frac{3}{4}} .
\end{align}
Combining \eqref{eq:Ast-<h} and \eqref{eq:Ast->h}, we deduce \eqref{(4.8)-critic} for $\varepsilon_2=\varepsilon  >0$.
\paragraph{Convergence in probability.} 
We omit the proof of \eqref{eq:convAt} as it follows the same lines as the paragraph ``Convergence in probability" in the proof of Lemma \ref{lem:bound-Khn}.

~~
\smallskip

Applying Lemma~\ref{lem:SSL}, we get
\begin{align}\label{eq:lem34-ssl-goal}
   \Big\| \Big( \EE^S | \mathcal{A}_t-\mathcal{A}_s|^m \Big)^{\frac{1}{m}} \Big\|_{L^m} &\leq \|A_{s,t}\|_{L^m} + C \, \| f \|_\infty n^{-\frac{1}{2}+\varepsilon}\,  (t-s)^{\frac{1}{2} + \varepsilon}.
\end{align}
To bound $\|A_{s,t}\|_{L^m}$, notice that
\begin{align}\label{eq:boundAst}
   \|A_{s,t}\|_{L^m} & \leq \int_u^t \int_{\T} p_{T-r}(x,y) \|\EE^u  [f(\phi_r(y)+O_r(y)) - f(\phi_r(y) + O_r^n(y)) ] \|_{L^m} \, dy dr  .
\end{align}
Repeating the same arguments used to obtain \eqref{(4.8)-critic} (notice the right-hand side in \eqref{eq:backtoAst} is the same up to a factor $2$) we get
\begin{align*}
  \|A_{s,t}\|_{L^m} & \leq C \| f \|_\infty n^{-\frac{1}{2}+\varepsilon} (t-s)^{\frac{1}{2}+\varepsilon}. 
\end{align*}
Plugging the previous inequality in \eqref{eq:lem34-ssl-goal} yields the desired result.
\end{proof}

Now we generalise the previous lemma to any adapted process $\phi : [0,1] \times \T \times \Omega \rightarrow \mathbb{R}$. For such $\phi$, we recall the pseudo-norm \eqref{def:holder-norm-phi}.

\begin{prop}\label{prop:ssl-on-final}
Let  $m \in [2,+\infty)$ and $\varepsilon \in (0, \frac{1}{2})$.  There exists a constant $C$ such that for any $(s,t) \in \Delta_{[0,1]}$, any $f  \in \mathcal{C}_b^\infty (\R, \R)$, any $\mathbb{F}$-adapted process $\phi : [0,1] \times \T \times \Omega \rightarrow \mathbb{R}$ and any $n\in \N^*$, there is
\begin{align}\label{eq:ssl-on-final}
& \sup_{x \in \mathbb{T}} \Big\| \int_s^t \int_{\T} p_{T-r}(x,y) \Big( f\left(\phi_r(y)+O_r(y)\right) - f\left(\phi_r(y) + O_{r_h}^n(y)\right) \Big) dy dr \Big\|_{L^m} \nonumber \\
& \quad\leq C \, \| f \|_\infty n^{-\frac{1}{2}+\varepsilon}\,  (t-s)^{\frac{1}{2} + \varepsilon} +  C \| f \|_{\mathcal{C}^1}  \big( \phi \big)_{\mathcal{C}^{1,0}_{[0,1],x} L^{1,\infty}}  n^{-1+\varepsilon} (t-s)^{1+\varepsilon} .
\end{align}
\end{prop}

\begin{proof}
Let $x \in \T$ and $0\leq S\leq T \leq 1$. We aim to apply the Stochastic Sewing Lemma (Lemma~\ref{lem:SSL}). For $(s,t) \in \Delta_{[S,T]}$, define
\begin{align*}
    \mathcal{A}_t &:=\int_S^t p_{T-r}(x,y) \Big( f\left(\phi_r(y)+O_r(y)\right) - f\left(\phi_r(y) + O_{r_h}^n(y)\right) \Big) dy dr \\
    A_{s,t}&:=\int_s^t p_{T-r}(x,y) \Big( f\left(\EE^s \phi_r(y)+O_r(y)\right) - f\left(\EE^s \phi_r(y) + O_{r_h}^n(y)\right) \Big)  dy dr .
\end{align*}
Let $\varepsilon \in (0,1/2)$. Assume without loss of generality that $\big( \phi \big)_{\mathcal{C}^{1,0}_{[0,1],x} L^{1,\infty}} < +\infty$, otherwise the result is obvious. We check the assumptions in order to apply Lemma \ref{lem:SSL} with $q=m$. To show that \eqref{eq:condsew1} and \eqref{eq:condsew2} hold true with $\varepsilon_1=\varepsilon_2=\varepsilon$, we prove that there exists  $C>0$ which does not depend on $s,t,S$ and $T$ such that

\begin{enumerate}[label=(\roman*)]
\item\label{en:(1b+)}  $\| \delta A_{s,u,t} \|_{L^m}  \leq C \| f \|_\infty n^{-\frac{1}{2}+\varepsilon}  (t-s)^{\frac{1}{2}+\varepsilon} $;

\item\label{en:(2b+)} $\|\EE^s \delta A_{s,u,t}\|_{L^m}  \leq C \| f \|_{\mathcal{C}^1}   \big( \phi \big)_{\mathcal{C}^{1,0}_{[0,1],x} L^{1,\infty}}  n^{-1+\varepsilon} (t-s)^{1+\varepsilon} $; 

\item\label{en:(3b+)}  If \ref{en:(1b+)} and \ref{en:(2b+)} are satisfied, \eqref{eq:convAt} gives the 
convergence in probability of $\sum_{i=1}^{N_k-1} A_{t^k_i,t^k_{i+1}}$ on any sequence $\Pi_k=\{t_i^k\}_{i=1}^{N_k}$ of partitions of $[S,t]$ with a mesh that converges to $0$. We will show 
that the limit is the previously defined process $\mathcal{A}$.
\end{enumerate}

If \ref{en:(1b+)}, \ref{en:(2b+)} and \ref{en:(3b+)} are assumed to hold, then Lemma \ref{lem:SSL} gives
\begin{align*}
& \Big\|  \int_s^t \int_{\T} p_{T-r}(x,y) \Big( f\left(\phi_r(y)+O_r(y)\right) - f\left(\phi_r(y) + O_{r_h}^n(y)\right) \Big) dy dr\Big\|_{L^m}   \\
&\quad \leq C \| f \|_\infty n^{-\frac{1}{2}+\varepsilon}  (t-s)^{\frac{1}{2}+\varepsilon} +C \| f \|_{\mathcal{C}^1}   \big( \phi \big)_{\mathcal{C}^{1,0}_{[0,1],x} L^{1,\infty}}  n^{-1+\varepsilon} (t-s)^{1+\varepsilon} + \| A_{s,t} \|_{L^m}.
\end{align*}
To bound $\| A_{{s},{t}} \|_{L^m}$, apply Lemma \ref{sec:ssl-on} with $\EE^s \phi_r$ in place of $\phi_r$ to get that
\begin{align}\label{eq:bound-Lp-norm-A+}
\| A_{s,t} \|_{L^m} & \leq C \| f \|_\infty n^{-\frac{1}{2}+\varepsilon}  (t-s)^{\frac{1}{2}+\varepsilon} . 
\end{align}
Thus, taking the supremum over $x \in \T$, we get \eqref{eq:ssl-on-final}.

~

Let us now verify that  \ref{en:(1b+)}, \ref{en:(2b+)} and \ref{en:(3b+)} are satisfied.

\paragraph{Proof of \ref{en:(1b+)}:}
For $u=(s+t)/2$, there is
\begin{align*}
 \|\delta A_{s,u,t}\|_{L^m} \leq  \|A_{s,t}\|_{L^m} +\|A_{s,u}\|_{L^m}  + \|A_{u,t}\|_{L^m} .
\end{align*}
Using \eqref{eq:bound-Lp-norm-A+} for each term separately, we get \ref{en:(1b+)}.

\paragraph{Proof of \ref{en:(2b+)}:}
For $u=(s+t)/2$, the triangle inequality and Jensen's inequality for 
conditional expectation give that
\begin{equation}\label{eq:decompAsutProp6.2}
\begin{split}
    \| \EE^s  \delta A_{s,u,t}\|_{L^m}
        &\leq \int_u^t \int_{\T} p_{T-r}(x,y) \Big\| \EE^s \EE^u \Big[f(\EE^s \phi_r(y)+O_{r}(y)) -f(\EE^u \phi_r(y)+O_r(y))  \\
        & \hspace{2cm} -f(\EE^s \phi_r(y)+ O_{r_h}^n(y)) + f(\EE^u \phi_r(y) + O_{r_h}^n(y))\Big]  \Big\|_{L^m} dy  dr .
\end{split}
\end{equation}

If $t-u \leq 2h$, we simply have
\begin{align}\label{eq:Ast--<h}
  \| \EE^s  \delta A_{s,u,t}\|_{L^m} 
   \leq C \| f \|_{\infty}   \big( \phi \big)_{\mathcal{C}^{1,0}_{[0,1],x} L^m} (t-s)^2 
   &\leq C \| f \|_{\infty}   \big( \phi \big)_{\mathcal{C}^{1,0}_{[0,1],x} L^m} h^{1-\varepsilon} (t-s)^{1+\varepsilon} \nonumber\\
 & \leq C \| f \|_{\infty}   \big( \phi \big)_{\mathcal{C}^{1,0}_{[0,1],x} L^m} n^{-2+2\varepsilon} (t-s)^{1+\varepsilon} .
\end{align}

If $t-u \ge 2h$, we split the integral in the right-hand side of \eqref{eq:decompAsutProp6.2} first between $u$ and $u+2h$, and then between $u+2h$ and $t$. Denote the two respective integrals $J_1$ and $J_2$. For $J_1$, we obtain $J_{1}\leq C \| f \|_{\infty}  \big( \phi \big)_{\mathcal{C}^{1,0}_{[0,1],x} L^m} n^{-2+2\varepsilon} (t-s)^{1+\varepsilon} $, as in the case $t-u \leq 2h$.
As for $J_2$, consider the $\mathscr{B}(\R) \otimes \mathcal{F}_{S}$-measurable random function $F_{s,u,r,y}(\cdot) := f(\EE^s \phi_r(y)+\cdot)-f(\EE^u \phi_r(y)+\cdot)$, for which $J_{2}$ reads
\begin{align*}
J_2 = \int_{u+2h}^t \int_{\T} p_{T-r}(x,y) \left\| \EE^s \EE^u  [F_{s,u,r,y}(O_r(y)) -F_{s,u,r,y}(O_{r_h}^n(y))] \right\|_{L^m} dy  dr .
\end{align*}
Lemma \ref{lem:reg-O} and Lemma \ref{lem:reg-On} give
\begin{align*}
J_2 = C \int_{u+2h}^t \int_{\T} p_{T-r}(x,y) \left\| \EE^s \left( G_{Q(r-u)} F_{s,u,r,y}(P_{r-u}O_u(y)) -G_{Q^n(r_h-u)} F_{s,u,r,y}(\widehat{O}_{u,r_h}^n(y))    \right) \right\|_{L^m} dy  dr .
\end{align*}
Using Lemma \ref{lem:reg-Pnh}$(i)$ (with $\alpha=0$ and $\beta=2-2\eta$ for $\eta \in (0,1)$), we obtain
\begin{align*}
J_2 
& \leq C  \int_{u+2h}^t \int_{\T} p_{T-r}(x,y)  \Big\| \left( | Q(r-u) - Q^n(r_h-u) |^{1-\eta}+ | P_{r-u} O_u(y)-\widehat{O}^n_{r_h,u}(y) |^{2-2\eta} \right)   \\ 
& \hspace{4cm} \times \left( Q^n(r_h-u) \wedge Q(r-u) \right)^{-1+\eta} \| F_{s,u,r,y} \|_\infty  \Big\|_{L^m} dy  d r .
\end{align*}
Since $Q(r-u) \ge C (r-u)^{\frac{1}{2}}$ and $Q^n(r_h-u) \ge C (r_h-u)^{\frac{1}{2}} \ge \frac{C}{\sqrt{2}} (r-u)^{\frac{1}{2}}$ (see Lemma~\ref{lem:bound-Qn}), it comes
\begin{align*}
J_2 & \leq C  \int_{u+2h}^t \int_{\T} p_{T-r}(x,y) \Big\| \left( | Q(r-u) - Q^n(r_h-u) |^{1-\eta}+| P_{r-u}O_u(y)-\widehat{O}_{r_h,u}^n(y) |^{2-2\eta} \right)    (r-u)^{-\frac{1}{2}+\frac{\eta}{2}} \\ 
& \hspace{4cm} \times \| f \|_{\mathcal{C}^1}  \left| \EE^s \left[\EE^s \phi_r(y)-\EE^u\phi_r(y)\right]\right| \Big\|_{L^m} dy dr .
\end{align*}
Using \eqref{eq:Q-Qn} and \eqref{eq:O-On}, we get
\begin{align*}
J_2 &  \leq C \| f \|_{\mathcal{C}^1} \big(\phi\big)_{\mathcal{C}^{1,0}_{[0,1],x} L^{1,\infty}} n^{-1+3\eta}    (u-s) \int_{u+2h}^t  (r-u)^{-\frac{1}{2}+\frac{\eta}{2}} d r \\
& \leq C   \| f \|_{\mathcal{C}^1} \big( \phi \big)_{\mathcal{C}^{1,0}_{[0,1],x} L^\infty} n^{-1+3\eta} (t-s)^{\frac{3}{2}} . 
\end{align*}
Hence, taking $\eta=\frac{1}{3}\varepsilon$ and combining the bounds on $J_1$ and $J_2$, we have obtained that when $t-u \ge 2h$,
\begin{align}\label{eq:Ast-->h}
 \| \EE^s  \delta A_{s,u,t}\|_{L^m} & \leq C \| f \|_{\mathcal{C}^1}   \big( \phi \big)_{\mathcal{C}^{1,0}_{[0,1],x} L^{1,\infty}}  n^{-1+\varepsilon} (t-s)^{1+\varepsilon} .
\end{align}
Combining \eqref{eq:Ast--<h} and \eqref{eq:Ast-->h}, we deduce that for all $s \leq t$,
\begin{align*}%
 \|\EE^s \delta A_{s,u,t}\|_{L^m} & \leq C \| f \|_{\mathcal{C}^1}   \big( \phi \big)_{\mathcal{C}^{1,0}_{[0,1],x} L^{1,\infty}}  n^{-1+\varepsilon} (t-s)^{1+\varepsilon} .
\end{align*}

\paragraph{Proof of \ref{en:(3b+)}:} We proceed similarly to the third step of previous proofs, using the Lipschitz property of $f$ and that $\big( \phi \big)_{\mathcal{C}^{1,0}_{[0,1],x} L^{1,\infty}}<\infty$.
\end{proof}

\subsection{H\"older regularity of $\mathcal{E}^{2,n,k}$}\label{subsec:reg-E2}

Combining the results of Proposition \ref{prop:ssl-on-final} and Proposition \ref{prop:ssl-onh-final}, we deduce the following. 

\begin{proposition}\label{cor:ssl-o-onh-final}
Let $m \in [2,+\infty)$ and $\varepsilon \in (0, \frac{1}{2})$. There exists $C$ such that for any $(s,t) \in \Delta_{[0,1]}$, any  $f  \in \mathcal{C}_b^\infty (\R, \R) \cap \mathcal{B}^\gamma_{p}$, any $\mathbb{F}$-adapted process $\phi : [0,1] \times \T \times \Omega \rightarrow \mathbb{R}$ and any $n\in \N^*$, there is
\begin{align}\label{eq:ssl-o-onh-final}
\begin{split}
& \sup_{x \in \mathbb{T}} \Big\| \int_s^t \int_{\T} p_{t-r}(x,y) \left( f(\phi_r(y)+O_{r}(y)) - f(\phi_{r_h}(y_n) + O_{r_h}^n(y_n)) \right) dy dr \Big\|_{L^m}  \\
& \leq C ( 1+\| f \|_{\mathcal{B}^{\gamma}_p}) \bigg( \| f \|_{\infty} n^{-\frac{1}{2}+\varepsilon} + n^{2\varepsilon} \left(\| \phi-\pi_n \phi \|_{L^{\infty,\infty}_{[h,1],x} L^m} +n^{-\frac{1}{2}} \right)  \\ 
& \quad +  (\phi)_{\mathcal{C}^{1,0}_{[0,1],x} L^{1,\infty}} \Big(  \| \phi- \pi_n \phi \|_{L^{\infty,\infty}_{[h,1],x} L^m} +n^{-\frac{1}{2}}+  \| f \|_{\mathcal{C}^1} n^{-1+\varepsilon} \Big)  +  (\phi-\pi_n \phi )_{\mathcal{C}^{\frac{1}{2}+\varepsilon,0}_{[0,1],x} L^{m,\infty}} \bigg) (t-s)^{\frac{1}{2}+\varepsilon}   .
\end{split}
\end{align}
\end{proposition}

\begin{proof}
Let $x \in \mathbb{T}$ and $0 \leq S \leq T \leq 1$. For $(s,t) \in \Delta_{[S,T]}$, write
\begin{align*}
& \Big\|  \int_s^t \int_{\T} p_{T-r}(x,y) \left(f(\phi_r(y)+O_{r}(y)) - f(\phi_{r_h}(y_n) + O_{r_h}^n(y_n)) \right) dy dr \Big\|_{L^m}  \\
& \quad  \leq \Big\|  \int_s^t \int_{\T} p_{T-r}(x,y) \left( f(\phi_r(y)+O_{r}(y)) - f(\phi_{r}(y) + O_{r_h}^n(y))\right) dy dr \Big\|_{L^m}  \\ 
& \quad \quad + \Big\| \int_s^t \int_{\T} p_{T-r}(x,y) \left( f(\phi_r(y)+O^n_{r_h}(y)) - f(\phi_{r_h}(y_n) + O_{r_h}^n(y_n))\right) dy dr \Big\|_{L^m}  \\ 
& \quad =: J_1 + J_2 .
\end{align*}
Apply Proposition~\ref{prop:ssl-on-final} to $J_1$ and Proposition~\ref{prop:ssl-onh-final} to $J_2$ to get the result.
\end{proof}

Using the regularisation result of Proposition~\ref{cor:ssl-o-onh-final}, we can now state a bound on $\mathcal{E}^{2,n,k}$ (which was defined in \eqref{def:E2}).
\begin{corollary}\label{cor:bound-E2}
Assume that \eqref{eq:cond-gamma-p-H} holds.
Let $m \ge 2$ and $\varepsilon \in (0,\frac{1}{4})$. There exists $C$ such that for any $ (s,t) \in \Delta_{[0,1]}$ and any $n\in \N^*$, $k \in \mathbb{N}$, we have
\begin{align*}%
\sup_{x \in \mathbb{T}} \Big\| \mathcal{E}^{2,n,k}_{s,t}(x) \Big\|_{L^m} 
  \leq C (1+\| \psi_0 \|_{\mathcal{C}^{\frac{1}{2}-\varepsilon}} ) \Big( (1+\| b^k \|_{\infty}) n^{-\frac{1}{2}+\varepsilon} + (1+\| b^k \|_{\infty}) \| b^k \|_{\mathcal{C}^1}  n^{-1+\varepsilon} \Big) (t-s)^{\frac{1}{2}} .
\end{align*}
\end{corollary}

\begin{proof}
Apply Proposition~\ref{cor:ssl-o-onh-final} with $\phi = v^{n,k}+P^n \psi_0$ to get
\begin{align*}
&\sup_{x \in \mathbb{T}} \|  \mathcal{E}^{2,n,k}_{s,t}(x) \|_{L^m}\\
&= \sup_{x \in \mathbb{T}} \Big\| \int_s^t \int_{\T} p_{t-r}(x,y)  \Big( b^k(v^{n,k}_r(y) + P_r^n \psi_0(y)+ O_r(y)) - b^k(v^{n,k}_{r_h}(y_n) + P_{r_h}^n \psi_0(y_n) + O^n_{r_h}(y_n)) \Big)\, dy dr \Big\|_{L^m}  \\ 
&\hspace{1.5cm}\leq C (1+ \|b^k\|_{\mathcal{B}^{\gamma}_p}) \Big( \| b^k\|_{\infty} n^{-\frac{1}{2}+\varepsilon} + n^{2\varepsilon} (\| \phi-\pi_n \phi \|_{L^{\infty,\infty}_{[h,1],x} L^m} +n^{-\frac{1}{2}})  \\ 
&\hspace{1.5cm} \quad +  (\phi)_{\mathcal{C}^{1,0}_{[0,1],x} L^{1,\infty}} \Big(  \| \phi- \pi_n \phi \|_{L^{\infty,\infty}_{[h,1],x} L^m}+n^{-\frac{1}{2}} +  \| b^k \|_{\mathcal{C}^1} n^{-1+\varepsilon} \Big)  +  (\phi-\pi_n \phi )_{\mathcal{C}^{\frac{1}{2}+\varepsilon,0}_{[0,1],x} L^{m,\infty}} \Big) (t-s)^{\frac{1}{2}+\varepsilon}   .
\end{align*}
Now use the following properties: from the definition of the sequence $(b^k)_{k\in \N}$, one has ${\sup_{k}  \|b^k\|_{\mathcal{B}^{\gamma}_p}<\infty}$; using that $\psi_{0}$ is deterministic, $(v^{n,k}+P^n \psi_0)_{\mathcal{C}^{1,0}_{[0,1],x} L^{1,\infty}} = (v^{n,k})_{\mathcal{C}^{1,0}_{[0,1],x} L^{1,\infty}}$ and $\big(v^{n,k}+P^n \psi_0 - \pi_{n}(v^{n,k}+P^n \psi_0)\big)_{\mathcal{C}^{1,0}_{[0,1],x} L^{1,\infty}} = (v^{n,k} - \pi_{n}v^{n,k})_{\mathcal{C}^{1,0}_{[0,1],x} L^{1,\infty}}$; and 
$$ \| \phi- \pi_n \phi \|_{L^{\infty,\infty}_{[h,1],x} L^m} \leq  \| v^{n,k}- \pi_n v^{n,k} \|_{L^{\infty,\infty}_{[h,1],x} L^m} + { \| P^n \psi_0- \pi_n P^n \psi_0 \|_{L^{\infty,\infty}_{[h,1],x}}}.$$
 Hence the previous inequality now reads
\begin{align}\label{eq:discrete-final-comb}
\begin{split}
&\sup_{x \in \mathbb{T}} \|  \mathcal{E}^{2,n,k}_{s,t}(x) \|_{L^m}\\
&\quad \leq C \bigg( \| b^k \|_{\infty} n^{-\frac{1}{2}+\varepsilon} + n^{-\frac{1}{2}+2\varepsilon}  + n^{2\varepsilon} \| v^{n,k}- \pi_n v^{n,k}  \|_{L^{\infty,\infty}_{[h,1],x} L^m} + n^{2\varepsilon} \sup_{\substack{t \in [0,1] \\ x \in \T}} |P^n_t \psi_0(x)  -P^n_{t_h} \psi_0(x_n) |  \\
&\quad\quad+ (v^{n,k})_{\mathcal{C}^{1,0}_{[0,1],x} L^{1,\infty}} \Big(\| v^{n,k}- \pi_n v^{n,k}  \|_{L^{\infty,\infty}_{[h,1],x} L^m} +  \sup_{\substack{t \in [0,1] \\ x \in \T}} |P^n_t \psi_0(x)  -P^n_{t_h} \psi_0(x_n) | +n^{-\frac{1}{2}}  +  \| b^k \|_{\mathcal{C}^1} n^{-1+\varepsilon} \Big)\\
& \quad\quad   +  (v^{n,k}-\pi_n v^{n,k} )_{\mathcal{C}^{\frac{1}{2}+\varepsilon,0}_{[0,1],x} L^{m,\infty}}   \bigg)  (t-s)^{\frac{1}{2}+\varepsilon} .
\end{split}
\end{align}
Observe that for $s \leq u  < t$, using \eqref{eq:comp-seminorms1} with $\rho=v^{n,k}$ and 
$Y=\int_0^s P^n_{(t-\theta)_h} b^k(u^{n,k}_{\theta_{h}}(y)) d \theta$, we get
\begin{align*}
\EE^s |\EE^u v^{n,k}_r(x)-\EE^s v^{n,k}_r(x)|   
& \leq 2 \EE^s \Big|\int_s^t \int_{\T} p^n_{(t-\theta)_h}(x,y) \, b^k(u^{n,k}_{\theta_h}(y_n)) \, dy d\theta\Big| \\
& \leq C \| b^k \|_\infty (t-s) ,
\end{align*}
where we used Lemma~\ref{lem:reg-Pnh}$(iii)$ with $\alpha=0$ in the last inequality. 
Therefore, we get that
\begin{align}\label{eq:E2-1}
\big( v^{n,k} \big)_{\mathcal{C}^{1,0}_{[0,1],x} L^{1,\infty}} \leq C \| b^k \|_\infty .
\end{align}
Moreover, by Lemma~\ref{lem:reg-Pnh}$(iii)$ (with $\alpha=1/2-\varepsilon$), we have
\begin{align}\label{eq:E2-2}
|P^n_t \psi_0(x)  -P^n_{t_h} \psi_0(x_n) | \leq C \| \psi_0 \|_{\mathcal{C}^{\frac{1}{2}-\varepsilon}} n^{-\frac{1}{2}+\varepsilon}. 
\end{align}
Finally, by Lemma~\ref{lem:reg-v-t} and Lemma~\ref{lem:reg-v-t-2}, recall that
\begin{align}\label{eq:E2-3}
\begin{split}
\| v^{n,k} -\pi_n v^{n,k}\|_{L^{\infty,\infty}_{[0,1],x} L^m} & \leq C  \| b^k \|_{\infty} n^{-1+ \varepsilon}   \leq  C  \| b^k \|_{\mathcal{C}^1} n^{-1+\varepsilon} \\
 (v^{n,k}-\pi_n v^{n,k} )_{\mathcal{C}^{\frac{1}{2}+\varepsilon,0}_{[0,1],x} L^{m,\infty}} &\leq \big( v^{n,k}-\pi_n v^{n,k} \big)_{\mathcal{C}^{\frac{3}{4},0}_{[0,1],x} L^\infty} \leq C \| b^k \|_{\infty} n^{-\frac{1}{2}+\varepsilon} .
\end{split}
\end{align}
Injecting \eqref{eq:E2-1}, \eqref{eq:E2-2} and \eqref{eq:E2-3} in \eqref{eq:discrete-final-comb}, it comes
\begin{align*}
\sup_{x \in \mathbb{T}} \|  \mathcal{E}^{2,n,k}_{s,t}(x) \|_{L^m} 
&\leq C \bigg( \| b^k \|_{\infty} n^{-\frac{1}{2}+\varepsilon} + n^{-\frac{1}{2}+2\varepsilon}  +  \| b^k \|_{\mathcal{C}^1} n^{-1+3\varepsilon} + \| \psi_0 \|_{\mathcal{C}^{\frac{1}{2}-\varepsilon}} n^{-\frac{1}{2}+3\varepsilon} \\
&\quad\quad+ \| b^k \|_{\infty} \Big(\| b^k \|_{\infty} n^{-1+ \varepsilon} +  \| \psi_0 \|_{\mathcal{C}^{\frac{1}{2}-\varepsilon}} n^{-\frac{1}{2}+\varepsilon} +n^{-\frac{1}{2}}  +  \| b^k \|_{\mathcal{C}^1} n^{-1+\varepsilon} \Big)\\
& \quad\quad   + \| b^k \|_{\infty} n^{-\frac{1}{2}+\varepsilon}   \bigg)  (t-s)^{\frac{1}{2}+\varepsilon} 
\end{align*}
and up to changing $3\varepsilon$ into $\varepsilon$, we finally get the result. 
\end{proof}

\begin{appendices}

\section{General lemmas and estimates}\label{app:general-lemmas}

In this section, we recall and prove estimates on the discrete and continuous heat semigroup on the torus $P^n$ and $P$, the variance of the discrete and continuous OU processes, some Besov and regularity estimates. Finally we state the Stochastic Sewing Lemma, its critical version and an application used several times in Section~\ref{sec:reg-O}.

\subsection{Heat and Besov regularity estimates}

We gather in the following lemma some classical estimates on the continuous heat semigroup, as well as their less standard counterpart for the discrete heat semigroup $P^n$. In the following lemmas, recall that $h$ denotes the time step related to the space discretization parameter $n$ by the relation $h = c(2n)^{-2}$.

\begin{lemma}\label{lem:reg-Pnh}
Let $\alpha\leq\beta\in [0,1]$.
\begin{enumerate}[label=(\roman*)]
\item Let $\mathcal{S}$ be either $P$ or $G$, defined on the domain $\mathcal{D}$ which is respectively either $\T$ or $\R$. There exists $C>0$ such that for any $f\in \mathcal{C}^\alpha(\mathcal{D})$, any $(s,t)\in \Delta_{[0,1]}$ and any $x,y \in \mathcal{D}$, 
\begin{align*}
|\mathcal{S}_{t}f(x) - \mathcal{S}_{s}f(y)| \leq C\, \|f\|_{\mathcal{C}^\alpha(\mathcal{D})}\, \left( (t-s)^{\frac{\beta}{2}} + |x-y|^{\beta} \right) \, s^{\frac{\alpha-\beta}{2}}.
\end{align*}

\item There exists $C>0$ such that for any $f\in L^\infty(\T)$, any $n\in \N^*$, any $t \in [h,1]$, any $x\in \T$ and any $z\in \{x,x_{n}\}$, 
\begin{align*}
|P^n_{t}f(x) - P^n_{t_{h}}f(z)| \leq C\, \|f\|_{L^\infty(\T)}\, \left(\log(2n)\right)^{\frac{\beta}{2}} n^{-\beta}\, t^{-\frac{\beta}{2}}.
\end{align*}

\item There exists $C>0$ such that for any $f\in \mathcal{C}^\alpha(\T)$, any $n\in \N^*$, any $t \in [0,1]$, any $x\in \T$ and any $z\in \{x,x_{n}\}$, 
\begin{align*}
 \| P_t^n f \|_{\mathcal{C}^\alpha} & \leq C \| f \|_{\mathcal{C}^\alpha}\\
|P^n_{t}f(x) - P^n_{t_{h}}f(z)| & \leq C\, \|f\|_{\mathcal{C}^\alpha(\T)}\, n^{-\alpha} .
\end{align*}
\end{enumerate}
\end{lemma}

\begin{proof}
\begin{enumerate}[label=(\roman*),itemsep=0pt]
\item This is \cite[Lemma 2.2.1 (i)]{butkovsky2021optimal}.

\item This is the first statement of \cite[Lemma 2.2.6]{butkovsky2021optimal} with $\alpha=0$.

\item The first inequality is from \cite[Lemma 2.2.5]{butkovsky2021optimal}. The second one is a combination of the first statement of \cite[Lemma 2.2.6]{butkovsky2021optimal} with $\beta = \alpha$ when $t\geq h$ and of the second statement of \cite[Lemma 2.2.6]{butkovsky2021optimal} when $t\leq h$.
\end{enumerate}
\end{proof}

The next lemma provides comparisons between the continuous and discrete semigroups.
\begin{lemma}\label{lem:P-Pn}
~
\begin{enumerate}[label=(\roman*)]
\item Let $\alpha\in [0,2]$. There exists $C>0$ such that for any $n\in \N^*$, any $t \in [h,1]$ and any $x\in \T$,
\begin{align*}
 \left\|p_{t}(x,\cdot) - p_{t_{h}}^n(x,\cdot) \right\|^2_{L^2(\T)} \leq C\, n^{-\alpha}\, t^{-\frac{\alpha+1}{2}}.
\end{align*}

\item Let $\alpha\in [0,1]$. There exists $C>0$ such that for any $f\in \mathcal{C}^\alpha(\T)$, any $n\in \N^*$ and any $t \in [0,1]$ and any $x\in \T$,
\begin{align*}
\left| P_{t}^n f(x) - P_{t} f(x) \right| \leq C\, n^{-\alpha}\, \|f\|_{\mathcal{C}^\alpha(\T)}.
\end{align*}
\end{enumerate}
\end{lemma}

\begin{proof}
The first point is exactly \cite[Lemma 2.2.7]{butkovsky2021optimal} and the second one is \cite[Lemma 2.2.9]{butkovsky2021optimal}.
\end{proof}

The following Lemma states useful estimates on $Q$, $Q^n$ (defined in \eqref{def:discrete-OU}) and their difference.

\begin{lemma}\label{lem:bound-Qn}
~
\begin{enumerate}[label=(\roman*)]
\item There exists a constant $C>0$ such that for any $t \in (0,1]$,
\begin{align}\label{eq:lower-bound-Q}
C^{-1} \sqrt{t} \leq Q(t) \leq C \sqrt{t} .
\end{align}

\item  Let $\alpha\in [0,2)$. There exists a constant $C>0$ such that for any $n \in \mathbb{N}^*$ and $t\in (0,1]$,
\begin{align*}
|Q^n(t) - Q(t)| \leq C \, n^{-\frac{\alpha}{2}}\, t^{\frac{1}{2} - \frac{\alpha}{4}}.
\end{align*}

\item There exists a constant $C>0$ such that for any $n \in \mathbb{N}^*$,
\begin{align*}
\forall t\in (0,1],~ Q^n(t) \leq C \sqrt{t}  \quad \text{ and } \quad \forall t\in [h,1],~  Q^n(t) \geq C^{-1} \sqrt{t}.
\end{align*}

\end{enumerate}

\end{lemma}

\begin{proof}
\begin{enumerate}[label=(\roman*),itemsep=0pt]
\item The inequality \eqref{eq:lower-bound-Q} is a direct consequence of the definition of $Q$, see \cite[Eq. (2.51)]{butkovsky2021optimal}.

\item This is \cite[Lemma 2.3.3]{butkovsky2021optimal}.

\item If $t \ge h$, then by the previous point applied with $\alpha=0$, one gets that $Q^n(t) \leq Q(t) + C t^{1/2}$. Then use the first point. 
If $t < h$, then in view of (2.26) in \cite{butkovsky2021optimal}, there is $Q^n(t) = 2 n t \leq t^{1/2} (2 n h^{1/2})$. Using that $h=c(2n)^{-2}$, then $Q^n(t) \leq \sqrt{c} t^{1/2}$. 

As for the lower bound, this is \cite[Lemma 2.3.4]{butkovsky2021optimal}.
\end{enumerate}
\end{proof}

We refer to Lemma A.2 in \cite{athreya2020well} for a proof of the following Lemma.

\begin{lemma}\label{lem:besov-spaces}
Let $f$ a tempered distribution on $\R$ and let $\beta\in \R$, $p\in [1,\infty]$. For any $a_1,a_2,a_3 \in \mathbb{R}$ and $\alpha, \alpha_1, \alpha_2 \in [0,1]$, there is
\begin{itemize}
\item[(i)] $\| f(a + \cdot ) \|_{\mathcal{B}_p^\beta} \leq \| f \|_{\mathcal{B}_p^\beta}$ .
\item[(ii)] $\| f(a_1 + \cdot) - f(a_2 + \cdot) \|_{\mathcal{B}_p^\beta} \leq C |a_1 - a_2 |^{\alpha} \| f \|_{\mathcal{B}_p^{\beta+\alpha}}$ .
\item[(iii)] $\| f(a_1 + \cdot) - f(a_2 + \cdot) -f(a_3 + \cdot) + f(a_3+a_2-a_1+\cdot) \|_{\mathcal{B}_p^\beta} \leq C |a_1 - a_2 |^{\alpha_1} |a_1 - a_3 |^{\alpha_2} \| f \|_{\mathcal{B}_p^{\beta+\alpha_1 + \alpha_2}} .$
\end{itemize} 
\end{lemma}

\smallskip

Let us now recall classical heat estimates in Besov spaces,  from \cite{bahouri2011fourier,athreya2020well} (see also the proof of \cite[Lemma 3.1]{GHR2023}). 
\begin{lemma}\label{eq:reg-S}
Let $\beta \in \R$, $p\in [1,\infty]$ and $f \in \mathcal{B}_p^\beta$.
Then 
\begin{enumerate}[label=(\roman*)]

\item If $\beta-\frac{1}{p}<0$, $\| G_t f \|_{\infty} \leq C\, \|f \|_{\mathcal{B}_p^\beta}\, t^{\frac{1}{2}(\beta - \frac{1}{p})}$, for all $t > 0$.
 
 \item $\|G_t f - f\|_{\mathcal{B}_p^{\beta-\varepsilon}} \leq C\, t^{\frac{\varepsilon}{2}}\, \| f \|_{\mathcal{B}_p^\beta}$ for all $\varepsilon\in (0,1]$ and $t>0$. 
 In particular,  $\lim_{t \rightarrow
 0} \|G_t f -f\|_{\mathcal{B}_p^{\tilde{\beta}}}=0$ for every $\tilde{\beta}< \beta$.

 \item If $\beta-\frac{1}{p}<0$, $\|G_t f \|_{\mathcal{C}^1} \leq C\, \| f \|_{\mathcal{B}_p^\beta}  \, t^{\frac{1}{2}(\beta- \frac{1}{p}-1)}$ for all $t>0$.

\end{enumerate}
\end{lemma}

In the next lemma, we give estimates on the time regularity of the conditional expectation of random functions of the space-time Ornstein-Uhlenbeck process. See \cite[Lemma C.4]{athreya2020well} for a proof.

\begin{lemma}\label{lem:reg-O}
Let $\beta \in \R$, $p\in [1,\infty]$, $x \in \mathbb{T}$ and $(s,t)\in \Delta_{[0,1]}$. Assume that $\beta-1/p <0$ and recall $Q$ the variance of the OU process defined in \eqref{def:O-Q}. Let $f: \mathbb{R} \times \Omega \rightarrow \mathbb{R}$ be a bounded $\mathscr{B}(\mathbb{R}) \otimes \mathcal{F}_{s}$-measurable function. Then
\begin{equation}\label{eq:condexpf(OU)}
\mathbb{E}^s f\left(O_t(x)\right)=G_{Q(t-s)} f\left(P_{t-s} O_s(x)\right) .
\end{equation}
In addition, we have
$$ \left|\EE^s f\left(O_t(x)\right)\right| \leq C\|f\|_{\mathcal{B}_p^\beta}(t-s)^{\frac{1}{4}(\beta-\frac{1}{p})}  .$$
\end{lemma}

The following lemma states the same result as Lemma \ref{lem:reg-O} but for the discrete OU process $O^n$. The proof is an adaptation of the proof of \cite[Lemma C.4]{athreya2020well} and uses Lemma \ref{lem:bound-Qn}.

\begin{lemma}\label{lem:reg-On}
Let $\beta \in \R$, $p\in [1,\infty]$, $x \in \mathbb{T}$ and $0 \leq s \leq u \leq t \leq 1$. Assume that $\beta-1/p <0$ and recall $Q^n$ the variance of the discrete OU process defined in \eqref{def:discrete-OU}, and $\widehat{O}^n_{s,t}$ defined in \eqref{def:Onst} for $(s,t) \in \Delta_{[0,1]}$.
Let $f: \mathbb{R} \times \Omega \rightarrow \mathbb{R}$ be a bounded $\mathscr{B}(\mathbb{R}) \otimes \mathcal{F}_{s}$-measurable function. Then
$$
\EE^s f\left(O^n_t(x)\right)=G_{Q^n(t-s)} f\big(\widehat{O}^n_{s,t}(x)\big) .
$$
In addition, when $t-s \geq h$ (recall that $h=c(2n)^{-2}$), there is
$$ \left|\EE^s f\left(O^n_t(x)\right)\right| \leq C\|f\|_{\mathcal{B}_p^\beta}(t-s)^{\frac{1}{4}(\beta-\frac{1}{p})}  .$$
\end{lemma}
\begin{proof}
Since $O^n_t(x)=O^n_t(x)-\widehat{O}^n_{s,t}(x)+ \widehat{O}^n_{s,t}(x)$ with $\widehat{O}^n_{s,t}(x)$ $\mathcal{F}_s$-measurable and $O^n_t(x)-\widehat{O}^n_{s,t}(x)$ is independent of $\mathcal{F}_s$, we get
\begin{align*}
\EE\big[ ( O^n_t(x)-\widehat{O}^n_{s,t}(x))^2\big] = \int_s^t \int_\T p^n_{(t-r)_h}(x,y)^2 \, dy dr = Q^n(t-s) .
\end{align*}
Hence
$\EE^s f\left(O^n_t(x)\right) = G_{Q^n(t-s)}f(\widehat{O}^n_{s,t}(x))$. 
Using Lemma \ref{eq:reg-S}$(i)$, we get
\begin{align*}
|\EE^s f\left(O^n_t(x)\right)| & \leq C \| f \|_{\mathcal{B}_p^\beta}\, Q^n(t-s)^{\frac{1}{2}(\beta-\frac{1}{p})} .
 \end{align*}
 We conclude using that $\beta-\frac{1}{p}<0$ and $Q^n(t-s) \geq C\sqrt{t-s}$ when $t-s\geq h$ (by Lemma~\ref{lem:bound-Qn}$(iii)$).
\end{proof}

\begin{lemma}\label{lem:diffhatOn}
Let $m \ge 2$. There exists a constant $C>0$ such that for any $n \in \N^*$ and $(s,t) \in \Delta_{[0,1]}$ with $t-s \ge h$, we have
\begin{align*}
\| \widehat{O}^n_{s,t}(y) -\widehat{O}_{s,t}^n(y_n) \|_{L^m} \leq C n^{-\frac{1}{2}} .
\end{align*}
\end{lemma}

\begin{proof}
By the triangle inequality, 
\begin{align*}
\| \widehat{O}^n_{s,t}(y) -\widehat{O}_{s,t}^n(y_n) \|_{L^m} & \leq \| \widehat{O}^n_{s,t}(y) - P_{t-s}O_{s}(y) \|_{L^m}+ \| P_{r-s}O_s(y_n)-\widehat{O}^n_{s,t}(y_n) \|_{L^m}  \\ & \quad + \| P_{t-s} O_{s}(y)-P_{t-s}O_s(y_n) \|_{L^m} .
\end{align*}
Then use \cite[Eq (2.53)]{butkovsky2021optimal} for the first two terms (with $\beta=1$) and use \cite[Proposition 2.3.1]{butkovsky2021optimal} for the last term.
\end{proof}

\subsection{Existence of functions with prescribed decay of their Littlewood-Paley blocks}\label{subsec:LPblocks}

The $\mathcal{B}^\gamma_{\infty}$ norm is given by
\begin{align*}
\| b \|_{\mathcal{B}_{\infty}^{\gamma}}= \sup_{j \ge -1} 2^{j \gamma} \| \Delta_j  b  \|_{\infty},
\end{align*}
where for $j\geq 0$, $\Delta_j$ is defined by
\begin{align*}
\Delta_j b (x) = \int_{\R} 2^j h(2^j(x-y)) b(y) \, dy ,
\end{align*} 
where $h$ is the inverse Fourier transform of a radial smooth function supported on an annulus, see e.g. \cite[Proposition 2.10]{bahouri2011fourier} (for $j=-1$, one considers a function $\tilde{h}$ instead of $h$, which is the inverse Fourier transform of a smooth function supported on a centred ball).

\begin{lemma}\label{lem:example}
Let $\gamma \in (-1,0)$. In any neighbourhood of $\gamma$, there exists $\eta$ for which there exists $b\in \mathcal{B}_{\infty}^{\eta}(\R^d)$ such that for some $J \in \N$ and $C>0$, 
\begin{align}\label{eq:delta_j}
\forall j\geq J,\quad \| \Delta_j b \|_{\infty} \ge C\, 2^{-j\eta} .
\end{align}
\end{lemma}

\begin{proof}
We provide an example in $\R$, but a similar one works in $\R^d$. Let $b(x)=x^{\gamma} \mathds{1}_{(0,1)}(x)$. Then 
\begin{align*}
\Delta_j b(x) = \int_{0}^1 2^j h(2^j(x-y) y^{\gamma} \, dy = 2^{-j \gamma} \int_0^{2^j} h(2^j x -y) y^{\gamma} \, dy .
\end{align*}
It follows that
\begin{align*}
2^{j \gamma} \| \Delta_j b \|_{\infty} = \sup_{x \in \R} \Big| \int_0^{2^j} h(x -y) y^{\gamma} \, dy \Big| \ge \Big| \int_0^{2^j} h(-y) y^{\gamma} \, dy\Big| .
\end{align*}
By a dominated convergence argument, using that $\gamma >-1$ and that $h$ is a Schwartz function, we have that $\int_0^{2^j} h(-y) y^{\gamma} dy$ converges to $\int_{\R_{+}} h(-y) y^{\gamma} dy \in \R$ as $j$ goes to $\infty$. From this and the equality above, one can deduce that  $\sup_{j \ge -1} 2^{j \gamma} \| \Delta_j b \|_{\infty} < \infty$ and therefore that $b \in \mathcal{B}_\infty^\gamma$. Moreover, this implies that there exists $J \ge 0$, such that for any $j \ge J$, we have
\begin{align*}
2^{j \gamma} \| \Delta_j b \|_{\infty} \ge \frac{1}{2} \Big| \int_{\R_{+}} h(-y) y^{\gamma} \, dy\Big|.
\end{align*}
Now we claim that the integral in the right-hand side can only vanish for isolated values of $\gamma\in (0,1]$. Indeed, the function $H:\eta \mapsto \int_{\R_{+}} h(-y) y^{\eta} dy$ is  analytic  for $\eta \in (-1,0]$. Since for $\eta = 0$, we have $\int_{\R_{+}} h(-y)  dy = \| \mathds{1}_{\R_{+}} \|_{\mathcal{B}_\infty^0}>0$, we know that $H$ cannot vanish on any open interval. Hence, in any neighbourhood of $\gamma$, there exists $\eta$ such that $H(\eta)>0$. So \eqref{eq:delta_j} is satisfied.
\end{proof}

\subsection{The Stochastic Sewing Lemma and applications}

\label{appendix:ssl}We recall the Stochastic Sewing Lemma as it is stated in \cite{athreya2020well}. 

\begin{lemma}[Theorem 4.1 in \cite{athreya2020well} with $\alpha_1=\alpha_2=0$] \label{lem:SSL}
Let $0\leq S<T\leq 1$, $m \in [2, \infty)$ and $q \in [m,\infty]$. Let $(\Omega,\mathcal{F},\mathbb{F},\mathbb{P})$ be a filtered probability space, let $A: \Delta_{[S,T]} \rightarrow L^m$ a process such that $A_{s,t}$ is $\mathcal{F}_t$-measurable for any $(s,t) \in \Delta_{[S,T]}$. We assume that there are constants $\Gamma_1,\Gamma_2\geq 0$, and $\varepsilon_1,\varepsilon_2>0$ such that for any $(s,t) \in \Delta_{[S,T]}$ and $u=(s+t)/2$,
\begin{align}
    \|\EE^s[\delta A_{s,u,t}]\|_{L^q}&\leq \Gamma_1 \,  (t-s)^{1+\varepsilon_1},\label{eq:condsew1}\\
   \Big\| \Big( \EE^S | \delta A_{s,u,t} |^m \Big)^{\frac{1}{m}}  \Big\|_{L^q} &\leq \Gamma_2 \,  (t-s)^{\frac{1}{2}+\varepsilon_2}. \label{eq:condsew2}
\end{align}
Then there is a process $(\mathcal{A}_t)_{t\in [S,T]}$ such that, for any $t \in [S,T]$ and any sequence of partitions $\Pi_k=\{t_i^k\}_{i=1}^{N_k}$ of $[S,t]$ whose mesh size converges to zero, one has
\begin{align} \label{eq:convAt}
    \mathcal{A}_t=\lim_{k\rightarrow \infty}\sum_{i=1}^{N_k-1} A_{t_i^k,t_{i+1}^k} \text{ in probability.} 
\end{align}

In addition, there is a constant $C$ which depends only on $\varepsilon_1$, $\varepsilon_2$, $m$ and independent of $S,T$ such that for every $(s,t) \in \Delta_{[S,T]}$, there is
\begin{align*}
    \Big\| \Big( \EE^S | \mathcal{A}_t-\mathcal{A}_s-A_{s,t} |^m \Big)^{\frac{1}{m}} \Big\|_{L^q} \leq C\, \Gamma_1  (t-s)^{ 1+\varepsilon_1} + C\, \Gamma_2  (t-s)^{\frac{1}{2} \varepsilon_2},
\end{align*}
and
\begin{align*}
    \|\EE^S[\mathcal{A}_t-\mathcal{A}_s-A_{s,t}]\|_{L^{q}}\leq C\, \Gamma_1 (t-s)^{ 1+\varepsilon_1}.
\end{align*}
\end{lemma}

\begin{remark}\label{rmk:critical-sewing}
In Propositions~\ref{prop:drift-approx} and~\ref{prop:bound-E1-critic}, we use an adaptation of the Stochastic Sewing Lemma with critical-exponent, which comes from \cite[Theorem 4.5]{athreya2020well}. With the same notations and assumptions as the previous lemma (choose $q=m$), and with the additional hypothesis that there exist $\Gamma_3, \varepsilon_4 >0$, $\Gamma_4\geq 0$ such that
\begin{align}\label{eq:sts4}
\left\|\EE^s\left[\delta A_{s, u, t}\right]\right\|_{L^m} \leq \Gamma_3|t-s|+\Gamma_4|t-s|^{1+\varepsilon_4}.
\end{align}
Then for any $(s,t) \in \Delta_{[S,T]}$,
\begin{align}\label{eq:sts-critic}
\left\|\mathcal{A}_t-\mathcal{A}_s-A_{s, t}\right\|_{L^m} \leq C \Gamma_3\left(1+\left|\log \frac{\Gamma_1 T^{\varepsilon_1}}{\Gamma_3}\right|\right)(t-s)+C \Gamma_2(t-s)^{\frac{1}{2}+\varepsilon_2}+C \Gamma_4(t-s)^{1+\varepsilon_4}.
\end{align}
\end{remark}

We also recall a regularisation property of the Ornstein-Uhlenbeck process which was originally stated as Lemma~6.1 in \cite{athreya2020well}. This result is used several times in Section~\ref{sec:reg-O}.

\begin{lemma}\label{lem:6.1athreya}
Let $0 \leq  S \leq T$. Let $\beta \in (-2,0), m \in [2, \infty), n \in [m, \infty]$ and $p \in [n, \infty]$. There exists a constant $C$ such that for any bounded measurable function $h:\mathbb{R} \times[S, T] \times \T \times \Omega \rightarrow \mathbb{R}$ satisfying:
\begin{enumerate}
\item  for any fixed $(z, r, x) \in \mathbb{R} \times[S, T] \times \T$, the random variable $h(z, r, x)$ is $\mathcal{F}_S$-measurable;
\item there exists a constant $\Gamma_h>0$ such that
\begin{align*}
\sup _{(r, x) \in [S, T] \times \T}\ |\| h(\cdot, r, x)\|_{\mathcal{B}_p^\gamma}\|_{L^n} \leq \Gamma_h;
\end{align*}
\end{enumerate}
then for any $t \in[S, T]$, the following inequality holds true:
\begin{align*}
\sup_{x \in \T} \left\| \left( \EE^S \left| \int_S^t \int_{\T} p_{T-r}(x,y) h\left(O_r(y), r, y\right) d y d r \right|^m \right)^{\frac{1}{m}}  \right\|_{L^n} \leq C \Gamma_h (t-S)^{1+\frac{1}{4}(\beta-\frac{1}{p})} .
\end{align*}
\end{lemma}

\begin{proof}
Let $x \in \T$. Apply \cite[Lemma 6.1]{athreya2020well} with $X_r(y) = p_{T-r}(x,y)$. Since $\int_\T p_{T-r}(x,y) dy =1$, all the assumptions of the Lemma are satisfied. To conclude, take the supremum over $x \in \T$ in the result.
\end{proof}

\section{A critical Gr\"onwall-type lemma}\label{app:B}

In Lemma 3.1 of \cite{GHR2023}, a Gr\"onwall-type lemma with logarithmic factor and perturbation $\eta$ was established for vector-valued functions depending on a single parameter. Here an extension to two-parameter functions is needed.

\begin{lemma}\label{lem:rate-critical}
For $(E, \|\cdot\|)$ a normed vector space, $\ell >0, \ C_{1}, C_{2} \ge 0$ and $\eta\in(0,1)$, we define $\mathcal{R}_{2}(\eta,\ell, C_{1}, C_{2})$ the set of functions from $\Delta_{[0,1]}$ to $E$ which fulfills the following conditions:
\begin{enumerate}[label=(\roman*)]
\item $f$ is bounded, $f_{0,0}=0$ and for any $(s,t)\in \Delta_{[0,1]}$ such that $t-s \leq \ell$,
\begin{equation}\label{eq:boundIncf}
\begin{split}
\|f_{s,t}\| &\leq C_1 \, (\|f_{0,\cdot}\|_{L^\infty_{[s,t]}E} + \eta) \, (t-s)^{\frac{1}{2}}  \\
&\quad + C_2 \, ( \|f_{0,\cdot}\|_{L^\infty_{[s,t]}E}+\eta ) \, \big| \log \big( \|f_{0,\cdot}\|_{L^\infty_{[s,t]}E} + \eta \big)\big| \, (t-s). 
\end{split}
\end{equation}

\item for any $(s,t)\in \Delta_{[0,1]}$, $ \|f_{0,t}\| - \|f_{0,s}\| \leq \|f_{s,t}\|$.%
\end{enumerate}
Then for any $\delta\in(0, e^{-C_{2}})$, there exists $\bar{\eta} \equiv \bar{\eta}(C_{1},C_{2},\ell,\delta)$ such that for any $\eta<\bar{\eta}$ and any $f\in \mathcal{R}_{2}(\eta,\ell, C_{1}, C_{2})$,
\begin{equation*}
\| f_{0,\cdot} \|_{L^\infty_{[0,1]} E} \leq \eta^{e^{-C_{2}}-\delta} .
\end{equation*}
\end{lemma}

\begin{remark}
\begin{itemize}
\item For a one-parameter function $f$, denote $f_{s,t}$ the increment $f_{t}-f_{s}$. Then the condition $(ii)$ above is automatically satisfied. Hence one recovers Lemma~3.1 from \cite{GHR2023}.

\item In this paper, the two-parameter function satisfies the relation $f_{s,t} = f_{0,t} - P_{t-s} f_{0,s}$. The contraction property of the Gaussian semigroup ensures that condition $(ii)$ is fulfilled. More generally, the above lemma could be useful for differential equations written in mild form.
\end{itemize}
\end{remark}

\begin{proof}
The proof follows the same lines as the proof from \cite{GHR2023}. Hence assume without any loss of generality that $C_1, C_2 >0$. Let $\delta \in (0,\frac{1}{2} e^{-C_2})$ and $a > 1$ such that $e^{-C_2 \frac{a\log a}{a-1}} =  e^{-C_2}-\delta$, let $\varepsilon \equiv \varepsilon(C_2,\delta) \in (0,1)$ be such that $e^{-C_2 \frac{a\log a}{(a-1)(1-\varepsilon)}} \geq e^{-C_2}-2 \delta$. Define $\alpha = 1- e^{-C_2 \frac{a\log a}{(a-1)(1-\varepsilon)}}$. For $\eta \in (0,1)$ and $f\in \mathcal{R}_{2}(\eta,\ell,C_{1},C_{2})$, define again an increasing sequence of times as follows: $t_0=0$ and for $k \in \N$, 
\begin{equation*}
t_{k+1} = 
\begin{cases}
\inf \{t>t_{k}\colon ~  \eta+ \|f_{0,t}\| \ge a^{k+1}\, \eta \} \wedge 1 & \text{ if } t_{k}<1,\\
1 & \text{ if } t_{k}=1 .
\end{cases}
\end{equation*}
By \eqref{eq:boundIncf} and the boundedness of $f$, the mapping $(s,t)\mapsto \|f_{s,t}\|$ is continuous. By definition of $(t_k)_{k\in \N}$, it follows that for any $k$,
$\|f_{0,\cdot}\|_{L^\infty_{[0,t_{k}]}E} = a^k\eta - \eta \leq a^k \eta$.
Let 
\begin{equation*}%
N = \left\lfloor - \alpha \frac{\log(\eta)}{\log(a)} \right\rfloor -1 ,
\end{equation*}
and let $\bar{\eta}_{0} \equiv \bar{\eta}_{0}(C_{2},\delta)$ be such that $N\geq 1$ when $\eta<\bar{\eta}_{0}$.
As in \cite{GHR2023}, it suffices to show the following statement to prove the lemma:
\begin{align}\label{eq:statementEpsBar}
\mbox{There exists $\bar{\eta} \equiv \bar{\eta}(C_1,C_{2},\ell, \delta)$ such that } \forall \eta<\bar{\eta} \mbox{ and } \forall f\in \mathcal{R}_{2}(\eta,\ell,C_{1},C_{2}), ~ t_{N+1}=1.
\end{align}
To prove the statement \eqref{eq:statementEpsBar}, fix $\eta < \bar{\eta}_{0}$ and $f\in \mathcal{R}_{2}(\eta,\ell,C_{1},C_{2})$. 
Let $N_{0} = \inf\left\{ k\in \N\colon~ t_{k+1} = 1 \right\}$.
Proving that $N_{0}\leq N$ will imply that $t_{N+1} = 1$. First, if $N_{0} = 0$, we have directly that $N\geq N_{0}$. Thus we now we assume that $N_{0}\geq 1$.
 For any $k \leq N_{0}-1$, one has $\eta+\|f_{0,t_{k}}\| = a^k\, \eta$ and $\eta+\|f_{0,t_{k+1}}\| = a^{k+1}\, \eta$. Now use that $\|f_{0,t_{k}}\|\leq \|f_{0,t_{k+1}}\|$ and condition $(ii)$ to get that  $ \|f_{t_{k},t_{k+1}} \|\geq (a^{k+1}-a^k) \eta$.
 Then this inequality with \eqref{eq:boundIncf} suffice, as in \cite{GHR2023}, to prove that for some $\bar{\eta}_1>0$, for any $\eta < \bar{\eta}_1$, there is
\begin{align*}
1 &\leq \frac{C_1}{\sqrt{\ell}} \frac{a}{a-1}\,  (t_{k+1}-t_{k})^{\frac{1}{2}} + C_2 \frac{a}{a-1} \,  |\log (a^{k+1} \eta) | \, (t_{k+1}-t_{k}) .
\end{align*}
The above inequality implies that $t_{N+1} = 1$, exactly as in \cite{GHR2023}.
\end{proof}

\end{appendices}

\bibliographystyle{abbrvnat}

\end{document}